\documentclass[12pt, letterpaper]{amsart}

\newif\ifPDF
\ifx\pdfoutput\undefined\PDFfalse
\else \ifnum \pdfoutput > 0 \PDFtrue
        \else \PDFfalse
        \fi
\fi

\usepackage[centertags]{amsmath}
\usepackage{amsfonts}
\usepackage{mathrsfs}
\usepackage{textcomp}
\usepackage{amssymb}
\usepackage{amsthm}
\usepackage{newlfont}
\usepackage[all]{xy}

\ifPDF
  \usepackage[pdftex]{color, graphicx}
  \usepackage[pdftex, bookmarks, colorlinks]{hyperref}
  \hypersetup{colorlinks=false}

\else
  \usepackage{color}
  \usepackage[dvips]{graphicx}
  \usepackage[dvips]{hyperref}
\fi

\usepackage[scale=0.8]{geometry}

\usepackage[pagewise, mathlines, displaymath]{lineno}

\newtheorem{thm}{Theorem}[section]
\newtheorem{cor}[thm]{Corollary}
\newtheorem{lem}[thm]{Lemma}
\newtheorem{prop}[thm]{Proposition}
\theoremstyle{definition}
\newtheorem{defn}[thm]{Definition}
\theoremstyle{remark}
\newtheorem{rem}[thm]{Remark}
\newtheorem{example}[thm]{Example}
\numberwithin{equation}{section}

\newcommand{\norm}[1]{\left\Vert#1\right\Vert}
\newcommand{\abs}[1]{\left\vert#1\right\vert}
\newcommand{\set}[1]{\left\{#1\right\}}

\newcommand{\Int}{\mathbb Z}
\newcommand{\Comp}{\mathbb C}

\newcommand{\eps}{\varepsilon}

\newcommand{\F}{\mathcal{F}}

\newcommand{\Kone}{\mathrm{K}_1}

\newcommand{\M}[2]{\mathrm{M}_{#1}(#2)}

\renewcommand{\varnothing}{\O}

\begin{document}


\title{Stable rank of transformation group C*-algebras}

\author{Chun Guang Li}
\address{College of Mathematics and Statistics, Northeast Normal University, Changchun, Jilin, China, 130024}
\email{licg864@nenu.edu.cn}

\author{Zhuang Niu}
\address{Department of Mathematics and Statistics, University of Wyoming, Laramie, Wyoming, USA, 82071}
\email{zniu@uwyo.edu}

\date{\today}
\dedicatory{Dedicated to Professor George A.~Elliott on the occasion of his 75th birthday}


\begin{abstract}
It is shown that, for a free and minimal $\Int^d$-action on a compact Hausdorff space $X$, the transformation group C*-algebra $\mathrm{C}(X)\rtimes\Int^d$ always has stable rank one, i.e., the invertible elements are dense. 
Moreover, the C*-algebra $\mathrm{C}(X)\rtimes\Int^d$ is also shown to be classified by the Elliott invariant if, and only if, the strict order of its Cuntz semigroup is determined by the traces. That is, it satisfies the Toms-Winter conjecture.

In fact, for any free and minimal $\Gamma$-action on $X$, where $\Gamma$ is a countable discrete amenable group, if $(X, \Gamma)$ has the uniform Rokhlin property and Cuntz comparison of open sets, then the C*-algebra $\mathrm{C}(X)\rtimes\Gamma$ is shown to have stable rank one and to satisfy the Toms-Winter conjecture.
\end{abstract}

\maketitle

\setcounter{tocdepth}{1}
\tableofcontents

\section{Introduction}

The topological stable rank of a unital C*-algebra $A$, denoted by $\mathrm{tsr}(A)$, was introduced by Rieffel in his seminal paper \cite{Rieffel-DimStr} as a topological version of the Bass stable rank of a ring. Consider the set of $n$-tuples generating $A$ as a left ideal, $$Lg_n:=\{(x_1, x_2, ..., x_n)\in A^n: Ax_1 + Ax_2+ \cdots+ Ax_n = A\}.$$ Then the topological stable rank of $A$, denoted by $\mathrm{tsr}(A)$, is the smallest $n$ such that $Lg_n$ is dense in $A^n$ (if no such $n$ exists, then the topological stable rank of $A$ is $\infty$). It was shown in \cite{HV-sr} that the topological stable rank of  a C*-algebra agrees with its Bass stable rank. Thus, we may just refer it as stable rank.

The stable rank models dimension of a topological space. It was shown in \cite{Rieffel-DimStr} that the stable rank of the commutative C*-algebra $\mathrm{C}(X)$, where $X$ is a compact Hausdorff space, is $\lfloor \frac{\mathrm{dim}(X)}{2}\rfloor + 1$, where $\lfloor \cdot \rfloor$ denotes the integer part. It was shown in \cite{Vill-sr} that for any $n\in\{1, 2, ..., \infty\}$, there exists a simple unital separable C*-algebra $A$ ($A$ can be chosen to be the limit of an inductive sequence of homogeneous C*-algebras) such that $\mathrm{tsr}(A) = n$. 

The class of C*-algebras with stable rank one is particularly interesting. Any such C*-algebra $A$ is stably finite, has cancellation of projections, and has the property that $\Kone(A)$ is canonically isomorphic to $\mathrm{U}(A)/\mathrm{U}_0(A)$. It is also well known that a C*-algebra has stable rank one if, and only if, $A = \overline{\mathrm{GL}(A)}$, where $\mathrm{GL}(A)$ denotes the group of invertible elements of $A$.

Many classes of simple C*-algebras have been shown to have stable rank one. For instance, any simple unital finite C*-algebra which absorbs a UHF algebra or, more generally, absorbs the Jiang-Su algebra $\mathcal Z$, has stable rank one (see \cite{RorUHF} and \cite{Ror-Z-stable}, respectively). Such C*-algebras, when nuclear, certainly are well behaved from the perspective of the classification program. For non-nuclear C*-algebras, it was shown in \cite{DHR-sr1} that the reduced group C*-algebra $\mathrm{C}^*_\mathrm{r}(G_1 * G_2)$ has stable rank one if $\abs{G_1}\geq 2$ and $\abs{G_2} \geq 3$.

On the other hand, even beyond the classifiable C*-algebras,  remarkably, it was shown by Elliott, Ho, and Toms (\cite{EHT-sr1}, also see \cite{Ho-thesis}) that any simple unital AH algebra with diagonal maps (this class of C*-algebras contains the exotic AH algebras of \cite{Vill-perf} and \cite{Toms-Ann}, which cannot be classified by the ordered K-groups together with the traces---the Elliott invariant), whether classifiable or not, always has stable rank one. 

This result is in fact the main motivation of the current paper: Consider the C*-algebra of a minimal homeomorphism. In general, its behaviour is expected to be parallel to the behaviour of an AH algebra with diagonal maps (see, for instance, \cite{Niu-MD} and \cite{EN-MD0} on classifiability and mean dimension), and it had been anticipated for some time in the C*-algebra community that the C*-algebra of a minimal homeomorphism should always have stable rank one. (Note that, as shown in \cite{GK-Dyn}, there exists a minimal homeomorphism of an infinite compact Hausdorff space such that the corresponding C*-algebra is not classifiable in terms of Elliott invariant.)

A considerable amount of work has been done concerning this question, and finally it was solved recently by Alboiu and Lutley in \cite{Lutley-Alboiu}. For the C*-algebra of a minimal homeomorphism, one can consider the orbit-breaking subalgebra, which was introduced by Putnam for Cantor dynamical systems (\cite{Put-PJM}) and then constructed for a general minimal homeomorphism by Q.~Lin (\cite{Lin-Q-RSA}). It was shown by Archey and Phillips  (\cite{A-NCP-LAlg}) that if the orbit-breaking subalgebra has stable rank one, then the transformation group C*-algebra must have stable rank one. If the minimal dynamical system has a Cantor factor (so that the orbit-breaking subalgebra is an AH algebra with diagonal maps),  with the result of \cite{EHT-sr1}, one has that the transformation group C*-algebra has stable rank one (see \cite{A-NCP-LAlg}) (this result was generalized by Suzuki in \cite{Suzuki-sr1} to the C*-algebra of a minimal almost finite groupoid). For a general minimal homeomorphism (i.e., $\Int$-action), the orbit-breaking subalgebra might not be AH, but it is still a unital inductive limit of subhomogeneous C*-algebras with diagonal maps (the DSH algebras of \cite{Lutley-Alboiu}). Alboiu and Lutley show in \cite{Lutley-Alboiu} that any unital simple DSH algebra has stable rank one, and thus the C*-algebra of a minimal homeomorphism has stable rank one.


Beyond the case of $\Int$-actions, however, it is not clear how to construct large subalgebras in general. To overcome this difficulty,  Uniform Rokhlin Property (URP) and Cuntz-comparison of Open Sets (COS) (see Definitions \ref{Defn-URP} and \ref{Defn-COS}) for a topological dynamical system $(X, \Gamma)$, where $\Gamma$ is a countable discrete amenable group, were introduced in \cite{Niu-MD-Z}. In \cite{Niu-MD-Z}, it was shown that, with the (URP) and (COS), the radius of comparison of the crossed product C*-algebra $\mathrm{C}(X)\rtimes\Gamma$ is dominated by half of the mean dimension of $(X, \Gamma)$ (\cite{Niu-MD-Z}), and the C*-algebra $\mathrm{C}(X)\rtimes\Gamma$ is classified by its Elliott invariant if $(X, \Gamma)$ has mean dimension zero (\cite{Niu-MD-Z-absorbing}). Moreover, it was also shown in \cite{Niu-MD-Zd} that any free and minimal $\Int^d$-action has the (URP) and (COS).

In this paper, we again consider these two properties, and we show that if a free and minimal $\Gamma$-action has the (URP) and (COS), then the transformation group C*-algebra must have stable rank one (Theorem \ref{main-thm}). Since any free and minimal $\Int^d$-action has the (URP) and (COS), the C*-algebra $\mathrm{C}(X) \rtimes \Int^d$, classifiable or not, always has stable rank one:
\theoremstyle{theorem}
\newtheorem*{thmN}{Theorem}
\begin{thmN}[Corollary \ref{sr-Z}]\label{thm-Z}
Let $\Int^d$ act freely and minimally on a compact Hausdorff space $X$. Then $\mathrm{tsr}(\mathrm{C}(X) \rtimes \Int^d) = 1.$
\end{thmN}

As consequences of stable rank one (some of them are well known), we obtain the following properties of the crossed product C*-algebra $A=\mathrm{C}(X)\rtimes\Int^d$, where $(X, \Int^d)$ is free and minimal (in general, those properties also hold for any C*-algebra $A=\mathrm{C}(X)\rtimes\Gamma$, where $(X, \Gamma)$ is free, minimal, and has the (URP) and (COS)):
\begin{itemize}

\item $A$ has cancellation of projections, cancellation in Cuntz semigroup, and $\mathrm{U}(A)/\mathrm{U}_0(A) \cong \Kone(A)$ (Corollary \ref{cor-cancellation}).

\item The approximately unitary equivalence class of a homomorphism from an AI algebra to $A$ is determined by the induced map between the Cuntz semigroups (Corollary \ref{cor-AI}).

\item Any strictly positive lower semicontinous affine function on $\mathrm{T}(A)$ can be realized as the rank function of some positive element of $A \otimes\mathcal K$ (Corollary \ref{Cu-surj}). 

\item $A$ absorbs the Jiang-Su algebra tensorially if, and only if, $A$ has strict comparison of positive elements (Corollary \ref{cor-Z}). That is, $A$ satisfies the Toms-Winter conjecture. 
(Note that any metrizable Choquet simplex can arise as the trace simplex of a transformation group C*-algebra $\mathrm{C}(X)\rtimes\Int$.)

\item The real rank of $A$ is either $0$ or $1$ (Corollary \ref{cor-RR}).
\end{itemize}

To prove the theorem above, we introduce Property (D) (Definition \ref{defn-Prop-D}), a version of diagonalizability. Any finite C*-algebra with Property (D) is shown to have stable rank one (Theorem \ref{thm-ab-tsr1}). In Sections \ref{D-1} and \ref{D-2}, it is shown that any transformation group C*-algebra of a free and minimal $\Gamma$-action with (URP) and (COS) has Property (D) (Proposition \ref{main-prop}), and thus the main theorem follows. 

Property (D) can also be applied to other C*-algebras. In Section \ref{D-remark}, we observe that simple unital $\mathcal Z$-stable algebras and simple unital AH algebras with diagonal maps have Property (D). These C*-algebras (the first when finite) are known to have stable rank one (\cite{Ror-Z-stable} and \cite{EHT-sr1}), and hence Property (D) provides an alternative approach to the stable rank of these C*-algebras. 

\subsubsection*{Acknowledgements} The research of the second named author is supported by an NSF grant (DMS-1800882). The result in this paper was obtained during the visit of the first named author to the University of Wyoming in 2019-2020, which was supported by a CSC visiting scholar fellowship (No.~201906625028). The first named author thanks the Department of Mathematics and Statistics at  the University of Wyoming for its hospitality. The research of the first named author is also partly supported by an NNSF grant of China (No.~11401088).

\section{Notation and preliminaries}

\renewcommand{\varnothing}{\O}

\subsection{Topological Dynamical Systems}

\begin{defn}
Consider a topological dynamical system $(X, \Gamma)$, where $X$ is a separable compact Hausdorff space, and $\Gamma$ is a discrete group which acts on $X$ from the right. 
The dynamical system $(X, \Gamma)$ is said to be minimal if $$Y\gamma = Y,\quad \gamma \in\Gamma,$$ for some closed set $Y\subseteq X$ implies $Y=\varnothing$ or $Y=X$; it is said to be free if $$x\gamma = x$$ for some $x\in X$ and $\gamma\in \Gamma$ implies $\gamma=e$.

\end{defn}

\begin{defn}
A Borel measure $\mu$ on $X$ is invariant under the action $\sigma$ if for any Borel set $E\subseteq X$, one has $$\mu(E) = \mu(E\gamma),\quad \gamma\in\Gamma.$$ Denote by $\mathcal M_1(X, \Gamma)$ the set of all invariant Borel probability measures on $X$. It is a Choquet simplex under the weak* topology.
\end{defn}

\begin{defn}\label{defn-amenable}
Let $\Gamma$ be a countable discrete group. Let $K\subseteq\Gamma$ be a finite set and let $\delta>0$. Then a finite set $E \subseteq \Gamma$ is said to be $(K, \eps)$-invariant if $$\frac{\abs{EK\Delta E}}{\abs{E}} < \eps.$$

The group $\Gamma$ is amenable if there is a sequence  $(\Gamma_n)$ of finite subsets of $\Gamma$ such that for any $(K, \eps)$, there is $N$ such that $\Gamma_n$ is $(K, \eps)$-invariant for any $n>N$. The sequence $(\Gamma_n)$ is called a F{\o}lner sequence.

The $K$-interior of a finite set $E\subseteq\Gamma$ is defined as $$\mathrm{int}_K(E) = \{\gamma\in E: \gamma K \subseteq E\},$$ and the $K$-boundary of $E$ is defined as $$\partial_KE:= E \setminus \mathrm{int}_K(E) = \{\gamma\in E: \textrm{$\gamma\gamma' \notin E$ for some $\gamma'\in K$}\}.$$

Note that $$\abs{E\setminus\mathrm{int}_K(E)} \leq\abs{K}\abs{EK\setminus E}\leq \abs{K}\abs{EK\Delta E} ,$$ and hence for any $\eps>0$, if $E$ is $(K, {\eps}/{\abs{K}})$-invariant, then $$\frac{\abs{E\setminus\mathrm{int}_K(E)}}{\abs{E}} < \eps.$$

\end{defn}

\begin{rem}
If a set $E\subseteq\Gamma$ is $(\mathcal F, \eps)$-invariant, then, for any $\gamma\in\Gamma$, the left translation $\gamma E$ is again $(\mathcal F, \eps)$-invariant.
\end{rem}

\begin{defn}
An (exact) tiling of a discrete group consists of
\begin{itemize}
\item a finite collection $\mathcal S=\{\Gamma_1, ...,\Gamma_n\}$ of finite subsets of $\Gamma$ containing the unit $e$, called the shapes,
\item a finite collection $\mathcal C = \{C(S): S\in\mathcal S\}$ of disjoint subsets of $\Gamma$, called center sets,
\end{itemize}
such that the left translations
$$cS,\quad c\in C(S),\ S\in\mathcal S$$
form a partition of $\Gamma$.
\end{defn}

\begin{rem}
If $\Gamma$ is amenable, then it follows from \cite{DHZ-tiling} that for any finite set $\F\subseteq\Gamma$ and any $\eps>0$, there is a tiling of $\Gamma$ such that all its shapes are $(\mathcal F, \eps)$-invariant.
\end{rem}

\subsection{Crossed product C*-algebras}
Consider a topological dynamical system $(X, \Gamma)$. Then the group $\Gamma$ acts (from the left) on the C*-algebra $\mathrm{C}(X)$ by $$\gamma(f)= f\circ \gamma.$$
The (full) crossed product C*-algebra $A=\mathrm{C}(X)\rtimes\Gamma$ is defined to be the universal C*-algebra generated by $\mathrm{C}(X)$ and unitaries $u_\gamma$, $\gamma\in\Gamma$, with respect to the relations
$$ u_\gamma fu_\gamma^*=f\circ\gamma,\ u_{\gamma_1}u^*_{\gamma_2} = u_{\gamma_1\gamma_2^{-1}},\ u_e=1_A,\quad \gamma, \gamma_1, \gamma_2 \in \Gamma $$
The C*-algebra $A$ is nuclear if $\Gamma$ is amenable (see, for instance, Corollary 7.18 of \cite{Williams}). If, moreover, $(X, \Gamma)$ is minimal, then the C*-algebra $A$ is simple (Theorem 5.16 of \cite{Effros-Hahn} and  Th\'{e}or\`{e}me 5.15 of \cite{ZM-prod}), i.e., $A$ has no non-trivial closed two-sided ideals.

\subsection{Cuntz-sub-equivalence and rank functions}

\begin{defn}
Let $A$ be a C*-algebra, and let $a, b\in A^+$. The element $a$ is said to be Cuntz sub-equivalent to $b$, denoted by $a \precsim b$, if there are $x_i$, $y_i$, $i=1, 2, ...$, such that $$\lim_{i\to\infty} x_iby_i = a.$$
\end{defn}

\begin{example}
Let $f, g\in\mathrm{C}(X)$ be positive elements, and consider the open sets $$E := f^{-1}(0, +\infty)\quad\mathrm{and}\quad F:=g^{-1}(0, +\infty).$$ Then $f \precsim g$ if and only if $E \subseteq F$. In particular, the Cuntz equivalence class of a positive element of $\mathrm{C}(X)$ determines and is determined by its open support. 
\end{example}

Throughout this paper, we use the following notation:
\begin{defn}\label{defn-eps-cut}
For any $\eps>0$, define the function $f_\eps: [0, 1] \to [0, 1]$ by 
$$f_{\eps}(t) = \left\{\begin{array}{ll} 0, & t<\eps/2, \\ \textrm{linear}, & \eps/2 \leq t < \eps, \\ 1, & t\geq \eps .\end{array}\right.$$
\end{defn}

\begin{lem}[Proposition 2.4(iv) of \cite{RorUHF2}]\label{rordom-lem}
Let $A$ be a C*-algebra, and let $a, b\in A$ be positive. If $a \precsim b$, then for any $\delta>0$, there is $\eps>0$ and $r\in A$ such that $$f_\delta(a) = r^*f_\eps(b)r.$$ In particular, denoted by $v = f^{\frac{1}{2}}_\eps(b)r\in A$, one has
$$f_\delta(a) = v^*v\quad\mathrm{and}\quad vv^* \in\mathrm{Her}(b).$$
\end{lem}


\begin{defn}\label{rank-fn}
Let $A$ be a C*-algebra, let $\mathrm{T}(A)$ denote the set of all tracial states of $A$, equipped with the topology of pointwise convergence. Note that if $A$ is unital, the set $\mathrm T(A)$ is a Choquet simplex.

Let $a$ be a positive element of $\mathrm{M}_\infty(A)$ and $\tau \in \mathrm{T}(A)$; define
$$\mathrm{d}_\tau(a):=\lim_{n\to\infty}\tau(a^{\frac{1}{n}})=\mu_\tau(\mathrm{sp}(a)\cap(0, +\infty)),\quad a\in A^+,$$ 
where $\mu_\tau$ is the Borel probability measure induced by $\tau$ on the spectrum of $a$. It is well known that if $a \precsim b$, then
$$\mathrm{d}_\tau(a) \leq \mathrm{d}_\tau(b),\quad \tau\in\mathrm{T}(A).$$

\end{defn}

\begin{example}
Consider $h\in \mathrm{C}(X)^+$ and let $\mu$ be a Borel probability measure on $X$, where $X$ is a compact Hausdorff space. Then
$$\mathrm{d}_{\tau_\mu} = \mu(f^{-1}(0, +\infty)),$$
where $\tau_\mu$ is the trace of $\mathrm{C}(X)$ defined by $$\tau_\mu(f) = \int f\mathrm{d}\mu,\quad f\in\mathrm{C}(X).$$

If $A = \mathrm{M}_n(\mathrm{C}_0(X))$, where $X$ is a locally compact Hausdorff space, then, for any positive element $a\in \mathrm{M}_\infty(A) \cong \mathrm{M}_\infty(\mathrm{C}_0(X))$ and any $\tau\in\mathrm T(A)$, one has $$\tau(a) =\int_{X} \frac{1}{n}\mathrm{Tr}(a(x))d\mu_\tau \quad\mathrm{and}\quad \mathrm{d}_\tau(a) =\int_{X} \frac{1}{n}\mathrm{rank}(a(x))d\mu_\tau, $$ where $\mu_\tau$ is the Borel measure on $X$ induced by $\tau$.

\end{example}

\begin{defn}\label{defn-phi-E}
For each open set $E \subseteq X$, pick a continuous function 
\begin{equation}
\varphi_{E}: X \to [0, 1] 
\end{equation} 
such that 
\begin{enumerate}
\item $E = \varphi_E^{-1}((0, 1])$ and
\item there is an open set $V\subseteq E$ such that $\varphi_E|_V = 1$. (In particular, $\norm{\varphi_E} = 1$.)
\end{enumerate}

For instance, one can pick $\varphi_E(x) = \min\{\frac{1}{\eps}d(x, X\setminus E), 1\}$, where $d$ is a compatible metric on $X$ and $\eps>0$ is sufficiently small. This notation will be used throughout this paper. 

Note that the hereditary sub-C*-algebra $\overline{\varphi_EA\varphi_E}$ is independent of the choice of individual function $\varphi_E$, where $A$ is a C*-algebra containing $\mathrm{C}(X)$ (in a specified way). Therefore, we shall also denote $\overline{\varphi_EA\varphi_E}$ by $\mathrm{Her}(E)$.
\end{defn}

\subsection{Order zero maps and Rokhlin towers}
\begin{defn}[Order zero maps]

Let $A, B$ be C*-algebras. A linear map $\phi: A \to B$ is said to be of order zero if $$a \perp b \Longrightarrow \phi(a) \perp \phi(b),\quad a, b\in A^+.$$

\end{defn}

Let $A$ be a C*-algebras and $\phi: \mathrm{M}_n(\Comp) \to A$ be a c.p.~order zero map. Consider the C*-algebra $C:=\mathrm{C}^*(\phi(\mathrm{M}_n)) \subseteq A$, and set $\phi(1_n) = h$. Then, by \cite{WZ-ndim}, $h\in C\cap C'$, $\norm{h} = \norm{\phi}$, and there is a homomorphism $\pi_\phi: \mathrm{M}_n(\Comp) \to \mathcal M(C)\cap \{h\}'\subseteq A^{**}$ such that $$\phi(a) = \pi_\phi(a) h,\quad a\in \mathrm{M}_n(\Comp).$$ Moreover, $$C \cong \mathrm{M}_n(\mathrm{C}_0((0, 1])).$$

\begin{defn}\label{cp-cal}
Let $\phi: \mathrm{M}_n(\Comp) \to A$ be a c.p.~order zero map, and let $f\in \mathrm{C}_0((0, \norm{h}])$, where $h=\phi(1_n)$. Then the map $$\mathrm{M}_n(\Comp) \ni a \mapsto \pi_\phi(a)f(h) \in A$$ is again an order zero map, where $\pi_\phi$ is as above. Denote this new order zero map by $f(\phi)$.

An order zero map $\psi: \mathrm{M}_n(\Comp) \to A$  is said to be extendable if there is a c.p.~order zero map $\psi': \mathrm{M}_n(\Comp) \to A$ such that $\psi=f_{\delta}(\psi')$ for some $\delta>0$.
\end{defn}

The following staement is a characterization of c.p.~order zero maps with domain a matrix algebra:
\begin{lem}[\cite{Winter-Cdim-II}]\label{dd-0}
Let $v_1, v_2, ..., v_{n}\in A$, where $A$ is a C*-algebra, such that 
\begin{itemize}
\item $v^*_1v_1 = v_2^*v_2 = \cdots = v_n^*v_n$,
\item $v^*_1v_1, v_1v_1^*, v_2v_2^*, ..., v_nv_n^*$ are mutually orthogonal, and
\item $\norm{v^*_1v_1} = 1$.
\end{itemize}
Then there is an order zero map $\phi: \mathrm{M}_{n+1}(\Comp) \to A$ such that $\norm{\phi} = 1$,
$$\phi(e_{0, 0}) = v_1^*v_1,\quad \phi(e_{i, i}) = v_iv_i^*,\quad i=1, 2, ..., n.$$
Conversely, every c.p.~order zero map $\mathrm{M}_{n+1}(\Comp) \to A$ with norm $1$ arises in this way, for suitable $v_1, ..., v_n$, which may be taken to be $\phi(e_{1, 0})$, $\phi(e_{2, 0})$, ..., $\phi(e_{n, 0})$.
\end{lem}

\begin{defn}
A Rokhlin tower for  a dynamical system $(X, \Gamma)$ is a pair $(B, \Gamma_0)$, where $B\subseteq X$ and $\Gamma_0\subseteq\Gamma$ is finite, such that the sets $$B\gamma,\quad \gamma\in\Gamma_0,$$ are mutually disjoint. It is called an open tower if the base set $B$ is open. Without loss of generality, one may assume that $\Gamma_0$ contains the unit of $\Gamma$.
\end{defn}

One can naturally construct order-zero maps from Rokhlin towers, as follows.

Let $(B, \Gamma_0)$ be a tower, 
and pick a positive function $e: X \to [0, 1]$ such that $e^{-1}((0, 1])\subseteq B$. 
Let $\gamma_1, \gamma_2 \in \Gamma_0$ and consider the contraction
$$ v:=u^*_{\gamma_2}e^{\frac{1}{2}}u_{\gamma_1} \in \mathrm{C}(X)\rtimes\Gamma.$$ Then
$$v^*v = u^*_{\gamma_1} e u_{\gamma_1} = e\circ\gamma_1^{-1}\quad\mathrm{and}\quad vv^* = u^*_{\gamma_2} e u_{\gamma_2}=e\circ\gamma_2^{-1}.$$

In general, if $F_1, F_2 \subseteq \Gamma_0$ are two disjoint sets with $\abs{F_1} = \abs{F_2}$, pick a one-to-one correspondence $\theta: F_1 \to F_2$, and consider the element 
$$v := \sum_{\gamma \in F_1}u^*_{\theta(\gamma)}e^{\frac{1}{2}}u_{\gamma} \in \mathrm{C}(X)\rtimes\Gamma.$$ Then
\begin{eqnarray*}
v^*v & = & \sum_{\gamma_1, \gamma_2\in F_1} u^*_{\gamma_1}e^{\frac{1}{2}}u_{\theta(\gamma_1)} u^*_{\theta(\gamma_2)}e^{\frac{1}{2}}u_{\gamma_2} \\
& = & \sum_{\gamma_1, \gamma_2\in F_1} u^*_{\gamma_1}u_{\theta(\gamma_1)}(e^{\frac{1}{2}}\circ\theta(\gamma_1)^{-1}) (e^{\frac{1}{2}} \circ \theta(\gamma_2)^{-1}) u^*_{\theta(\gamma_2)}u_{\gamma_2} \\
& = & \sum_{\gamma\in F_1} u^*_{\gamma}e u_{\gamma},
\end{eqnarray*}
and the same calculation shows that
\begin{equation*}
vv^* =  \sum_{\gamma\in F_1} u^*_{\theta(\gamma)}e u_{\theta(\gamma)} = \sum_{\gamma\in F_2} u^*_{\gamma}e u_{\gamma}.
\end{equation*}

Now, suppose there are given mutually disjoint sets $$\Gamma_{0, 1}, \Gamma_{0, 2}, ..., \Gamma_{0, n} \subseteq \Gamma_0$$ such that $$\abs{\Gamma_{0, 1}} = \abs{\Gamma_{0, 2}} = \cdots = \abs{\Gamma_{0, n}}.$$

Consider the orthogonal positive elements $$e_1 : = \sum_{\gamma\in\Gamma_{0, 1}}u^*_\gamma e u_\gamma,\ ... ,\  e_n := \sum_{\gamma\in\Gamma_{0, n}}u^*_\gamma e u_\gamma.$$ Then the calculation above shows that there are $v_1, v_2, ..., v_{n-1}$ such that
$$ v^*_1v_1 = v_2^*v_2 = \cdots = v_{n-1}^*v_{n-1} = e_1 $$
and
$$v_1v_1^* = e_2, v_2v_2^* = e_2, ..., v_{n-1}v_{n-1}^* = e_n.$$ So, by Lemma \ref{dd-0}, there is an order zero map $\phi: \mathrm{M}_{n}(\Comp) \to A$ such that $$\ \phi(e_{i, i}) = e_i,\quad i=1, ..., n.$$ 

In summary, one has the following lemma:
\begin{lem}\label{existence-0-map}
Let $(B, \Gamma_0)$ be a Rokhlin tower, and let $\Gamma_{0, 1}, \Gamma_{0, 2}, ..., \Gamma_{0, n} \subseteq \Gamma_0$ be mutually disjoint sets such that 
$$\abs{\Gamma_{0, 1}} = \abs{\Gamma_{0, 2}} = \cdots = \abs{\Gamma_{0, n}}.$$ 
Let $e: X\to [0, 1]$ be a continuous function such that $e^{-1}((0, 1]) \subseteq B$. Set $$e_1 : = \sum_{\gamma\in\Gamma_{0, 1}}u^*_\gamma e u_\gamma,\ ... ,\  e_n := \sum_{\gamma\in\Gamma_{0, n}}u^*_\gamma e u_\gamma.$$ Then there is an order zero map $\phi: \mathrm{M}_{n}(\Comp) \to A$ such that
$$\phi(e_{i, i}) = e_i,\quad i=1, ..., n.$$
\end{lem}

\subsection{Uniform Rokhlin Property and Cuntz-comparison of Open Sets}

\begin{defn}[\cite{Niu-MD-Z}]\label{Defn-URP}
A dynamical system $(X, \Gamma)$ is said to have the uniform Rokhlin property (URP) if for any finite set $\mathcal F$, and any $\eps>0$, there are open Rokhlin towers $(B_1, F_1)$, ..., $(B_S, F_S)$ such that $F_1, F_2, ..., F_S$ are $(\mathcal F, \eps)$-invariant, the sets  
$$B_s\gamma,\quad \gamma\in F_s, s=1, 2, ..., S$$
are mutually disjoint, and
$$\mu(X\setminus\bigsqcup_{s=1}^S\bigsqcup_{\gamma\in F_s}B_s\gamma) < \eps,\quad \mu\in\mathcal M_1(X, \Gamma).$$
\end{defn}

\begin{rem}
If $E\subseteq X$ is a closed subset, then $\mu(E) < \eps $ for all $\mu\in\mathcal{M}_1(X, \Gamma)$  if, and only if, the orbit capacity of $E$ is at most $\eps$ (see, for instance, \cite{Lindenstrauss-Weiss-MD}).
\end{rem}

\begin{defn}[\cite{Niu-MD-Z}]\label{Defn-COS}
A topological dynamical system $(X, \Gamma)$ is said to have ($\lambda, m$)-Cuntz-comparison of open sets, where $\lambda\in (0, +\infty)$ and $m\in\mathbb N$,  if, for any open sets $E, F \subseteq X$ with $$\mu(E) \leq \lambda \mu(F),\quad \mu\in\mathcal{M}_1(X, \Gamma),$$ one has $$\varphi_E \precsim \underbrace{\varphi_F\oplus\cdots\oplus \varphi_F}_m\quad\mathrm{in}\ \mathrm{M}_{\infty}(\mathrm{C}(X)\rtimes\Gamma).$$

A topological dynamical system is said to have Cuntz-comparison of open sets (COS) if it has ($\lambda, m$)-Cuntz-comparison of open sets for some $\lambda$ and $m$.
\end{defn}

\begin{thm}[Theorem 4.2 and Theorem 5.5 of \cite{Niu-MD-Zd}]
Any free and minimal dynamical system $(X, \Int^d)$ has the (URP) and (COS).
\end{thm}

\section{Some lemmas}

In this section, we shall develop some lemmas on dimension drop C*-algebras and order zero c.p.c.~maps with domain a matrix algebra. Let us start with a simple observation: 

\begin{lem}\label{one-lem}
Let $a, c$ be elements of a unital C*-algebra, and assume that $ac=c$ and $a$ is self-adjoint.  Then, 
$$f(a)c = f(1)c,\quad f\in\mathrm{C}([-\norm{a}, \norm{a}]),$$
where $f(a)$ denotes the continuous functional calculus acting on $a$ using $f$.
\end{lem}
\begin{proof}

If $f(t) = \sum_{k=0}^n c_kt^k$, then $$f(a)c = (\sum_{k=0}^n c_ka^k)c = \sum_{k=0}^n c_ka^kc = \sum_{k=0}^n c_kc  = (\sum_{k=0}^n c_k)c = f(1)c. $$ The general statement follows from the Weierstrass Theorem and Spectral Theorem.
\end{proof}

It is well known that the universal unital C*-algebra generated by $v$ with respect to relations
$$vv^* \perp v^*v\quad \mathrm{and}\quad \norm{vv^*} \leq 1$$ is
$$D=\{f: [0, 1] \to \mathrm{M}_2(\Comp): f(0) \in \Comp 1_2\},$$
with $v$ corresponding to 
$$[0, 1] \ni t \mapsto \left( \begin{array}{cc} 0 & \sqrt{t} \\ 0 & 0 \end{array} \right).$$ Using this identification, one has the following statement:

\begin{lem}\label{switch}
Let $A$ be a unital C*-algebra, and let $v\in A$. Consider $a=vv^*$ and $b=v^*v$, and assume that $\norm{a} \leq 1$ and $a\perp b$. 
Define 
$$w = \cos(\frac{\pi}{2}(vv^*+v^*v)) + g(vv^*)v - g(v^*v)v^*,$$
where 
$$g(t)=\frac{\sin (\frac{\pi}{2}t)}{\sqrt{t}},\quad t\in(0, 1].$$
Then $w\in \mathrm{C}^*\{ v, 1_A\}$ is a unitary such that, 
\begin{equation}
\textrm{if $ac = ca =c$, then $b(w^*cw) = (w^*cw)b = w^*cw$},
\end{equation}
and,
\begin{equation}
\textrm{if $bc = cb =c$, then $a(wcw^*) = (wcw^*)a = wcw^*$}.
\end{equation}
\end{lem}

\begin{proof}
Noting that $vv^*+v^*v$ is central in $\mathrm{C}^*(1, v)$, $vv^*\perp v^*v$, and $$v^*h(vv^*) = h(v^*v)v^*,\quad h\in\mathrm{C}_0((0, 1]),$$ 
one has
\begin{eqnarray*} 
ww^* & = & (\cos(\frac{\pi}{2}(vv^*+v^*v)) + g(vv^*)v - g(v^*v)v^*)\\
&& (\cos(\frac{\pi}{2}(vv^*+v^*v)) + v^*g(vv^*) - vg(v^*v))  \\
& = & \cos^2(\frac{\pi}{2}(vv^*+v^*v))  +  g^2(vv^*)v v^* + g^2(v^*v)v^*v\\
& = & \cos^2(\frac{\pi}{2}(vv^*+v^*v)) + \sin^2(\frac{\pi}{2}v^*v) + \sin^2(\frac{\pi}{2}vv^*)  \\
& = & \cos^2(\frac{\pi}{2}(vv^*+v^*v)) + \sin^2(\frac{\pi}{2}(v^*v+vv^*)) \\
& = &  1.
\end{eqnarray*}
The same calculation shows that $w^*w=1$, and so $w$ is unitary.

Let $c\in A$ be an element satisfying $$ac=ca =c.$$ Note that, 
for any $f\in\mathrm{C}([0, 1])$, by Lemma \ref{one-lem}, one has
\begin{equation}\label{identity}
f(a)c = f(1)c = cf(a).
\end{equation} 
Therefore,
\begin{eqnarray*}
w^*cw & = &  (\cos(\frac{\pi}{2}(vv^*+v^*v)) + v^*g(vv^*) - vg(v^*v)) c \\
&&(\cos(\frac{\pi}{2}(vv^*+v^*v)) + g(vv^*)v - g(v^*v)v^*) \\
& = & (\cos(\frac{\pi}{2}(a)) + v^*g(vv^*)) c (\cos(\frac{\pi}{2}(a)) + g(vv^*)v)\\
& = & v^*g(vv^*) c g(vv^*)v \\
& = &v^*cv,
\end{eqnarray*}
and hence
\begin{equation*}
(w^*cw)b = v^*cvb = v^*cvv^*v = v^*cav = v^*cv = w^*cw
\end{equation*}
and
\begin{equation*}
b(w^*cw) = bv^*cv = v^*vv^*cv = v^*acv = v^*cv = w^*cw.
\end{equation*}

A similar calculation shows that 
$$a(wcw^*) = (wcw^*)a = wcw^*$$
if $bc = cb =c$.
\end{proof}

The following lemma is crucial in the proof of Proposition \ref{prop-ab-tsr1}, in which it provides an element that behaves in the same way as a lower triangular matrix.
\begin{lem}\label{n-shifting}
Let $A$ be a unital C*-algebra, and let $v_1, v_2, ..., v_n\in A$ be elements satisfying
\begin{itemize}
\item $v^*_1v_1 = v_2^*v_2 = \cdots = v_n^*v_n$,
\item $v^*_1v_1, v_1v_1^*, v_2v_2^*, ..., v_nv_n^*$ are mutually orthogonal, and
\item $\norm{v^*_1v_1} = 1$.
\end{itemize}
Then there is a unitary $w\in A$ with the following properties:

\begin{enumerate}

\item\label{plus-1-shifting} If $E_1, E_2, ..., E_n\in A$ are mutually orthogonal positive elements such that
$$[v_1^*v_1, E_i] = 0\quad\mathrm{and}\quad (v_iv_i^*) E_{i} = v_iv_i^*,\quad i=1, 2, ..., n,$$ 
then $$wE_i\in \overline{(E_1 + E_i + E_{i+1})A},\quad 1\leq i\leq n-1.$$

\item\label{keep-zero} If $D\subseteq A$ is a hereditary sub-C*-algebra such that $$v_iv_i^* \in D,\quad i=1, 2, ..., n,$$ 
and $d\in D$ is an element such that $$[d, v_1^*v_1] =0\quad\textrm{and}\quad v_n^* d = 0,$$ then $$wd\in\overline{DA}.$$

\item\label{squeez-zero} If $c\in A^+$ satisfies $$(v^*_1v_1)c=c,$$ then $$wc \in \overline{(v_1v_1^*)A}.$$

\end{enumerate}
\end{lem}

\begin{proof}

Consider the universal C*-algebra $\mathcal A$ generated by $v_1, v_2, ..., v_n$ with respect to the relations
\begin{itemize}
\item $v^*_1v_1 = v_2^*v_2 = \cdots = v_n^*v_n$,
\item $v_1^*v_1, v_1v_1^*, v_2v_2^*, ..., v_nv_n^*$ are mutually orthogonal,
\item $\norm{v^*_1v_1} = 1$.
\end{itemize}
It is well known (see \cite{Lor-Mon}) that $\mathcal A$ is isomorphic to the dimension drop C*-algebra
\begin{equation}\label{dim-drop}
D_{n+1}:=\{f\in\mathrm{C}([0, 1], \mathrm{M}_{n+1}): f(0) = 0_{n+1}\}\cong \mathrm{C}_0((0, 1])\otimes\M{n+1}{\Comp}
\end{equation}
with $$v_i(t) = \sqrt{t} \otimes e_{i, 0},\quad t\in[0, 1],\ i=1, 2, ..., n, $$
where $e_{i, j}$, $i, j = 0, 1, ..., n$, are the canonical matrix units of $\mathrm{M}_{n+1}(\Comp)$.

With the identification \eqref{dim-drop}, consider the unitary $w\in D_{n+1}+\Comp 1$ defined by
\begin{equation*}
w(t) = \left\{\begin{array}{ll} 
1_{n+1}, & \textrm{if $t=0$}, \\
\\
\left(\begin{array}{cccccc}
 1_{n-k}  & & & & &  \\
 & \cos\frac{\pi}{2}n(t-\frac{k-1}{n}) & 0& 0&\cdots & \sin\frac{\pi}{2}n(t-\frac{k-1}{n}) \\
 & -\sin\frac{\pi}{2}n(t-\frac{k-1}{n})& 0 & 0 & \cdots &  \cos\frac{\pi}{2}n(t-\frac{k-1}{n}) \\
 & & -1& 0 &\cdots &  0 \\
 & & & \ddots & \ddots&  \vdots \\
 & & & & -1& 0
\end{array}\right), & \textrm{if $t\in [\frac{k-1}{n}, \frac{k}{n}]$.}
\end{array}\right.
\end{equation*}
Then the image of $w$ in $A$, which will still be denoted by $w$, has the properties of the lemma.

Indeed, put
$$ w_{i, j} = v_iv^*_j,\quad i, j=1,  ..., n,$$
and
$$w_{0, 0} = v_1^*v_1,\quad w_{i, 0} = v_i,\quad w_{0, j} = v_j^*,\quad i, j=1, 2, ..., n.$$
Then
\begin{equation*}
w=1+ \sum_{i=0}^n g_{i, i}(w_{i, i}) + \sum_{i=1}^n g_{i, i-1}(w_{i, i})w_{i, i-1} + \sum_{i=0}^{n-1} g_{i, n}(w_{i, i})w_{i, n},
\end{equation*}
for some functions $g_{i, j} \in \mathrm{C}_0((0, 1])$. Note that 
\begin{equation}\label{g-condition}
g_{0, 0}(1) = -1.
\end{equation}

Let $E_1, E_2, ..., E_n$ be mutually orthogonal positive elements of $A$ satisfying 
\begin{equation}\label{E-over}
w_{i, i} E_{i} = w_{i, i},\quad i=1, 2, ..., n,
\end{equation}
and 
\begin{equation}\label{E-comm}
[w_{0, 0}, E_i] = 0.
\end{equation}
Note that, by \eqref{E-over} and polar decomposition, one has
$$v^*_jE_i = v_j' (v_jv^*_j)^{\frac{1}{2}}E_i = v_j' w_{j, j}^{\frac{1}{2}}E_i  = v_j' w_{j, j}^{\frac{1}{2}}E_j E_i = 0,\quad j\neq i, \ 1\leq i, j \leq n,$$ where $v_j'$ is a partial isometry in the bidual.  
Then, for any $1\leq i\leq n-1$, 
$$w_{i_1, i_2} E_i = v_{i_1}v_{i_2}^*E_i = 0,\quad i_2\neq i,\  1 \leq i_1, i_2\leq n,$$
and hence  
\begin{eqnarray*}
wE_i & = & (1+ \sum_{j=0}^n g_{j, j}(w_{j, j}) + \sum_{j=1}^n g_{j, j-1}(w_{j, j})w_{j, j-1} + \sum_{j=0}^{n-1} g_{j, n}(w_{j, j})w_{j, n}) E_i \\
& = & E_i + g_{0, 0}(w_{0, 0}) E_i + g_{i, i}(w_{i, i})E_i + g_{1, 0}(w_{1, 1})w_{1, 0} E_i + g_{i+1, i}(w_{i+1, i+1})w_{i+1, i} E_i. 
\end{eqnarray*}
By \eqref{E-comm}, $$g_{0, 0}(w_{0, 0})E_i =E_i g_{0, 0}(w_{0, 0})\in E_iA.$$
By \eqref{E-over}, $$g_{i, i}(w_{i, i}) E_i = E_i g_{i, i}(w_{i, i})\in E_iA,$$
$$g_{1, 0}(w_{1, 1})w_{1, 0} E_i = E_{1}g_{1, 0}(w_{1, 1})w_{1, 0} E_i \in E_1A,$$
and
$$g_{i+1, i}(w_{i+1, i+1})w_{i+1, i} E_i = E_{i+1}g_{i+1, i}(w_{i+1, i+1})w_{i+1, i} E_i \in E_{i+1}A.$$
Therefore,
$$wE \in  \overline{(E_1+E_i + E_{i+1})A},\quad 1\leq i\leq n-1,$$ as required for Property \eqref{plus-1-shifting}.

For Property \eqref{keep-zero}, let $D$ be a hereditary sub-C*-algebra such that $$w_{i, i} \in D,\quad i=1, 2, ..., n,$$ 
and let $d\in D$ be an element satisfying $$[d, w_{0, 0}] =0\quad\mathrm{and}\quad w_{0, n} d = 0.$$ Then 
\begin{eqnarray*}
wd & = & (1+ \sum_{i=0}^n g_{i, i}(w_{i, i}) + \sum_{i=1}^n g_{i, i-1}(w_{i, i})w_{i, i-1} + \sum_{i=0}^{n-1} g_{i, n}(w_{i, i})w_{i, n}) d\\
& = & d+ g_{0, 0}(w_{0, 0})d +\sum_{i=1}^n g_{i, i}(w_{i, i})d + \sum_{i=1}^n g_{i, i-1}(w_{i, i})w_{i, i-1}d + \sum_{i=1}^{n-1} g_{i, n}(w_{i, i})w_{i, n}d \\
& \in & \overline{DA}.
\end{eqnarray*}

For Property \eqref{squeez-zero}, let $c \in A$ be a positive element such that $$w_{0, 0}c = c.$$ Then, by Lemma \ref{one-lem} and \eqref{g-condition}, one has
\begin{eqnarray*}
wc & = & (1+ \sum_{i=0}^n g_{i, i}(w_{i, i}) + \sum_{i=1}^n g_{i, i-1}(w_{i, i})w_{i, i-1} + \sum_{i=0}^{n-1} g_{i, n}(w_{i, i})w_{i, n}) c \\
& = & c + g_{0, 0}(w_{0, 0})c + g_{1, 0}(w_{1, 1})w_{1, 0}c \\
& = &(1+g_{0, 0}(1))c+ g_{1, 0}(w_{1, 1})w_{1, 0}c \\
& = & g_{1, 0}(w_{1, 1})w_{1, 0} c \in \overline{w_{1, 1} A} = \overline{v_1v_1^*A},
\end{eqnarray*}
as required.
\end{proof}

\begin{lem}\label{p-unitary}
Let $A$ be a unital C*-algebra, and let $\phi: \mathrm{M}_n(\Comp)\to A$ be an extendable order zero c.p.c.~map. Denote by $e_i=\phi(e_{i, i})$ and $h=\phi(1_n)$. Then, for any permutation $\sigma: \{1, 2, ..., n\} \to \{1, 2, ..., n\}$, there is a unitary $u\in A$ such that $$[u, h] = 0\quad\mathrm{and}\quad u^*e_iu = e_{\sigma(i)},\quad i=1, 2, ..., n.$$ 
\end{lem}
\begin{proof}
Since $\phi$ is extendable, there is an order zero map $\tilde{\phi}: \mathrm{M}_n(\Comp) \to A$ and $\delta>0$ such that $\phi=f_{\delta}(\tilde{\phi})$. Note that 
$$C:=\mathrm{C}^*\{\tilde{\phi}(\mathrm{M}_n(\Comp))\} \cong \{f:[0, 1]\to \mathrm{M}_n(\Comp): f(0) = 0_n\} \cong \mathrm{C}_0((0, 1])\otimes \mathrm{M}_n(\Comp)$$ with $e_i = f_\delta \otimes e_{i, i}$, $i=1, 2, ..., n$, under the isomorphism. Denote by $U\in \mathrm{M}_n(\Comp)$ the permutation unitary matrix such that $$U^*e_{i, i}U = e_{\sigma(i), \sigma(i)},\quad i=1, 2, ..., n.$$ Since the unitary group of $\mathrm{M}_n(\Comp)$ is path connected, there is a continuous path of unitaries $$[0, \frac{\delta}{2}]\ni t\mapsto U_t\in\M{n}{\Comp}$$ such that $U_0=1_n$ and $U_{\delta/2} = U$. Set
$$u: t\mapsto u(t) = \left\{\begin{array}{ll}
U_t, & t\in[0, \frac{\delta}{2}], \\
U, & t\in [\frac{\delta}{2}, 1].
\end{array}\right.$$
Then the unitary $u\in C+ \Comp 1_A\subseteq A$ has the desired property.
\end{proof}

\section{Nilpotent elements, order zero maps, and limits of invertible elements}


An element $a$ of a C*-algebra is said to be nilpotent if $a^n=0$ for some $n\in\mathbb N$. It is well known that if $a$ is nilpotent, then $a+\eps 1_A$ is invertible for any non-zero $\eps$; in fact,
\begin{equation}\label{nil-inv}
(a+\eps 1_A)^{-1} = \frac{1_A}{\eps} - \frac{a}{\eps^2} + \frac{a^2}{\eps^3} + \cdots + (-1)^{n-1}\frac{a^{n-1}}{\eps^{n}} + \cdots,
\end{equation}
where the infinite series is eventually zero, as $a$ is nilpotent. Hence any nilpotent element is in the closure of invertible elements.

The following lemma is a modified version of Lemma 4.2 of \cite{A-NCP-LAlg}.
\begin{lem}[c.f. Lemma 4.2 of \cite{A-NCP-LAlg}]\label{invt-approx}
Let $A$ be a finite unital C*-algebra and let $a\in A$. Suppose that there exist positive elements $b_1, b_2, c_1, c_2$ such that
\begin{enumerate}
\item\label{unit-b} $b_1 + b_2 = 1_A$,
\item $c_1 + c_2 = 1_A$,
\item $\mathrm{C}^*\{b_1, b_2, c_1, c_2\}$ is commutative,
\item $b_1c_1 = c_1$,
\item $b_2c_2 = b_2$,
\item $c_1b_2 = 0$,
\item\label{nil-assumption} there are unitaries $u, v\in A$ such that 
$$u(b_2a(b_1-c_1)+b_2ab_2)v\quad\textrm{and}\quad u(b_2ab_2)v$$ are nilpotent,
\begin{equation}\label{in-Her}
v(u(b_2ab_2)v)^nu\in \overline{b_2Ab_2},\quad n=1, 2, ..., 
\end{equation}
and
$$c_1vub_2 = 0$$
\item\label{zero-left} $b_1a = 0$.
\end{enumerate}
Then $a\in \overline{\mathrm{GL}(A)}$
\end{lem}

\begin{proof}
The proof is the similar to that of Lemma 4.2 of \cite{A-NCP-LAlg}, but with $c_3=b_3=0$.

Fix an arbitrary $\eps>0$ for the time being, and note that
$$1_A=c_1 + (b_1-c_1) + b_2.$$
Put
$$a_{3, 1} = b_2ac_1$$
$$a_{3, 2} = b_2a(b_1-c_1)$$
$$a_{3, 3} = b_2ab_2.$$
Then, by \eqref{unit-b} and \eqref{zero-left},
$$ a_{3, 1} + a_{3, 2} + a_{3, 3} = b_2 a (c_1 + (b_1-c_1)+b_2) = b_2 a (b_1+b_2) = b_2 a= (b_1+b_2) a = a.$$

Note that
$$a_{3, 3} = u^{-1}a'_{3, 3}v^{-1},$$ where  the element  $$a'_{3, 3} := ua_{3, 3}v = u(b_2ab_2)v$$ is nilpotent (by \eqref{nil-assumption}).

Put $$t_0=u^{-1}(a'_{3, 3} + \eps1_A)v^{-1}.$$ Then $t_0$ is invertible and
\begin{equation}\label{reduce-0}
\norm{t_0 - a_{3, 3}} = \norm{ \eps u^{-1}v^{-1}} = \eps.
\end{equation}

Using \eqref{nil-inv}, write
$$t_0^{-1} =v(\frac{1_A}{\eps} - \frac{a'_{3, 3}}{\eps^2} + \frac{(a'_{3, 3})^2}{\eps^3} + \cdots + (-1)^{n-1}\frac{(a'_{3, 3})^{n-1}}{\eps^{n}}+\cdots)u = t_2 + \frac{vu}{\eps},$$ 
where, by \eqref{in-Her}, $$t_2 := -\frac{va_{3, 3}'u}{\eps^2} + \frac{v(a_{3, 3}')^2u}{\eps^3} + \cdots + (-1)^{n-1} \frac{v(a_{3, 3}')^{n-1}u}{\eps^{n}} +\cdots \in \overline{b_2Ab_2}.$$ (Note that the terms in the series above are eventually zero.) 
Therefore, as $$c_1b_2=0 \quad\mathrm{and}\quad c_1vub_2 = 0,$$ one has
\begin{equation}\label{0-prod-1}
a_{3, 1}t_0^{-1}a_{3, 1}  =  b_2ac_1 (t_2+\frac{vu}{\eps} )b_2ac_1 = 0
\end{equation}
and
\begin{equation}\label{0-prod-2}
a_{3, 1}t_0^{-1}a_{3, 2}  =  b_2ac_1 (t_2+\frac{vu}{\eps} )b_2a(b_1-c_1) = 0.
\end{equation}
Hence by \eqref{0-prod-1}, one has
\begin{equation}\label{reduce-1}
(a_{3, 1}+t_0)t_0^{-1}(1_A-a_{3, 1}t_0^{-1}) = a_{3, 1}t_0^{-1} - a_{3, 1} t_0^{-1}a_{3, 1} t_0^{-1} + 1_A -a_{3, 1}t_0^{-1}=1_A
\end{equation}
and
$$t_0^{-1}(1_A-a_{3, 1}t_0^{-1})(a_{3, 1}+t_0) = t_0^{-1} a_{3, 1} + 1_A -t_{0}^{-1}a_{3, 1}t_0^{-1}a_{3, 1} - t_0^{-1}a_{3, 1}= 1_A.$$
In particular, the element $(a_{3, 1}+t_0)$ is invertible with
$$(a_{3, 1}+t_0)^{-1} = t_0^{-1}(1_A-a_{3, 1}t_0^{-1}).$$
Then, by this together with \eqref{0-prod-2}, one has
\begin{eqnarray}\label{reduce-2}
& & t_0^{-1}(1_A-a_{3, 1}t_0^{-1})(a_{3, 1} + a_{3, 2} + t_0) \\
& = &  (a_{3, 1} + t_0)^{-1} (a_{3, 1} + a_{3, 2} + t_0) \nonumber \\
& = & 1_A + (a_{3, 1} + t_0)^{-1}a_{3, 2} \nonumber \\
& = & 1_A + t_0^{-1}(1-a_{3, 1}t_0^{-1}) a_{3, 2} \nonumber \\
& = & 1_A + t_0^{-1} a_{3, 2}. \nonumber 
\end{eqnarray}

Consider the element 
\begin{eqnarray}\label{defn-t}
t_1 & :=& 1_A+t_0^{-1} a_{3, 2} \\
&=&  t_0^{-1}(t_0 + a_{3, 2}) \nonumber \\
& = &t_0^{-1}(u^{-1}(a'_{3, 3} + \eps1_A)v^{-1} + u^{-1}a'_{3, 2}v^{-1}) \nonumber \\
& = & t_0^{-1}u^{-1}(a'_{3, 3}  + a'_{3, 2} + \eps1_A)v^{-1}, \nonumber
\end{eqnarray}
where $$a'_{3, 2} = ua_{3, 2}v = u(b_2a(b_1-c_1))v.$$
Since the element $$a'_{3, 3}  + a'_{3, 2} = u((b_2ab_2) + b_2a(b_1-c_1))v$$ is nilpotent,  $t_1$ is invertible.

Put 
$$y = (a_{3, 1}+t_0) t_1.$$
Since $a_{3, 1}+t_0$ and $t_1$ are invertible, one has $$y\in\mathrm{GL}(A).$$
We can now calculate as follows,  applying \eqref{reduce-1} in the third step, \eqref{reduce-2} in the fifth step, the definition of $t_1$ (see \eqref{defn-t}) in the sixth step, and \eqref{reduce-0} in the last step:
\begin{eqnarray*}
\norm{a-y} & = & \norm{a_{3, 1} + a_{3, 2} + a_{3, 3} - (a_{3, 1} + t_0)t_1} \\
& \leq & \norm{a_{3, 3} - t_0} + \norm{1_A(a_{3, 1} + a_{3, 2} + t_0) - (a_{3, 1} + t_0) t_1} \\
& \leq & \norm{a_{3, 3} - t_0} + \norm{(a_{3, 1} + t_0)t_0^{-1}(1_A-a_{3, 1}t_0^{-1}) (a_{3, 1}+a_{3, 2} + t_0) - (a_{3, 1} + t_0) t_1} \\
& \leq & \norm{a_{3, 3} - t_0} + \norm{a_{3, 1} + t_0} \norm{t_0^{-1}(1_A-a_{3, 1}t_0^{-1}) (a_{3, 1}+a_{3, 2} + t_0) - t_1} \\
& = & \norm{a_{3, 3} - t_0} + \norm{a_{3, 1} + t_0} \norm{1_A + t_0^{-1}a_{3, 2} - t_1} \\
& = & \norm{a_{3, 3} - t_0} = \eps.
\end{eqnarray*}
Since $\eps$ is arbitrary, this shows that $a\in\overline{\mathrm{GL}(A)}$.
\end{proof}

\begin{lem}\label{nilpotent-0}
Let $A$ be a unital C*-algebra, and let $\phi: \mathrm{M}_{n}(\Comp) \to A$ be an extendable order zero c.p.c~map (see Definition \ref{cp-cal}). Let $\delta \in (0, 1)$, and set $$\phi(1_n) = h \quad\textrm{and}\quad \phi(e_{i, i}) = e_i,\quad i=1, 2, ..., n,$$ and set $$f_\delta(h) = b_2, \quad 1-b_2 = b_1, \quad  f_{\delta/2}(h) = c_2, \quad  1-f_{\delta/2}(h) = c_1.$$ 

Then, for any $m\in\mathbb N$ with $m\leq n$ and any $$1\leq i_1 < i_2 < \cdots < i_m \leq n,$$ there are unitaries $u, v\in A$ with the following two properties:
\begin{enumerate}

\item $c_1vub_2 = 0$.

\item\label{trangular} 
If $a\in \overline{c_2Ac_2}$ is an element such that there is $d\in\mathbb N\cup\{0\}$ with $d< m$ such that 
$$e_i a e_j =0,\quad \textrm{whenever $j-i > d$}, $$ 
and
$$e_{i_j} a = a e_{i_j} = 0,\quad 1\leq j \leq m,$$ 
then $uav$ is nilpotent. If, moreover, $a\in\overline{b_2Ab_2}$, then $$v(uav)^ku\in \overline{b_2Ab_2},\quad k=1, 2, ... .$$
\end{enumerate}
\end{lem}

\begin{proof}
With the given $m$ and $i_1, i_2, ..., i_{m}$, define the permutation
$$\sigma:\{1, 2, ..., n\} \to \{1, 2, ..., n\}$$ 
by
stretching $\{1, 2, ..., n-m\}$ to $\{1, 2, ..., n\}\setminus\{i_1, i_2, ..., i_m\}$, and then moving $\{n-m+1, ..., n\}$ to fill $\{i_1, i_2, ..., i_m\}$. More precisely, write
$$
\left\{
\begin{array}{l}
I_0 = \{1 \leq i \leq n-m: i < i_1\}, \\
I_1 = \{1 \leq i \leq n-m : i_1\leq i,\ i+1 < i_2\}, \\
\cdots \\
I_{m-1} = \{1 \leq i \leq n-m : i_{m-1} \leq i + m-2,\ i+m-1 < i_m  \}, \\
I_m = \{1 \leq i \leq n-m : i_m \leq i+m-1 \},
\end{array}
\right.
$$
and note that $\{1, 2, ..., n-m\} = I_0 \sqcup I_1\sqcup \cdots \sqcup I_m$ (some of $I_k$, $k=1,..., m$, might be empty).
Then,
$$
\sigma(i) = \left\{
\begin{array}{ll}

i+k,    & \textrm{if $i\in I_k$}, \\

i_{i-n+m},   & n-m+1 \leq i\leq n.

\end{array}
\right.
$$

Note that for any $d\in\mathbb N\cup\{0\}$,
\begin{equation}\label{stretch}
\sigma(j) - \sigma(i) > d,\quad\textrm{if $1\leq i, j \leq n-m$ and $j-i > d$}.
\end{equation}

Since $\phi$ is extendable, by Lemma \ref{p-unitary}, there is a unitary $w_1\in A$ such that
$$[w_1, h]=0 \quad\mathrm{and}\quad w_1^*e_iw_1 = e_{\sigma(i)},\quad i=1, 2, ..., n.$$
By Lemma \ref{p-unitary} again, there is a unitary $w_2\in A$ satisfying the conditions $$[w_2, h] = 0$$ and
$$w_2^*e_iw_2 = \left\{
\begin{array}{ll}
e_{i+n-m}, & 1\leq i \leq m, \\
e_{i-m}, & m+1 \leq i \leq n.
\end{array}
\right.$$
Then, the unitaries $$u:=w_2w_1\quad\mathrm{and}\quad v:=w_1^*$$ possess the property of the lemma. 

Indeed, since $w_1$ and $w_2$ commute with $h$ and $c_1b_2 = 0$, one has $$c_1(vu)b_2 =  c_1(w_1^*w_2w_1) b_2 = (w_1^*w_2w_1) (c_1b_2) =  0.$$

Let $a\in \overline{c_2Ac_2}$  satisfy
$$e_i a e_j =0,\quad \textrm{whenever $j-i > d$} $$
for some (fixed) non-negative integer $d<m$,
and
$$e_{i_j} a = a e_{i_j} = 0,\quad 1\leq j \leq m.$$

Consider the element $w_1aw_1^*$. Note that, for any $n-m+1 \leq i \leq n$, $$\sigma(i) \in\set{i_1, i_2, ..., i_m}.$$ Hence,
\begin{equation}\label{zero-d-left}
e_i(w_1aw_1^*) = w_1(e_{\sigma(i)}a)w_1^* = 0,\quad n-m+1 \leq i \leq n
\end{equation}
and
\begin{equation}\label{zero-d-right}
(w_1aw_1^*)e_i = w_1(ae_{\sigma(i)})w_1^* = 0,\quad n-m+1 \leq i \leq n.
\end{equation}

Also note that, for any $1\leq i, j\leq n-m$ satisfying $j-i > d$, by \eqref{stretch}, also $\sigma(j) - \sigma(i) > d$, and hence $$e_i(w_1aw_1^*)e_j = w_1(e_{\sigma(i)}ae_{\sigma(j)})w_1^* = 0.$$ Together with \eqref{zero-d-left} and \eqref{zero-d-right}, this implies 
\begin{equation}\label{triangular-middle}
e_i(w_1aw_1^*)e_j = 0,\quad j-i>d,\ 1\leq i, j\leq n.
\end{equation}

Consider the element $uav=w_2w_1aw_1^*$, and note that for any $1\leq i, j \leq n$ with $j \geq i$,

\begin{itemize}

\item if $1\leq i \leq m$,  then $n-m+1 \leq i+n-m \leq n$, and by \eqref{zero-d-left}, $$e_i(w_2w_1aw_1^*)e_j = w_2 (e_{i+n-m}w_1aw_1^*)e_j = 0,$$ 

\item if $m+1 \leq i \leq n$, then $j-(i-m)\geq m > d$, and hence, by \eqref{triangular-middle}, $$e_i(w_2w_1aw_1^*)e_j = w_2 (e_{i-m}w_1aw_1^*e_j) = 0.$$

\end{itemize}
In other words, 
\begin{equation}\label{matrix-cut}
e_i(uav)e_j = 0,\quad j\geq i.
\end{equation}

Consider $$\tilde{e}_i := f_{\frac{\delta}{4}}(e_i),\quad i=1, 2, ..., n,$$
and it follows from \eqref{matrix-cut} that
 $$\tilde{e}_i(uav)\tilde{e}_j = 0,\quad j\geq i.$$

Since $uav \in \overline{c_2Ac_2}$, one has
$$uav = (\tilde{e}_1 + \cdots + \tilde{e}_n)uav (\tilde{e}_1 + \cdots + \tilde{e}_n) = \sum_{i, j=1}^n\tilde{e}_i(uav)\tilde{e}_j = \sum_{i > j}\tilde{e}_i(uav)\tilde{e}_j.$$
In particular, there is a decomposition 
$$uav = \sum_{i>j} a^{(1)}_{i, j},$$ where $a^{(1)}_{i, j} \in \overline{\tilde{e}_iA\tilde{e}_j}$. Direct calculation shows that
$$(uav)^2 = \sum_{i>k>j} a_{i, k}a_{k, j} = \sum_{i > j-1} a^{(2)}_{i, j},$$ where $a^{(2)}_{i, j} \in \overline{\tilde{e}_iA\tilde{e}_j}$.

Repeating this, one obtains $(uav)^{n} = 0.$
Thus, $uav$ is nilpotent.

If, moreover, $a \in \overline{b_2Ab_2}$, since $u$ and $v$ commute with $h$ (and hence commute with $b_2$), one has $$v(uav)^ku \in v(u(\overline{b_2Ab_2})v)^ku\subseteq \overline{b_2Ab_2},\quad k=1, 2, ... .$$
\end{proof}

\begin{prop}\label{trangular-invertible}
Let $A$ be a unital C*-algebra, and let $a\in A$. If there exist an extendable order zero c.p.c.~map $\phi: \mathrm{M}_{n}(\Comp) \to A$, a non-negative integer $d < m$, and $1\leq i_1 < i_2 < \cdots < i_m \leq n$ such that
\begin{enumerate}
\item $(1-h)a = 0$, where $h=\phi(1_{n})$,
\item $e_{i_j} a = a e_{i_j} = 0$, $j=1, 2, ..., m$, where $e_i = \phi(e_{i, i})$, and
\item $e_i a e_j =0$, if $j-i>d$,
\end{enumerate}
then $a\in\overline{\mathrm{GL}(A)}$.
\end{prop}

\begin{proof}
Pick an arbitrary $\delta\in(0, 1)$, and consider the elements $$b_2 = f_\delta(h),\quad c_2 = f_{\delta/2}(h),\quad b_1 = 1- f_\delta(h),\quad \mathrm{and}\quad c_1=1-f_{\delta/2}(h).$$ Note that $$c_2b_2=b_2,\quad b_1c_1 = c_1,\quad \mathrm{and}\quad c_1b_2 = 0.$$
Since $(1-h)a=0$, one has $a=ha$, and
$$b_2a = f_\delta(h) a = f_\delta(1)a = a.$$ Hence, $b_1a = 0.$

Consider the elements
$$b_2ab_2\quad\mathrm{and}\quad b_2a(b_1-c_1) + b_2ab_2,$$
and note that they are both in $\overline{c_2Ac_2}$. Also note that (since $h$ commutes with $e_i$, $i=1, 2, ..., n$)
$$e_ib_2ab_2e_j = b_2 e_i a e_j b_2 = 0,\quad   j-i>d,$$
and
$$e_{i_j} (b_2ab_2) = b_2e_{i_j}ab_2 = 0  = b_2a e_{i_j} b_2 = (b_2ab_2) e_{i_j},\quad j=1, 2, ..., m.$$
The same argument shows that
$$e_i(b_2a(b_1-c_1) + b_2ab_2)e_j = 0,\quad   j-i>d, $$
and
$$e_{i_j} (b_2a(b_1-c_1) + b_2ab_2)  = 0  = (b_2a(b_1-c_1) + b_2ab_2)e_{i_j},\quad j=1, 2, ..., m.$$

Denote by $u, v$ the unitaries in $A$ obtained by applying Lemma \ref{nilpotent-0} to $\phi$, $\delta$, $m$, and $i_1, i_2, ..., i_m$. Then, by Lemma \ref{nilpotent-0}, one has
$$c_1vub_2 = 0,$$
the elements 
$$u(b_2ab_2)v\quad\mathrm{and}\quad u(b_2a(b_1-c_1) + b_2ab_2)v$$
are nilpotent, and $$v(ub_2ab_2v)^nu \in \overline{b_2Ab_2},\quad n=1, 2, ....$$ Hence, by Lemma \ref{invt-approx},   $a\in\overline{\mathrm{GL}(A)}$.
\end{proof}

\section{Property (D) and stable rank one}

\begin{defn}\label{defn-Prop-D0}
Let $A$ be a unital C*-algebra. An element $a\in A$ will be said to be a $\mathcal D_0$-element if 
there exists a non-zero positive element $b\in A$ satisfying
\begin{equation*}
ba = ab = 0,
\end{equation*}
and there exists an order zero c.p.c.~map $$\phi:\mathrm{M}_{pq}(\Comp) \to A,$$ where $p, q \in\mathbb{N}$, and there exist $r, l\in\mathbb N$ such that, with $$e_i:=\phi(e_{i, i}),\quad i=1, 2, ..., pq,$$  
$$s_k:=e_{(k-1)p+1} + \cdots + e_{(k-1)p+r}, \quad k=1, ..., q,$$ and 
$$E_k:=e_{(k-1)p+1} + \cdots + e_{(k-1)p+p},\quad k=1, ..., q,$$ one has
         \begin{enumerate} 
                  
         \item\label{defn-diag-perp-3b}  $E_{k_1}aE_{k_2} = 0$, $k_2-k_1 \geq l$, $1\leq k_1, k_2 \leq q$,
         
          \item\label{defn-sizes} $qr > (l+1)p$,
         
         \item\label{defn-diag-perp-3c} for each $k=1, 2, ..., q$, there are positive elements $c_k, d_k$ with norm $1$  such that
         \begin{enumerate}
         
         \item\label{defn-D-in-zero}  $c_{k}, d_{k}   \in \overline{bAb},$
                  
         \item\label{defn-diag-perp-3e} 
         $c_{k} \perp s_k$ and $c_{k} \perp d_k,$ 
         
         \item\label{defn-diag-perp-3f}  
         $c_{k} E_k = c_{k}$ and  $d_{k} E_k = d_{k}, $ and 
         
          \item\label{defn-diag-perp-3g} 
          $s_k \precsim c_k$ and $1_A-h \precsim d_k,$
          where $h:= \phi(1_{pq})$.
         \end{enumerate}
  \end{enumerate}
  \end{defn}
  
Recall the following important notion:
\begin{defn}[Definition 3.1 of \cite{RorUHF}]
Let $A$ be a C*-algebra, then, define $$\mathrm{ZD}(A):=\{a\in A: d_1a = ad_2 =0\ \textrm{for some $d_1, d_2\in A^+\setminus\{0\}$}\}.$$
\end{defn}
\begin{defn}\label{defn-Prop-D}
The C*-algebra $A$ will be said to have Property (D) if for any $a\in \mathrm{ZD}(A)$ and any $\eps>0$, there are unitaries $u_1, u_2\in A$ and a $\mathcal D_0$-element $a'\in A$ such that $\norm{u_1au_2 - a'} < \eps.$
\end{defn}

It turns out that any $\mathcal D_0$-element is in the norm closure of the invertible elements. 
\begin{prop}\label{prop-ab-tsr1}
Let $a\in A$ be a $\mathcal D_0$-element of a unital C*-algebra $A$. Then $a\in \overline{\mathrm{GL}(A)}$.
\end{prop}

\begin{proof}
By definition, there is $b\in A^+\setminus\{0\}$ such that 
\begin{equation}\label{perp-red}
 ba = ab = 0;
\end{equation}
and there exists an order zero c.p.c.~map $$\phi':\mathrm{M}_{pq}(\Comp) \to A,$$ where $p, q \in\mathbb{N}$, and there exist $r, l\in\mathbb N$ such that, with $$e'_i:=\phi'(e_{i, i}),\quad i=1, 2, ..., pq,$$  
$$s_k:=e'_{(k-1)p+1} + \cdots + e'_{(k-1)p+r}, \quad k=1, ..., q,$$ and 
$$E_k:=e'_{(k-1)p+1} + \cdots + e'_{(k-1)p+p},\quad k=1, ..., q,$$ one has
         \begin{enumerate} 
                  
         \item\label{0-diag-perp-3b}  $E_{k_1}aE_{k_2} = 0$, $k_2-k_1 \geq l$, $1\leq k_1, k_2 \leq q$,
         
          \item\label{0-lr-control} $qr > (l+1)p$,
         
         \item\label{0-diag-perp-3c} there are positive elements $c_k, d_k$, $k=1, 2, ..., q$, with norm $1$  such that
         \begin{enumerate}
         
         \item\label{0-prop-cond-1a} $c_{k}, d_{k}   \in \overline{bAb},$
                  
         \item\label{0-diag-perp-3e} 
         $c_{k} \perp s_k$ and $c_{k} \perp d_k $,
         
         \item\label{0-diag-perp-3f}  
         $c_{k} E_k = c_{k}$ and $d_{k} E_k = d_{k}$, and 
         
          \item\label{0-diag-perp-3g} 
          $s_k \precsim c_k$ and $1_A-h' \precsim d_k,$
          where $h':= \phi'(1_{pq})$.
         \end{enumerate}
\end{enumerate}

As we shall show below, it follows that there exist an extendable order zero c.p.c.~map $$\phi: \mathrm{M}_{pq}(\Comp) \to A$$ 
and unitaries $u, w\in A$ such that, with
$$h:=\phi(1_{pq})\quad\mathrm{and}\quad e_{i} := \phi(e_{i, i}),\quad i=1, 2, ..., pq,$$
we have  
\begin{itemize}

\item $(1_A-h)(uw^*au^*) =0,$ 

\item $e_i(uw^*au^*)e_j = 0,$ $j-i>(l+1)q,$ and

\item  there are $e_{i_1}, ..., e_{i_{qr}} $ such that $$e_{i_j}(uw^*au^*) = (uw^*au^*) e_{i_j} = 0,\quad j=1, 2, ..., qr.$$

\end{itemize}

It then follows from Proposition \ref{trangular-invertible} (with $d=(l+1)p$ and $m=qr$) and Condition (\ref{0-lr-control}) above that $$uw^*au^*\in \overline{\mathrm{GL}(A)}.$$ Since $u, w$ are unitaries, one has $a\in \overline{\mathrm{GL}(A)}$,  and the proposition follows.

Let us verify the assertion. By Condition (\ref{0-prop-cond-1a}) and \eqref{perp-red} 
\begin{equation}\label{0-zero-C}
c_{k}a = ac_{k} = d_{k}a = ad_{k} = 0,\quad k=1, ..., q.
\end{equation}

Consider the positive element $f_\frac{1}{2}(h')$, and note that $1_A- f_{\frac{1}{2}}(h') \precsim 1_A-h'$. Then, by Condition (\ref{0-diag-perp-3g}), one has
$$1_A-f_{\frac{1}{2}}(h') \precsim d_{k},\quad k=1, 2, ..., q,$$ and therefore, by Proposition 2.4(iv) of \cite{RorUHF2} (see Lemma \ref{rordom-lem} above), there are $$v_1, v_2, ..., v_q \in A$$ such that
\begin{equation}\label{0-ladder-v}
 v_k^*v_k = f_{\frac{1}{2}}(1_A-f_{\frac{1}{2}}(h'))\quad\mathrm{and}\quad v_kv_k^* \in \mathrm{Her}(d_k)\subseteq \mathrm{Her}(b), \quad k=1, ..., q.
 \end{equation}
It follows from Condition (\ref{0-diag-perp-3f}) that $f_{\frac{1}{2}}(1_A-f_{\frac{1}{2}}(h')) \perp d_k;$ then, using Condition (\ref{0-diag-perp-3f}) again, one has $$f_{\frac{1}{2}}(1_A-f_{\frac{1}{2}}(h'))\perp v_kv^*_k\quad\mathrm{and} \quad   v_kv^*_k \in \mathrm{Her}(E_k),\quad 1\leq k\leq q,$$ and hence the elements 
$$v^*_1v_1, \ v_1v_1^*,\ v_2v_2^*, ..., v_qv_q^*$$ are mutually orthogonal.
Applying Lemma \ref{n-shifting} to $v_1, v_2, ..., v_q$, one obtains a unitary $w\in A$ with the properties of Lemma \ref{n-shifting}.

By Condition (\ref{0-diag-perp-3g}),
\begin{equation}\label{0-small-unit}
s_k \precsim c_{k},\quad k=1, ..., q.
\end{equation} 
Since $s_k,  c_k \in \overline{E_kAE_k}$, one may assume that the Cuntz sub-equivalences  \eqref{0-small-unit} hold in the hereditary sub-C*-algebra $\overline{E_kAE_k}$. By Proposition 2.4(iv) of \cite{RorUHF2} (see Lemma \ref{rordom-lem}), there is $z_k\in \overline{E_kAE_k}$ such that
\begin{equation}\label{0-unit-twist}
z_k^*z_k = f_\frac{1}{8}(s_k) \quad\mathrm{and} \quad z_kz_k^* \in \mathrm{Her}(c_k) \subseteq \mathrm{Her}(b).
\end{equation}
Note that, by Condition \eqref{0-diag-perp-3f}, 
\begin{equation}\label{location-zk}
z_k\in \overline{f_{\frac{1}{8}}(E_k)Af_{\frac{1}{8}}(E_k)}.
\end{equation}

By Condition (\ref{0-diag-perp-3e}),  
$$z_k^*z_k=f_{\frac{1}{8}}(s_k)  \perp \overline{c_{k}Ac_k} \ni z_kz_k^*.$$ 
Since $f_{\frac{1}{8}}(s_k)f_{\frac{1}{4}}(s_k)  = f_{\frac{1}{4}}(s_k)$, applying Lemma \ref{switch} to $v=z_k$, one has that, with 
\begin{equation}\label{const-uk}
u_k := \cos(\frac{\pi}{2}(z_kz_k^*+z_k^*z_k)) + z_k^*g(z_kz_k^*) - z_kg(z_k^*z_k),
\end{equation}
where 
$g(t)=\sin ({\pi t/2)}/{\sqrt{t}}$, $t\in(0, 1]$,
the element $u_k\in \mathrm{C}^*(z_k, 1)$ is a unitary such that 
$$u_k^*f_{\frac{1}{4}}(s_k)u_k \in \mathrm{Her}(c_k)\subseteq\mathrm{Her}(b).$$

By \eqref{location-zk}, one has 
 $u_k\in \overline{f_{\frac{1}{8}}(E_k)Af_{\frac{1}{8}}(E_k)} + \Comp 1_A$, and hence
$$u_k^* E_k u_k \in \overline{E_kAE_k},$$ 
$$u_k^* a u_k = a, \quad a\in\overline{E_{k'}AE_{k'}},\  k \neq k',$$
\begin{equation}\label{inv-1/16-0}
u_k^* f_{\frac{1}{8}}(E_k) u_k \in \overline{f_{\frac{1}{8}}(E_k)Af_{\frac{1}{8}}(E_k)}, \quad \mathrm{and}\quad [u_k, f_{\frac{1}{16}}(E_k)] = 0. 
\end{equation}
In particular, with $$u:=\prod_{k=1}^q u_k,$$ one has
 \begin{equation}\label{same-block}
 u^*E_ku \in\overline{E_kAE_k},\quad 1\leq k\leq q,
 \end{equation}
 \begin{equation}\label{inv-1/16}
 [u, f_{\frac{1}{16}}(E_k)] = 0,\quad 1\leq k\leq q,
 \end{equation}
 and
 \begin{equation}\label{0-prop-small-zero-unit}
 u^*f_{\frac{1}{4}}(s_k)u \in \mathrm{Her}(c_k)\subseteq\mathrm{Her}(b), \quad 1\leq k\leq q.
 \end{equation}

Consider the positive element $1_A-f_{\frac{1}{4}}(h')$, and note that $$(1_A-f_{\frac{1}{4}}(h')) v_1^*v_1 = (1_A-f_{\frac{1}{4}}(h')) f_{\frac{1}{2}}(1_A-f_{\frac{1}{2}}(h')) = 1_A-f_{\frac{1}{4}}(h').$$ It follows from Lemma \ref{n-shifting}(3) that
$$(1_A-f_{\frac{1}{4}}(h'))w^* \in \overline{A(v_1v_1^*)} \subseteq  \overline{A d_1}$$ and therefore, by \eqref{0-zero-C},
\begin{equation}\label{0-shrink-h}
(1_A- f_{\frac{1}{4}}(h'))w^*a = 0.
\end{equation}

Since $v_1^*v_1, E_k \in\mathrm{C}^*(e'_1, ..., e'_{pq})$, and this is a commutative algebra in view of Condition \eqref{0-diag-perp-3f} and \eqref{0-ladder-v}, one has
$$[v_1^*v_1, E_k] = 0\quad\mathrm{and}\quad (v_kv_k^*) E_k = v_kv_k^*,\quad k=1, 2, ..., q.$$ 
It then follows from Lemma \ref{n-shifting}(1) that
$$E_kw^* \in \overline{A(E_1 + E_k + E_{k+1})},\quad k=1, 2, ..., q-1,$$ and hence, by Condition (\ref{0-diag-perp-3b}), for any $k_2-k_1 \geq l+1$, where $1\leq k_1, k_2\leq q$, one has that 
$E_{k_1}w^* a E_{k_2} = 0,$
and then,  by \eqref{same-block},
\begin{equation*}
E_{k_1}(uw^*au^*) E_{k_2}  =  u(u^*E_{k_1} u)w^*a(u^* E_{k_2}  u)u^* \in u\overline{\overline{E_{k_1}AE_{k_1}}w^*a \overline{E_{k_2}AE_{k_2}} }u^* = \{0\}.
\end{equation*}
In particular, 
\begin{equation}\label{0-orth-shrink}
f_{\frac{1}{4}}(E_{k_1})(uw^*au^*) f_{\frac{1}{4}}(E_{k_2}) = 0,\quad k_2-k_1 \geq l+1.
\end{equation}


Consider the sum $$c:=\sum_{k=1}^q c_{k}.$$ 
Note that $c\in \overline{bAb}.$
By Condition (\ref{0-diag-perp-3f}),  one has that $ch' = c$; in particular, $[c, h'] = 0$, and hence, by \eqref{0-ladder-v}, $[c, v^*_1v_1] = 0.$ Also note
$ c\perp v_qv_q^* \in \mathrm{Her}(d_q).$   Then, it follows from Lemma \ref{n-shifting}(2) (with $c$ in the place of $d$ and $\overline{bAb}$ in the place of $D$) that 
\begin{equation}\label{0-zero-switch}
cw^* \in\overline{A(\overline{bAb})}.
\end{equation}
By \eqref{0-prop-small-zero-unit}, for each $k=1, 2, ..., q$ and $i=1,..., r$, one has 
\begin{equation}\label{0-zero-unit}
u^*f_{\frac{1}{4}}(e'_{(k-1)p+i})u \leq u^*f_{\frac{1}{4}}(s_k)u\in \mathrm{Her}(c_k) \subseteq \overline{cAc}\subseteq \overline{bAb}. 
\end{equation} 
From this, together with \eqref{0-zero-switch}, follows
$$ (u^*f_{\frac{1}{4}}(e'_{(k-1)p+i})u) w^*\in (\overline{cAc})w^*\subseteq \overline{cA{cw^*}}  \subseteq \overline{A(\overline{bAb})}.$$ Hence, 
$$f_{\frac{1}{4}}(e'_{n(k-1)+i})(uw^*au^*) = u((u^*f_{\frac{1}{4}}(e'_{(k-1)p+i})u)w^*)au^* \in u\overline{A(\overline{bAb})}au^*  =  \{0\};$$
on the other hand, by \eqref{0-zero-unit},
$$(uw^*au^*)f_\frac{1}{4}(e'_{(k-1)p+i}) = uw^*(a(u^*f_{\frac{1}{4}}(e'_{(k-1)p+i})u))u^* \in uw^*(a(\overline{bAb}))u^* = \{0\}.$$
Thus, for any $k=1, ..., q$ and $i=1, ..., r$, 
\begin{equation}\label{0-zero-unit-shrink}
f_{\frac{1}{4}}(e'_{(k-1)p+i})uw^*au^* = 0
\quad\mathrm{and}\quad
uw^*au^*f_{\frac{1}{4}}(e'_{(k-1)p+i}) = 0.
\end{equation}

Let us show that also 
\begin{equation}\label{0-shink-h-u}
u^*(1_A- f_{\frac{1}{4}}(h'))uw^*a = 0.
\end{equation}
By \eqref{inv-1/16}, one has
\begin{eqnarray}\label{decomp-zero}
 && u^*(1_A- f_{\frac{1}{4}}(h'))u \\
& = &  1_A - u^*\sum_{k=1}^q\sum_{i=1}^p f_{\frac{1}{4}}(e'_{(k-1)p+i}) u  \nonumber \\
& = & (1_A - f_{\frac{1}{16}}(h')) + f_{\frac{1}{16}}(h') - u^* \sum_{k=1}^q\sum_{i=1}^p  f_{\frac{1}{4}}(e'_{(k-1)p+i}) u \nonumber \\
& = & (1_A - f_{\frac{1}{16}}(h')) + \sum_{k=1}^q f_{\frac{1}{16}}(E_k) -   u^* \sum_{k=1}^q \sum_{i=1}^p f_{\frac{1}{4}}(e'_{(k-1)p+i}) u \nonumber\\
& = & (1 - f_{\frac{1}{16}}(h')) + u^*(\sum_{k=1}^q f_{\frac{1}{16}}(E_k) -   \sum_{k=1}^q \sum_{i=1}^p f_{\frac{1}{4}}(e'_{(k-1)p+i})) u \nonumber\\
& = & (1_A - f_{\frac{1}{16}}(h')) + \sum_{k=1}^q \sum_{i=1}^p u_k^* (f_{\frac{1}{16}}(e'_{(k-1)p+i})-f_{\frac{1}{4}}(e'_{(k-1)p+i})) u_k \nonumber\\
& = & (1_A - f_{\frac{1}{16}}(h')) + \sum_{k=1}^q \sum_{i=1}^p u_k^* \lambda_{k, i} u_k, \nonumber
\end{eqnarray}
where $$\lambda_{k, i} := f_{\frac{1}{16}}(e'_{(k-1)p+i})-f_{\frac{1}{4}}(e'_{(k-1)p+i}), \quad k=1, ..., q,\ i=1, ..., p.$$

Consider the elements
\begin{equation*}
(u_k^* \lambda_{k, i} u_k) w^*a,\quad k=1, ..., q,\ i=1, ..., p.
\end{equation*}
If $i= r+1, ..., p$, then, by \eqref{0-unit-twist},
$$z_k \lambda_{k, i} = \lambda_{k, i} z_k =0;$$ hence, by \eqref{const-uk},
$$[u_k, \lambda_{k, i}] = 0.$$ 
Since $\lambda_{k, i} \subseteq \mathrm{Her}(1-f_{\frac{1}{4}}(h'))$, together with \eqref{0-shrink-h}, one has
$$(u_k^* \lambda_{k, i} u_k) w^*a =  \lambda_{k, i} w^*a = 0.$$
If $i=1, ..., r$, then, by \eqref{const-uk},
\begin{eqnarray*}
\lambda_{k, i} u_k w^*a & = & \lambda_{k, i} \cos(\frac{\pi}{2}(z_kz_k^*+z_k^*z_k))w^*a + \lambda_{k, i} z_k^*g(z_kz_k^*) w^*a -  \lambda_{k, i} z_kg(z_k^*z_k) w^*a.
\end{eqnarray*}
By \eqref{0-unit-twist} and \eqref{0-zero-switch},
$$ \lambda_{k, i} z_k^*g(z_kz_k^*) w^*a  \in   \lambda_{k, i} z_k^*(\overline{cAc}) w^*a = \{0\}.$$
Using \eqref{0-zero-switch} again, one has $\lambda_{k, i} z_k = 0$, and hence $$ \lambda_{k, i} z_kg(z_k^*z_k) w^*a =  0. $$
Since $[\lambda_{k, i}, z_k^*z_k] = 0$ (by \eqref{0-zero-switch}), one has
$$\lambda_{k, i} (\frac{\pi}{2}z_k^*z_k)^{2n} \in \mathrm{Her}(\lambda_{k, i})\subseteq \mathrm{Her}(1-f_{\frac{1}{4}}(h')),\quad n=0, 1, ..., $$ and therefore, by \eqref{0-shrink-h},
\begin{eqnarray*}
\lambda_{k, i} \cos(\frac{\pi}{2}(z_kz_k^*+z_k^*z_k))w^*a & = &  \lambda_{k, i} \sum_{n=0}^\infty\frac{1}{(2n)!}(\frac{\pi}{2}(z_kz_k^*+z_k^*z_k))^{2n}w^*a\\
& = &\sum_{n=0}^\infty\frac{1}{(2n)!} \lambda_{k, i} (\frac{\pi}{2}z_k^*z_k)^{2n}w^*a \\
& = & 0. 
\end{eqnarray*}
This shows that
$$\lambda_{k, i} u_k w^*a = 0,\quad k=1, ...,q,\  i=1, 2, ..., r,$$
and hence 
\begin{equation}\label{tent-zero}
(u_k^* \lambda_{k, i} u_k) w^*a = 0,\quad k=1, ..., q,\ i=1, ..., p.
\end{equation}
This, together with \eqref{decomp-zero} and \eqref{0-shrink-h}, implies
\begin{eqnarray*}
u^*(1_A - f_{\frac{1}{4}}(h'))u w^*a = (1_A - f_{\frac{1}{16}}(h'))w^*a + \sum_{k=1}^q \sum_{i=1}^p u_k^* \lambda_{k, i} u_kw^*a = 0,
\end{eqnarray*}
which is the desired equation \eqref{0-shink-h-u}.


Hence, with $$\phi := f_{\frac{1}{4}}(\phi')\quad\mathrm{and}\quad h:=\phi(1_{pq}),$$ 
by \eqref{0-shink-h-u}, one has
$$(1_A-h)(uw^*au^*) = 0.$$

Set $$ \phi(e_{i, i}) = e_i,\quad i=1, 2, ..., pq.$$
For any $e_i, e_j$, $j - i >(l+1)p$, there are $1\leq k_1, k_2\leq q$ with $k_2-k_1 \geq l+1$ such that $$e_i \leq f_{\frac{1}{4}}(E_{k_1})\quad \mathrm{and}\quad e_j\leq f_{\frac{1}{4}}(E_{k_2}).$$ Then it follows from \eqref{0-orth-shrink} that $$e_i(uw^*au^*)e_j \in \overline{f_{\frac{1}{4}}(E_{k_1})Af_{\frac{1}{4}}(E_{k_1})} uw^*au^* \overline{f_{\frac{1}{4}}(E_{k_2})A f_{\frac{1}{4}}(E_{k_2})} = \{0\}.$$

Set $$ \{(k-1)p+ i: k=1, ..., q,\  i=1,...,r\} = \{i_1, i_2, ..., i_{rq}\}.$$ Then it follows from \eqref{0-zero-unit-shrink} that
$$e_{i_j}(u(w^*a)u^*) = (u(w^*a)u^*) e_{i_j} = 0,\quad j=1, 2, ..., rq,$$
as desired. Finally, it is clear that $\phi$ is extendable. This proves that $\phi$ verifies the assertion.
\end{proof}

\begin{thm}\label{thm-ab-tsr1}
Let $A$ be a unital C*-algebra with Property (D). Then $$\mathrm{ZD}(A) \subseteq \overline{\mathrm{GL}(A)}.$$ 
If, moreover, $A$ is finite, then $A = \overline{\mathrm{GL}(A)}$ (in other words, $\mathrm{tsr}(A) = 1$).
\end{thm}

\begin{proof}
Let $a\in \mathrm{ZD}(A)$, and fix an arbitrary $\eps>0$ for the time being. Since $A$ has the Property (D), there exist unitaries $u_1, u_2\in A$ and a $\mathcal D_0$-element $a'\in A$ such that
$$\norm{u_1au_2 - a'} < \eps.$$
By Proposition \ref{prop-ab-tsr1}, 
$a'\in\overline{\mathrm{GL(A)}}.$
Since $u_1$, $u_2$ are unitaries, it follows that $$u_1^*a'u_2^* \in \overline{\mathrm{GL}(A)},$$ and hence  $$\mathrm{dist}(a, \overline{\mathrm{GL}(A)})< \eps.$$ Since $\eps>0$ is arbitrary, one has that
$a \in  \overline{\mathrm{GL}(A)},$ and therefore
\begin{equation}\label{ZD-contain}
\mathrm{ZD}(A) \subseteq \overline{\mathrm{GL}(A)}.
\end{equation}

If, moreover, the C*-algebra $A$ is finite, then by Proposition 3.2 of \cite{RorUHF}, $$A\setminus\mathrm{GL}(A) \subseteq \overline{\mathrm{ZD}(A)};$$
with this, together with \eqref{ZD-contain}, we have 
$$A = \mathrm{GL}(A) \cup (A\setminus\mathrm{GL}(A)) \subseteq \mathrm{GL}(A) \cup \overline{\mathrm{ZD}(A)}  \subseteq \mathrm{GL}(A) \cup \overline{\mathrm{GL}(A)} = \overline{\mathrm{GL}(A)}.$$
\end{proof}

\section{Non-invertible elements and zero divisors of $\mathrm{C}(X)\rtimes\Gamma$}\label{D-1}
In the following two sections, we shall show that the C*-algebra $\mathrm{C}(X)\rtimes\Gamma$ has Property (D) if $(X, \Gamma)$ has the (URP) and (COS). Since $\mathrm{C}(X)\rtimes\Gamma$ is finite, it follows that $\mathrm{C}(X)\rtimes\Gamma$ has stable rank one by Theorem \ref{thm-ab-tsr1}.

Recall that if $B$ is a sub-C*-algebra of $A$, a conditional expectation from $A$ to $B$ is a completely positive linear contraction $\mathbb E: A \to B$ such that $$\mathbb E(b) = b,\quad\mathbb E(ba) = b\mathbb E(a),\quad \mathrm{and}\quad \mathbb E(ab) = \mathbb E(a) b,\quad b\in B,\ a\in A.$$ (By \cite{Tomoyama-proj}, the last two equations and the complete positivity are redundant.)

If $(X, \Gamma)$ is a free dynamical system, and if $\mathbb E: \mathrm{C}(X)\rtimes \Gamma \to\mathrm{C}(X)$ is a conditional expectation, where $\mathrm{C}(X)\rtimes \Gamma$ is a crossed-product C*-algebra, then $$\mathbb E(u_\gamma) = 0,\quad \gamma\in\Gamma\setminus\{e\}.$$ Indeed, let $\gamma\in\Gamma\setminus\{e\}$ and consider $\mathbb E(u_\gamma)\in\mathrm{C}(X)$. Note that for all $g\in\mathrm{C}(X)$, since $u_\gamma^*gu_\gamma \in\mathrm{C}(X)$, $$g\mathbb E(u_\gamma) = \mathbb{E}(gu_\gamma)  = \mathbb{E}(u_\gamma u_\gamma^*gu_\gamma) = \mathbb{E}(u_\gamma)(u_\gamma^*gu_\gamma).$$ Assume that $\mathbb E(u_\gamma) \neq 0$. Then there is $x_0\in X$ such that $\mathbb E(u_\gamma)(x_0) \neq 0$. Since $(X, \Gamma)$ is free, one has $x_0 \neq x_0\gamma^{-1}$. Pick $g\in \mathrm{C}(X)$ such that $g(x_0) = 1$ and $g(x_0\gamma^{-1}) = 0$. Then
$$g(x_0)\mathbb E(u_\gamma)(x_0) = \mathbb E(u_\gamma)(x_0) \neq 0 = \mathbb E(u_\gamma)(x_0) g(x_0\gamma^{-1}) = \mathbb E(u_\gamma)(x_0) (u_\gamma^* gu_\gamma)(x_0),$$ which is a contradiction.

If $\Gamma$ is amenable, then a conditional expectation $\mathbb E: \mathrm{C}(X)\rtimes\Gamma \to \mathrm{C}(X)$ always exists, and is not only unique (see above) but also faithful (see, for instance, Proposition 4.1.9 of \cite{BO-book}). 

\begin{lem}\label{pre-alg-red}
Let $(X, \Gamma)$ be a free and minimal topological dynamical system, where $\Gamma$ is a countable discrete group and $X$ is a compact Hausdorff space. Denote by $A=\mathrm{C}(X)\rtimes\Gamma$ the crossed product C*-algebra, and assume that there is a faithful conditional expectation $\mathbb{E}: A \to \mathrm{C}(X)$.

Let $a\in A$ such that $ba=0$ for some non-zero positive element $b$. Then, for any $\eps>0$, there is unitary $u \in A$ and a (non-empty) open set $E\subseteq X$ such that $$\norm{\varphi_Eua} < \eps.$$ 
\end{lem}

\begin{proof}
Since $\mathbb E$ is faithful, without loss of generality, one may assume $$\norm{a} = 1\quad\mathrm{and}\quad \norm{\mathbb{E}(b)}=1.$$

Pick $\eps'>0$ such that if $\norm{ca} < \eps'$ for some positive element $c$ with $\norm{c} \leq (\norm{b} +1)^2$, then $\norm{c^\frac{1}{2}a} < {\eps}/(\norm{b} + 1).$
Pick $\eps''\in (0, 1)$ such that
$$ (\norm{b}+\eps'')\eps'' < \eps',$$
and pick $b'\in \mathrm{C}_{\mathrm c}(\Gamma, \mathrm{C}(X))$ such that $$\norm{b-b'} < \eps''.$$
Since $\norm{\mathbb{E}(b)} = 1$, one may assume that $ \norm{\mathbb{E}(b')} = 1.$

Write $$b' = \sum_{\gamma\in\Gamma_0}f_\gamma u_\gamma$$ for a finite set $\Gamma_0\subseteq\Gamma$ with $\Gamma_0 = \Gamma_0^{-1}$, where $f_\gamma\in\mathrm{C}(X)$. Since $\mathbb{E}(b') = f_e$, one has that 
$\norm{f_e} =1,$ and then there is $x_0\in X$ such that 
\begin{equation*}
\abs{f_e(x_0)} =1.
\end{equation*} 
Pick a neighbourhood $U$ of $x_0$ such that
$
\overline{\bigcup_{\gamma\in\Gamma_0} U\gamma} \neq X,
$
and pick an open set $W\neq X$ such that $$\overline{\bigcup_{\gamma\in\Gamma_0} U\gamma} \subseteq W.$$ Therefore, there is a continuous function $\varphi_W: X \to [0, 1]$ such that
\begin{equation}\label{W-function}
\varphi_W^{-1}((0, 1]) = W\quad\mathrm{and}\quad \bigcup_{\gamma \in\Gamma_0} U\gamma  \subseteq \varphi_W^{-1}(1).
\end{equation}
Pick a continuous function $\varphi_U: X \to [0, 1]$ so that 
\begin{equation}\label{x0-1}
\varphi_U^{-1}((0, 1]) = U\quad\mathrm{and}\quad \varphi_U(x_0) = 1.
\end{equation}
Note that
\begin{eqnarray*}
b'\varphi_U(b')^* & = & (\sum_{\gamma\in\Gamma_0}f_\gamma u_\gamma) \varphi_U (\sum_{\gamma\in\Gamma_0}f_\gamma u_\gamma)^* \\
& = & (\sum_{\gamma\in\Gamma_0}f_\gamma u_\gamma) \varphi_U (\sum_{\gamma\in\Gamma_0} u^*_\gamma \overline{f_\gamma}) \\
& = & \sum_{\gamma, \gamma'\in\Gamma_0} f_{\gamma'}u_{\gamma'}\varphi_U u^*_\gamma \overline{f_\gamma} \\
& = & \sum_{\gamma, \gamma'\in\Gamma_0} f_{\gamma'}(\overline{f_\gamma}\circ(\gamma'\gamma^{-1})) (\varphi_{U}\circ\gamma') u_{\gamma'\gamma^{-1}}.
\end{eqnarray*}
Hence by \eqref{W-function}, 
\begin{equation}\label{W-as-1}
\varphi_W(b'\varphi_U(b')^*) = b'\varphi_U(b')^*.
\end{equation}
Also note that, by \eqref{x0-1},
\begin{eqnarray*}
\mathbb{E}(b'\varphi_U(b')^*)(x_0) & = & \sum_{\gamma\in\Gamma_0}\abs{f_\gamma(x_0)}^2\varphi_{U}(x_0\gamma) \geq \abs{f_e(x_0)}^2 = 1;
\end{eqnarray*}
and in particular,
\begin{equation}\label{ub-EP}
\norm{\mathbb{E}(b'\varphi_U(b')^*)} \geq 1 .
\end{equation}

Set $$\frac{1}{\norm{\mathbb E(b'\varphi_U(b')^*)}}b'\varphi_U (b')^* = b''.$$ 
Note that
$$\mathbb E(b'') = \frac{1}{\norm{\sum_{\gamma\in\Gamma_0}\abs{f_\gamma}^2\varphi_U}} \sum_{\gamma\in\Gamma_0}\abs{f_\gamma}^2 \gamma(\varphi_{U}) .$$ 
It follows that there is $y_0\in X$ such that $\mathbb E(b'')(y_0) = 1$. Perturbing $f_\gamma$, $\gamma\in\Gamma_0$, and $\varphi_U$ to be locally constant around $y_0$, we may suppose there is an open neighbourhood $V\ni y_0$ such that 
\begin{equation}\label{1-in-V}
\mathbb E(b'')(x) = \mathbb E(b'')(y_0) = 1,\quad x\in V.
\end{equation}
Moreover, since $(X, \Gamma)$ is free, one may choose $V$ small enough that
\begin{equation}\label{wandering-away}
V \cap V\gamma = \varnothing,\quad \gamma\in\Gamma_0^2\setminus\{e\},
\end{equation}
and since the action is minimal (so any orbit is dense), choosing $V$ even smaller, we may suppose there is $\gamma_0\in\Gamma$ such that
\begin{equation}\label{moving-far}
V\gamma_0 \cap W = \varnothing.
\end{equation}

Note that, by \eqref{ub-EP}, 
\begin{eqnarray*} 
\norm{b''a} & = & \norm{\frac{1}{\norm{\mathbb E(b'\varphi_U(b')^*)}}b'\varphi_U (b')^*a} \\
& = & \frac{1}{\norm{\mathbb E(b'\varphi_U(b')^*)}}\norm{b'\varphi_U (b')^*a} \\
& \approx_{(\norm{b}+\eps'')\eps''} & \frac{1}{\norm{\mathbb E(b'\varphi_U(b')^*)}}\norm{b'\varphi_U ba}=0.
\end{eqnarray*}
In other words,
\begin{equation}
\norm{b''a} \leq (\norm{b}+\eps'')\eps'' <\eps'.
\end{equation}
Since 
\begin{equation}\label{small-b''}
\norm{b''} \leq \norm{b'}^2 \leq (\norm{b}+1)^2,
\end{equation} 
by the choice of $\eps'$, one has
\begin{equation}\label{almost-perp}
\norm{(b'')^{\frac{1}{2}}a} < \frac{\eps}{\norm{b} + 1}.
\end{equation}

Now, choose a continuous function $h: X \to[0, 1]$ such that $h^{-1}((0, 1]) \subseteq V,$ and $h^{-1}(1)$ contains a neighbourhood of $y_0$. Note that $\norm{h} = 1$.

By \eqref{wandering-away}, 
$$h u_\gamma h = 0, \quad \gamma\in\Gamma_0^2\setminus\{e\},$$
and hence, writing $$b'' = \sum_{\gamma\in\Gamma_0^2} c_\gamma u_\gamma,$$  together with \eqref{1-in-V}, one has
\begin{equation}\label{cut-h}
hb''h = h (\sum_{\gamma\in\Gamma_0^2}c_\gamma u_\gamma)  h = \sum_{\gamma\in\Gamma_0^2}c_\gamma h u_\gamma h =c_e h^2 = \mathbb E(b'') h^2 = h^2.
\end{equation}

By \eqref{moving-far}, one has
$u_{\gamma_0} hu_{\gamma_0}^* \perp\varphi_W$ and hence, by \eqref{W-as-1},
\begin{equation}\label{perp-shift}
u_{\gamma_0} hu_{\gamma_0}^*  \perp b''.
\end{equation}

Consider 
$$v:=u_{\gamma_0} h (b'')^{\frac{1}{2}}.$$
Then, by \eqref{cut-h}, $$vv^* = u_{\gamma_0} h b''hu_{\gamma_0}^* = u_{\gamma_0} h^2 u_{\gamma_0}^*$$
and
$$v^*v =  (b'')^{\frac{1}{2}}  h^2 (b'')^{\frac{1}{2}}.$$

Pick an open set $E$ such that 
\begin{equation}\label{right-1}
\varphi_E vv^*= \varphi_E (u_{\gamma_0} h^2 u_{\gamma_0}^*) = \varphi_E.
\end{equation} 
(Such an $E$ exists because $h$ is constantly equal to $1$ in a neighbourhood of $y_0$.)

By \eqref{perp-shift}, $vv^*\perp v^*v$. By Lemma \ref{switch} and \eqref{right-1}, there is a unitary $u\in A$ such that
$$ (u^* \varphi_E u) (b'')^{\frac{1}{2}}  h^2 (b'')^{\frac{1}{2}}  = (u^* \varphi_E u) v^*v = u^* \varphi_E u. $$ 
Therefore, by \eqref{small-b''}, \eqref{almost-perp}, 
\begin{eqnarray*}
\norm{\varphi_Eua} & = & \norm{ u(u^*\varphi_Eu) a } = \norm{ u (u^* \varphi_E u) (b'')^{\frac{1}{2}}  h^2 (b'')^{\frac{1}{2}}  a } \\
& \leq & (\norm{b}+1)\norm{(b'')^{\frac{1}{2}}a}\\
& < & \eps,
\end{eqnarray*}
as desired.
\end{proof}

\begin{prop}\label{alg-red}
Let $(X, \Gamma)$ be a free and minimal topological dynamical system, where $\Gamma$ is a countable  discrete group and $X$ is a compact Hausdorff space. Consider the crossed product C*-algebra $A=\mathrm{C}(X)\rtimes\Gamma$. Assume that $A$ is finite and there is a faithful conditional expectation $\mathbb{E}: A \to \mathrm{C}(X)$.

Let $a\in A$ be a non-invertible element. Then, for any $\eps>0$, there exist $b\in \mathrm{C}_\mathrm{c}(\Gamma, \mathrm{C}(X))$, a (non-empty) open set $E\subseteq X$, and unitaries $u_1, u_2\in A$ such that 
$$\norm{u_1au_2-b} < \eps\quad
\textrm{and}\quad
\varphi_Eb = b\varphi_E = 0. $$
\end{prop}
\begin{proof}
Without loss of generality, one may assume that $\norm{a} = 1$. Let $\eps>0$ be given. Since $a$ is not invertible and $A$ is finite, by Proposition 3.2 of \cite{RorUHF}, there are $a'\in A$ and non-zero $b_1, b_2\in A^+$ such that 
$$\norm{a'} = 1,\quad \norm{a - a'} < \eps/5,\quad\mathrm{and}\quad b_1a' = 0 = a'b_2.$$ 

By Lemma \ref{pre-alg-red}, there are unitaries $u_1, u_2\in A$ and open sets $E', F' \subseteq X$ such that $$\norm{\varphi_{E'}u_1a'} < \eps/5 \quad\mathrm{and}\quad \norm{a'u_2\varphi_{F'}} < \eps/5.$$ Since $(X, \Gamma)$ is minimal (i.e., any orbit is dense), by passing to smaller open sets and changing the unitary $u_2$, one may assume that $E'=F'$.

Pick $a''\in \mathrm{C}_\mathrm{c}(\Gamma, \mathrm{C}(X))$ such that
$$\norm{u_1a'u_2 -a''} <\eps/5,$$ and note that
\begin{eqnarray*}
u_1au_2 & \approx_{\eps/5} & u_1a'u_2 \\
 & = & \varphi_{E'}u_1a'u_2\varphi_{E'} + (1-\varphi_{E'})u_1a'u_2\varphi_{E'} + \\
  &&  \varphi_{E'}u_1a'u_2(1-\varphi_{E'}) + (1-\varphi_{E'})u_1a'u_2(1-\varphi_{E'}) \\
  &\approx_{3\eps/5}& (1-\varphi_{E'})u_1a'u_2(1-\varphi_{E'}) \\
  &\approx_{\eps/5}& (1-\varphi_{E'})a''(1-\varphi_{E'}).
\end{eqnarray*} 
Pick an open set $E\subseteq \varphi^{-1}_{E'}(1)$ (so that $\varphi_E \varphi_{E'} = \varphi_E$), and define $$b= (1-\varphi_{E'})a''(1-\varphi_{E'}).$$
Then it is clear that 
$$\norm{u_1au_2 - b} <\eps\quad\mathrm{and}\quad \varphi_E b = b \varphi_E = 0,$$
as desired.
\end{proof}

%
%
%

\section{Stable rank of $\mathrm{C}(X)\rtimes\Gamma$}\label{D-2}

In this section, assuming $(X, \Gamma)$ has the (URP) and (COS), let us show that the element $b$ obtained in Proposition \ref{alg-red} is a $\mathcal D_0$-element (Proposition \ref{main-prop}). It follows that the C*-algebra $\mathrm{C}(X) \rtimes\Gamma$ has Property (D), and so has stable rank one by Theorem \ref{thm-ab-tsr1}.

Let $\Gamma$ be a discrete amenable group, and let $\Gamma_1$, $\Gamma_2$, ..., $\Gamma_T$ be finite subsets of $\Gamma$. Recall that $\Gamma$ is said to be tiled by $\Gamma_1$, $\Gamma_2$, ..., $\Gamma_T$ if there are group elements $$\gamma_{i, n},\quad n=1, 2, ...,\ i=1, ..., T,$$ such that 
$$\Gamma = \bigsqcup_{i=1}^T\bigsqcup_{n=1}^\infty \gamma_{i, n} \Gamma_i.$$ Note that if $\Gamma_i$ is (right) $(\mathcal F, \eps)$-invariant, then its (left) translation $\gamma_{i, n}\Gamma_i$ is also (right) $(\mathcal F, \eps)$-invariant.

\begin{lem}\label{div-lem-1}
Let $\Gamma$ be an infinite amenable group, and let $\Gamma_1, \Gamma_2, ..., \Gamma_T\subseteq \Gamma$ be finite sets which tile $\Gamma$. Let $\delta \in (0, 1]$, and let $n\in\mathbb N$. Then, there is a pair $(\mathcal F, \eps)$ such that if $F\subseteq\Gamma$ is $(\mathcal F, \eps)$-invariant, then there is $H\subseteq F$ such that $$\frac{\abs{H}}{\abs{F}} > 1-\delta,$$
and $H$ is tiled by $\Gamma_1, \Gamma_2, ..., \Gamma_T$ with multiplicities divisible by $n$.
\end{lem}

\begin{proof}
Set $\Gamma_1\cup\cdots\cup\Gamma_T = K$, and choose $\delta'>0$ such that
$$(1-\delta')(1-\frac{\delta'}{2}) > 1- \delta.$$
Choose $(\mathcal F, \eps)$ sufficiently large (with respect to the natural order relation), using amenablity, that if $F$ is $(\mathcal F, \eps)$-invariant, then
\begin{equation}\label{large-interior}
\frac{\abs{\mathrm{int}_K(F)}}{\abs{F}} > 1-\frac{\delta'}{2}.
\end{equation} 
Since $\Gamma$ is infinite, one may assume that $(\mathcal F, \eps)$ large enough  that if $F$ is $(\mathcal F, \eps)$-invariant, then 
\begin{equation}\label{very-far-F}
\abs{F} > \frac{2n(\abs{\Gamma_1} + \cdots + \abs{\Gamma_T})}{(2-\delta')\delta'}.
\end{equation} 
Then this $(\F, \eps)$ satisfies the requirement of the lemma. 

Indeed, let $F$ be an $(\mathcal F, \eps)$-invariant set. Since $\Gamma_1, ..., \Gamma_T$ tile $\Gamma$, by \eqref{large-interior}, there is a set $F'\subseteq F$ such that $F'$ can be tiled by  $\Gamma_1, ..., \Gamma_T$ (in fact, $F'$ may be chosen as the union of the tiles which intersect with $\mathrm{int}_KF$) and 
\begin{equation}\label{large-F'}
\frac{\abs{F'}}{\abs{F}} > 1-\frac{\delta'}{2}.
\end{equation}

Set $$(\bigsqcup_{i=1}^{m_1}\gamma_{1, i}\Gamma_1) \sqcup \cdots\sqcup (\bigsqcup_{i=1}^{m_T}\gamma_{T, i}\Gamma_T) = F',$$ where $\gamma_{i, j} \in\Gamma$ and $m_i$, $i=1, 2, ..., T$, are non-negative integers. Note that $$m_1\abs{\Gamma_1} + \cdots + m_T\abs{\Gamma_T} = \abs{F'}.$$

For each $m_i$, $i=1, 2, ..., T$, consider the remainder  $r_i$ when $m_i$ is divided by $n$. Then, set
$$(\bigsqcup_{i=1}^{m_1-r_1}\gamma_{1, m_1}\Gamma_1) \sqcup \cdots\sqcup (\bigsqcup_{i=1}^{m_T-r_T}\gamma_{T, i}\Gamma_T) = H.$$
It is clear that $H$ is tiled by $\Gamma_1, \Gamma_2, ..., \Gamma_T$ with multiplicities divisible by $n$. Moreover, by \eqref{large-F'} and \eqref{very-far-F},
\begin{eqnarray*}
1-\frac{\abs{H}}{\abs{F'}} & = & \frac{r_1\abs{\Gamma_1} + \cdots + r_T\abs{\Gamma_T}}{\abs{F'}} < \frac{n\abs{\Gamma_1} + \cdots + n\abs{\Gamma_T}}{\abs{F'}}\\ &< &  \frac{2n(\abs{\Gamma_1} + \cdots + \abs{\Gamma_T})}{(2-\delta')}\frac{1}{\abs{F}}< \delta',
\end{eqnarray*}
and hence, by \eqref{large-F'} again, 
$$\frac{\abs{H}}{\abs{F}} > (1-\delta')(1-\frac{\delta'}{2})> 1-\delta.$$
as desired.
\end{proof}

\begin{lem}\label{div-lem-2}
Let $\Gamma$ be an infinite amenable group, and let $\Gamma_1, \Gamma_2, ..., \Gamma_T\subseteq \Gamma$ be finite sets which tile $\Gamma$. Let $\delta\in (0, 1]$, let $n\in\mathbb N$, and let $K\subseteq \Gamma$ be a finite set. Then, there exists $(\mathcal F, \eps)$ such that if
$$F_1, F_2, ..., F_n$$
are mutually disjoint $(\mathcal F, \eps)$-invariant sets and $$\abs{F_1} = \abs{F_2} = \cdots = \abs{F_n},$$
then there are
$ H_1\subseteq F_1, ..., H_n\subseteq F_n$
such that
$$H_i K \subseteq F_i,\quad  
i=1, 2, ..., n,$$
each $H_i$ is tiled by $\Gamma_1, \Gamma_2, ..., \Gamma_T$ with multiplicities divisible by $n$,
$$\abs{H_1} = \abs{H_2} = \cdots = \abs{H_n},$$ and
 $$\frac{\abs{H_i}}{\abs{F_i}} > 1-\delta,\quad i=1, 2, ..., n.$$
\end{lem}

\begin{proof}
Without loss of generality, one may assume that $\delta$ is sufficiently small such that $$\frac{\delta}{1-\delta} < \frac{1}{T}.$$
Applying Lemma \ref{div-lem-1} to $$\frac{\delta}{2}\quad  \mathrm{and}\quad n\abs{\Gamma_1}\abs{\Gamma_2}\cdots\abs{\Gamma_T},$$ one obtains $(\mathcal F', \eps')$. Choose $(\mathcal F, \eps)$ such that if $F$ is $(\mathcal F, \eps)$-invariant, then $\mathrm{int}_KF$ is $(\mathcal F', \eps')$-invariant, and
\begin{equation}\label{off-shell-1}
\frac{\abs{\mathrm{int}_KF}}{\abs{F}} > 1-\frac{\delta}{2}.
\end{equation}
Then $(\mathcal F, \eps)$ satisfies the requirement of the lemma.

Indeed, let $F_1, F_2, ..., F_n$ be mutually disjoint $(\mathcal F, \eps)$-invariant sets with $$\abs{F_1} = \abs{F_2} = \cdots = \abs{F_n}.$$ Consider the sets $$\mathrm{int}_KF_1,\ \mathrm{int}_KF_2,\ ..., \mathrm{int}_KF_n.$$ Then each of them is $(\mathcal F', \eps')$-invariant. 
Also note that
$$(\mathrm{int}_KF_i) K \subseteq F_i,\quad (\mathrm{int}_KF_i) K \cap (\mathrm{int}_KF_j) = \varnothing,\quad  i, j=1, 2, ..., n,\  i\neq j,$$
Hence, by Lemma \ref{div-lem-1}, there are
$$F_1' \subseteq \mathrm{int}_KF_1,\  F_2'\subseteq \mathrm{int}_KF_2, ...,\  F'_n\subseteq \mathrm{int}_KF_n,$$
such that
\begin{equation}\label{off-shell-2}
\frac{\abs{F'_i}}{\abs{\mathrm{int}_KF_i}} > 1-\frac{\delta}{2},\quad i=1, 2, ..., n,
\end{equation}
and 
\begin{equation}\label{inn-decomp}
F_i' =(\bigsqcup_{j=1}^{m^{(i)}_1}\gamma^{(i)}_{1, j} \Gamma_1) \sqcup \cdots \sqcup (\bigsqcup_{j=1}^{m^{(i)}_T}\gamma^{(i)}_{T, j} \Gamma_T),\quad i=1, 2, ..., n,
\end{equation}
and each $m^{(i)}_t$, $i=1, 2, ..., n$, $t=1, 2, ..., T$, is divisible by $n\abs{\Gamma_1}\abs{\Gamma_2}\cdots\abs{\Gamma_T}$. 

It follows from \eqref{off-shell-1} and \eqref{off-shell-2} that for each $i=1, 2, ..., n$, 
\begin{eqnarray}\label{small-2nd-pert}
\abs{F_i} - \abs{F_i'} & = & (\abs{F_i} - \abs{\mathrm{int}_K(F_i)})  + (\abs{\mathrm{int}_K(F_i)} - \abs{F_i'}) \\
&<& \frac{\delta}{2}\abs{F_i} + \frac{\delta}{2}\abs{\mathrm{int}_KF_i} \nonumber \\
&\leq & \frac{\delta}{2}\abs{F_i} + \frac{\delta}{2}\abs{F_i} \nonumber \\
& = &  {\delta}\abs{F_i}. \nonumber
\end{eqnarray}

Set $$\min\{\abs{F_1'}, \abs{F_2'}, ..., \abs{F_n'}\} = D.$$ Since $\abs{F_1'}, \abs{F_2'}, ..., \abs{F_n'}$ are divided by $n\abs{\Gamma_1}\abs{\Gamma_2}\cdots\abs{\Gamma_T}$, there are non-negative integers $d_i$, $i=1, ..., n$, such that $$ \abs{F'_i} - D   = d_i\abs{\Gamma_1}\abs{\Gamma_2}\cdots\abs{\Gamma_T}n, \quad i=1, 2, ..., n.$$
By \eqref{small-2nd-pert} (note that $\abs{F_1} = \cdots =\abs{F_n}$),
\begin{equation}\label{final-off-shell}
\frac{D}{\abs{F_i}} > \frac{\abs{F_i} - \delta \abs{F_i}}{\abs{F_i}} =  1- \delta,\quad i=1, 2, ..., n,
\end{equation}
and so
$$ \abs{F'_i} - D \leq \abs{F_i} - D \leq \delta\abs{F_i},\quad i=1, 2, ..., n.$$

For each $i=1, 2, ..., n$, consider the set $$\{t_1, t_2, ..., t_S\} = \set{ t=1, 2, ..., T: m_{t}^{(i)} \neq 0}.$$ Then, there are integers $$0\leq c^{(i)}_{t_1}, ..., c^{(i)}_{t_S} \leq d_i\abs{\Gamma_1}\abs{\Gamma_2}\cdots\abs{\Gamma_T} = \frac{\abs{F_i'} - D}{n} \leq \frac{\delta\abs{F_i}}{n}\leq \frac{\delta\abs{F'_i}}{(1-\delta)n} $$
such that 
$$d_i\abs{\Gamma_1}\abs{\Gamma_2}\cdots\abs{\Gamma_T} = c^{(i)}_{t_1} \abs{\Gamma_{t_1}} + \cdots + c^{(i)}_{t_S} \abs{\Gamma_{t_S}}.$$ 
and
\begin{equation}\label{small-c}
c_{t_j}^{(i)} n \leq  m_{t_j}^{(i)},\quad j=1, ..., S.
\end{equation}
(Indeen, since 
$$m_{t_1}^{(i)} \abs{\Gamma_{t_1}} + \cdots m_{t_S}^{(i)} \abs{\Gamma_{t_S}} = \abs{F_i'}$$ and $S \leq T$, for at least one of $t_1, ..., t_S$, say, $t_1$, one has $$m_{t_1}^{(i)}\abs{\Gamma_{t_1}} \geq \frac{\abs{F'_i}}{T}. $$
Then, since 
$$d_i\abs{\Gamma_1}\abs{\Gamma_2}\cdots\abs{\Gamma_T} n < \frac{\delta\abs{F'_i}}{(1-\delta)} <\frac{\abs{F'_i}}{T} \leq m_{t_1}^{(i)} \abs{\Gamma_{t_1}},$$ the numbers 
$$c_{t_1}^{(i)} := d_i\abs{\Gamma_1}\abs{\Gamma_2}\cdots\abs{\Gamma_T}/\abs{\Gamma_{t_1}}\quad \mathrm{and} \quad c^{(i)}_{t_s} := 0,\quad s=2, ..., S,$$ have the desired property.)
For each $t\notin\set{t_1, ..., t_S}$, set $c_t^{(i)} = 0$. Then, one still has 
$$d_i\abs{\Gamma_1}\abs{\Gamma_2}\cdots\abs{\Gamma_T} = c^{(i)}_1 \abs{\Gamma_1} + \cdots + c^{(i)}_T \abs{\Gamma_T}$$
and
\begin{equation}\label{small-c}
c_{t}^{(i)} n \leq  m_{t}^{(i)},\quad t=1, 2, ..., T.
\end{equation}

Put
$$H_i =(\bigsqcup_{j=1}^{m^{(i)}_1-c^{(i)}_1n}\gamma^{(i)}_{1, j} \Gamma_1) \sqcup \cdots \sqcup (\bigsqcup_{j=1}^{m^{(i)}_T-c^{(i)}_Tn}\gamma^{(i)}_{T, j} \Gamma_T),\quad i=1, 2, ..., n.$$ 
(Note that, by \eqref{small-c}, $m_t^{(i)} - c_t^{(i)}n \geq 0$, $t=1, ..., T$.)
Since each $m_t^{(i)}$ is divisible by $n$, it is clear that each $H_i$ is tiled by $\Gamma_1, ..., \Gamma_T$ with multiplicities divisible by $n$.
Since
$$H_1 \subseteq \mathrm{int}_KF_1,\  H_2\subseteq \mathrm{int}_KF_2, ...,\  H_n\subseteq \mathrm{int}_KF_n,$$
one has
$$H_1K \subseteq F_1,\  H_2K\subseteq F_2, ...,\  H_nK\subseteq F_n.$$
 Also note that
$$\abs{H_1} = \abs{H_2} = \cdots = \abs{H_n} = D,$$
so that, by \eqref{final-off-shell},
$$\frac{\abs{H_i}}{\abs{F_i}} > 1-\delta,\quad i=1, 2, ..., n,$$
as desired.
\end{proof}

\begin{lem}\label{division-tower}
Let $\Gamma$ be an infinite amenable group, and let $(X, \Gamma)$ be a minimal dynamical system with the (URP). Let $\lambda>0$ be arbitrary, and let $O_{0, 1}, ..., O_{0, M},  O_{1, 1}, ..., O_{1, M} \subseteq X$ be mutually disjoint non-empty open sets and $$\{\kappa_{0,1}(=e), \kappa_{0, 2}, ..., \kappa_{0, M}, \kappa_{1,1}(=e), \kappa_{1, 2}, ..., \kappa_{1, M}\}\subseteq\Gamma$$ a finite family such that $$O_{i, m}=O_{i, 1}\kappa_{i, m},\quad i=0, 1, \ m=1, ..., M.$$ Put
$$\delta:=\min\{\mu(O_{i, m}): i=0, 1,\  m=1, ..., M,\  \mu\in\mathcal M_1(X, \Gamma)\},$$
and let $K\subseteq\Gamma$ be a symmetric finite set. (Since $(X, \Gamma)$ is minimal, $\delta \neq 0$.) Let $N>0$ be arbitrary.

Then, there are $n\in\mathbb N$ with $n>N$ and $(\mathcal F, \eps)$ 
such that if $(B, F)$ is a tower of $(X, \Gamma)$ with $F$ to be $(\mathcal F, \eps)$-invariant, then there is an order zero c.p.c.~map $$\phi: \mathrm{M}_{n^2}(\Comp) \to A,$$
where $A = \mathrm{C}(X)\rtimes\Gamma$,
such that if $$h:=\phi(1)\quad\textrm{and}\quad e_{i} := \phi(e_{i, i}),\quad i=1, 2, ..., n^2,$$
and 
$$b_k:= e_{n(k-1)+1} + \cdots + e_{n(k-1)+n},\quad k=1, 2, ..., n,$$
then 
$$e_i \in\mathrm{C}(X)\subseteq A$$
and
if $$E_i := e_i^{-1}((0, 1]),\quad i=1, 2, ..., n^2,$$
then
\begin{enumerate}

\item\label{property-1-lem} $$\bigsqcup_{i=1}^{n^2} E_i \subseteq \bigsqcup_{\gamma\in F} B\gamma,$$ and
         $$ \mu(\bigsqcup_{\gamma\in F} B\gamma \setminus \bigsqcup_{i=1}^{n^2} E_i) < \lambda \frac{\delta}{16n}\mu(\bigsqcup_{\gamma\in F} B\gamma),\quad \mu\in\mathcal M_1(X, \Gamma),$$

\item\label{property-2-lem} for each $k=1, 2, ..., n$, there are mutually disjoint open sets $O^{k}_{0, 1}, ..., O^{k}_{0, M}$ and $O^{k}_{1, 1}, ..., O^{k}_{1, M}$ such that

         \begin{enumerate}
         
         \item $O^k_{0, m} \subseteq O_{0, m}\cap \bigsqcup_{j=n(k-1)+4}^{nk} E_j$ and $O^k_{1, m} \subseteq O_{1, m}\cap \bigsqcup_{j=n(k-1)+1}^{nk} E_j$, $m=1, 2, ..., M$,
         
         \item $O^{k}_{i, m} = O^{k}_{i, 1}\kappa_{i, m}$, $i=0, 1$, $m=1, 2, ..., M$,
         
         \item $\mu(O^{k}_{0, 1}), \mu(O^{k}_{1, 1}) > \frac{\delta}{8n}\mu(\bigsqcup_{i=1}^{n^2}E_i)$, $\mu\in\mathcal M_1(X, \Gamma)$,
         
         \end{enumerate}

\item\label{property-4-lem} $b_{k_1} \perp u_\gamma b_{k_2} u^*_\gamma,\quad \gamma\in K,\ k_1\neq k_2,\ 1\leq k_1, k_2\leq n,$ where $u_\gamma\in A$ is the canonical unitary of $\gamma$. 

\end{enumerate}
\end{lem}

\begin{proof}

Choose a natural number $n>N$ such that 
\begin{equation}\label{large-n-0}
0 < \frac{1}{n-3} < \frac{\delta}{24},\quad \lambda\frac{\delta}{16n} < \frac{1}{2},\quad\mathrm{and}\quad \frac{3}{n} < \frac{\delta}{16}.
\end{equation}

Pick $(\mathcal F', \eps')$ such that if a finite set $\Gamma_0\subseteq \Gamma$ is $(\mathcal F', \eps')$-invariant, then
\begin{equation}\label{density-O}
\frac{1}{\abs{\Gamma_0}}\abs{\{\gamma\in\Gamma_0: x\gamma \in O_{i, m}\}} > \frac{\delta}{2},\quad x\in X,\ i=0, 1, \  m=0, 1, ..., M,
\end{equation}
and 
\begin{equation}\label{small-O-walk}
 \frac{\abs{\partial_{K_0^M}\Gamma_0}}{\abs{\Gamma_0}} < \frac{\delta}{16},
\end{equation}
where $$K_0:=\{\kappa_{0, 1}, \kappa_{0, 2}, ..., \kappa_{0, M}, \kappa_{1, 1}, \kappa_{1, 2}, ..., \kappa_{1, M}\}.$$

By Theorem 4.3 of \cite{DHZ-tiling}, there are  $(\mathcal F', \eps')$-invariant finite sets $$\Gamma_1,\ \Gamma_2,\ ...,\ \Gamma_T\subseteq\Gamma$$ which tile $\Gamma$. 
Applying Lemma \ref{div-lem-2} to $\lambda\delta/32n$, $n$, and $K$ with respect to the finite sets $\Gamma_1$,..., $\Gamma_T$, one obtains $(\mathcal F'', \eps'')$.

By Theorem 4.3 of \cite{DHZ-tiling} again, there are  $(\mathcal F'', \eps'')$-invariant finite sets $$\Gamma'_1,\ \Gamma'_2,\ ...,\ \Gamma'_{T'}\subseteq\Gamma$$ which tile $\Gamma$. 
Applying Lemma \ref{div-lem-1} to $\lambda\delta/32n$ and $n$ with respect to the finite sets $\Gamma'_1$, $\Gamma'_2$, ..., $\Gamma'_{T'}$, one obtains $(\mathcal F, \eps)$. Since $\Gamma$ is infinite, one may assume that $(\mathcal F, \eps)$ is sufficiently large that if $F$ is $(\mathcal F, \eps)$-invariant, then $\abs{F} > 2n^2$.

Then, $(\mathcal F, \eps)$ possesses the property of the lemma.

Indeed, let $(B, F)$ be a tower such that $F$ is $(\mathcal F, \eps)$-invariant. Then, by Lemma \ref{div-lem-1}, there is a finite set $R_1\subseteq F$ such that
\begin{equation}\label{small-R-1}
\frac{\abs{R_1}}{\abs{F}} < \lambda \frac{\delta}{32n}
\end{equation}
and $F\setminus R_1$ can be tiled by $\Gamma_1', ..., \Gamma_{T'}'$ with multiplicities divisible by $n$. Grouping the tilings appropriately, one has $$F\setminus{R_1} = \Gamma_1'' \sqcup \Gamma_2'' \sqcup \cdots\sqcup \Gamma_{n}'',$$ where $\Gamma_i''$, $i=1, ..., n$, are mutually disjoint $(\mathcal F'', \eps'')$-invariant sets and $$\abs{\Gamma_1''} = \abs{\Gamma_2''} = \cdots = \abs{\Gamma_n''}.$$

By Lemma \ref{div-lem-2} and the choice of $(\mathcal F'', \eps'')$, there are finite sets $\Gamma_i''' \subseteq \Gamma_i''$ such that 
$$\Gamma_i''' K \subseteq \Gamma_i'',\quad  i=1, 2, ..., n,$$
$$\frac{\abs{\Gamma'''_i}}{\abs{\Gamma_i''}} > 1 -\lambda\frac{\delta}{32n},\quad i=1, 2, ..., n,$$
$$\abs{\Gamma_1'''} = \abs{\Gamma_2'''} = \cdots = \abs{\Gamma_n'''},$$
and each $\Gamma_i'''$ is tiled by $\Gamma_1, ..., \Gamma_T$ with multiplicities divisible by $n$. Since $\Gamma_i''$, $i=1, ..., n$, are mutually disjoint, one has $$\Gamma_i'''K \cap \Gamma_j''' = \varnothing,\quad i, j = 1, 2, ..., n,\ i\neq j.$$

With $$R_2 = (F\setminus R_1)\setminus ( \Gamma_1''' \sqcup \Gamma_2''' \sqcup \cdots\sqcup \Gamma_{n}'''),$$
one has
$$
F\setminus(R_1\cup R_2) = \Gamma_1''' \sqcup \Gamma_2''' \sqcup \cdots\sqcup \Gamma_{n}''',
$$
and
\begin{equation}\label{small-R-2}
\frac{\abs{R_2}}{\abs{F}} \leq \lambda\frac{\delta}{32n}.
\end{equation}

Note that, by \eqref{large-n-0}, \eqref{small-R-1}, and \eqref{small-R-2}, one has
\begin{equation}\label{small-R-1-2}
2\abs{F\setminus(R_1\cup R_2)} > \abs{F}.
\end{equation}

Then, inside each $\Gamma_i'''$, since each such set is tiled by $\Gamma_1, ..., \Gamma_T$ (which are $(\mathcal F', \eps')$-invariant) with multiplicities divisible by $n$,  after regrouping, one has
$$\Gamma_i''' = \Gamma_{i, 1} \sqcup \cdots \sqcup \Gamma_{i, n},$$
where $\Gamma_{i, j}$ is $(\mathcal F', \eps')$-invariant with $$\abs{\Gamma_{i, 1}} = \abs{\Gamma_{i, 2}} = \cdots = \abs{\Gamma_{i, n}},\quad i=1, 2, ..., n.$$

In summary, one obtains the decomposition
$$F\setminus(R_1\cup R_2) = (\Gamma_{1, 1}\sqcup\cdots \sqcup \Gamma_{1, n}) \sqcup \cdots \sqcup (\Gamma_{n, 1}\sqcup\cdots \sqcup \Gamma_{n, n})$$
with the properties
\begin{enumerate}
\item \begin{equation}\label{equi-div}
\abs{\Gamma_{i_1, j_1}} = \abs{\Gamma_{i_2, j_2}},\quad 1\leq i_1, i_2, j_1, j_2\leq n,
\end{equation}
\item each $\Gamma_{i, j}$ is $(\mathcal F', \eps')$-invariant,
\item \begin{equation}\label{K-separated-level-1}
            (\Gamma_{i, 1}\sqcup\cdots \sqcup \Gamma_{i, n}) K \subseteq F,\quad i=1, 2, ..., n,\quad \textrm{and}
         \end{equation}   
          
\item if $i\neq j$, then 
         \begin{equation} \label{K-separated-level-2}
                    (\Gamma_{i, 1}\sqcup\cdots \sqcup \Gamma_{i, n})K \cap (\Gamma_{j, 1}\sqcup\cdots \sqcup \Gamma_{j, n}) = \varnothing. 
          \end{equation} 
\end{enumerate}

Set $\varphi_{B} = e$, and set $$ u^*_\gamma e u_\gamma = e_{\gamma},\quad \gamma\in F.$$
For each $1\leq i\leq n^2$, write $i=n(k-1)+j$, where $1\leq j\leq n$, and set
$$ \sum_{\gamma\in\Gamma_{k, j}} e_{\gamma} = e_i. $$
By \eqref{equi-div}, it follows from Lemma \ref{existence-0-map} that there is a order zero map
$$\phi: \mathrm{M}_{n^2}(\Comp) \to A$$
such that
$$\phi(e_{i, i}) = e_i,\quad i=1, 2, ..., n^2.$$

Recalling the notation $E_i:= e_{i}^{-1}((0, 1])$, note that, with $i=n(k-1)+j$ with $1\leq j\leq n$, one has $$E_i = \bigsqcup_{\gamma\in\Gamma_{k, j}}B\gamma.$$ It is clear that
\begin{equation}\label{smaller-E}
\bigsqcup_{i=1}^{n^2} E_i \subseteq \bigsqcup_{\gamma\in F}B\gamma.
\end{equation}
Hence by \eqref{small-R-1} and \eqref{small-R-2}, one has
$$ \mu(\bigsqcup_{\gamma\in F} B\gamma \setminus \bigsqcup_{k=1}^{n^2} E_k) = \mu(\bigsqcup_{\gamma\in R_1\cup R_2} B\gamma)   <\lambda \frac{\delta}{16n}\mu(\bigsqcup_{\gamma\in F} B\gamma),\quad \mu\in\mathcal M_1(X, \Gamma).$$ This proves Property (\ref{property-1-lem}).

Consider the sums $$b_k:=e_{n(k-1)+1} + \cdots + e_{nk},\quad k=1, ..., n.$$ Note that, with $$\Gamma_k:=\Gamma_{k, 1} \sqcup\cdots \sqcup \Gamma_{k, n},$$
one has 
$$b_k = \sum_{\gamma\in\Gamma_k} e_\gamma,$$
and hence, if $\Gamma_k\gamma \subseteq F$ for a group element $\gamma$, then 
$$u^*_\gamma b_k u_\gamma= \sum_{\gamma' \in\Gamma_k} u^*_\gamma e_{\gamma'} u_\gamma = \sum_{\gamma' \in\Gamma_k}  e_{\gamma'\gamma} = \sum_{\gamma'\in \Gamma_k \gamma}  e_{\gamma'}. $$

Thus, by \eqref{K-separated-level-1}, \eqref{K-separated-level-2} and the assumption that $K=K^{-1}$, one has that for any $k_1\neq k_2$, 
$$b_{k_1}\perp u_\gamma b_{k_2}u^*_\gamma,\quad \gamma\in K.$$
This proves Property (\ref{property-4-lem}).



Now consider the C*-algebra of the tower $(B, F)$, $$C := \mathrm{C}^*\{u_\gamma f: \gamma \in F, f\in\mathrm{C}_0(B)\}\subseteq A.$$  
Note that, by Lemma 3.12 of \cite{Niu-MD-Z}, 
$$ C \cong \mathrm{M}_{\abs{F}}(\mathrm{C}_0(B)),$$ 
and the isomorphism may be chosen such that, for any $g\in\mathrm{C}_0(\bigsqcup_{\gamma\in F} B\gamma)\subseteq \mathrm{C}(X)$, one has $g\in C$ and $$g \mapsto ( x\mapsto \sum_{\gamma\in F} g(x\gamma)e_{\gamma, \gamma}).$$
In particular, since for each $i=0, 1$, $k=1, 2, ..., n$ and $m=1, ..., M$, one has $b_k\varphi_{O_{i, m}} \in \mathrm{C}_0(\bigsqcup_{\gamma\in F} B\gamma)$, this implies $$b_k\varphi_{O_{i, m}} \in C.$$
Noting that $\Gamma_{i, j}$ are $(\F', \eps')$-invariant,  by \eqref{density-O} and \eqref{small-R-1-2}, regarding $b_k\varphi_{O_{i, m}}$ as an element of $ C \cong \mathrm{M}_{\abs{F}}(\mathrm{C}_0(B)),$ one has that for any $x\in B$,  
\begin{eqnarray*}
\mathrm{rank}(b_k\varphi_{O_{i, m}}(x)) & = & \abs{\set{\gamma\in F: (b_k\varphi_{O_{i, m}})(x\gamma) > 0}} \\
& = & \abs{\set{ \gamma\in \Gamma_{k, 1}\sqcup \Gamma_{k, 2} \sqcup\cdots\sqcup \Gamma_{k, n}: x\gamma \in O_{i, m}}}\\
& \geq &\frac{\delta}{2}(\abs{\Gamma_{k, 1}} + \cdots + \abs{\Gamma_{k, n}}) \\
& = & \frac{\delta}{2}n\abs{\Gamma_{k, 1}}  >  \frac{\delta}{4n}\abs{F}.
\end{eqnarray*}
Then, for any $\mu\in\mathcal M_1(X, \Gamma)$, by \eqref{smaller-E} for the last step,
\begin{eqnarray}\label{small-larger-set}
\mu(O_{i, m}\cap \bigsqcup_{j=n(k-1)+1}^{nk} E_j) & = &\mu (\set{x\in X: (b_k\varphi_{O_{i, m}})(x) > 0})\\
&=&\int_B \mathrm{rank}(b_k\varphi_{O_{i, m}}(x))\mathrm{d}\mu  \nonumber \\
& \geq & \int_B \frac{\delta}{4n}\abs{F}\mathrm{d}\mu \nonumber \\
& = & \frac{\delta}{4n}\abs{F}\mu(B) > \frac{\delta}{4n}\mu(\bigsqcup_{j=1}^{n^2}E_j). \nonumber
\end{eqnarray}

Now, for each $i=0, 1$, $k=1, 2, ..., n$,  let us construct open sets $O^{k}_{i, m}$, $m=1, 2, ..., M$. Note that (recall $\Gamma_k=\Gamma_{k, 1} \sqcup\cdots \sqcup \Gamma_{k, n})$ 
$$O_{i, m} \cap  \bigcup_{j=n(k-1)+1}^{nk} E_j = O_{i, m} \cap \bigsqcup_{\gamma\in \Gamma_{k}} B\gamma,\quad i=0, 1,\ m=1, 2, ..., M.$$
Consider the decomposition
$$\Gamma_k^\circ:= \Gamma_k\setminus (\Gamma_{k, 1}\cup\Gamma_{k, 2}\cup\Gamma_{k, 3})=\Gamma_{k, 4} \sqcup\Gamma_{k, 5}\sqcup \cdots \sqcup \Gamma_{k, n} ,$$ and define
$$O^k_{0, 1} := O_{0, 1} \cap\bigsqcup_{\gamma\in \mathrm{int}_{K_0^M}(\Gamma^\circ_{k})} B\gamma\quad
\mathrm{and}\quad O^k_{0, m}: = O^k_{0, 1} \kappa_{0, m},\quad m=1, 2, ..., M,$$ 
and
$$O^k_{1, 1} := O_{1, 1} \cap\bigsqcup_{\gamma\in \mathrm{int}_{K_0^M}(\Gamma_{k})} B\gamma\quad
\mathrm{and}\quad O^k_{1, m}: = O^k_{1, 1} \kappa_{1, m},\quad m=1, 2, ..., M.$$ 
Then it is clear that
$$O^{k}_{0, m} \subseteq O_{0, m}\cap \bigcup_{i=n(k-1)+4}^{nk} E_i \quad\mathrm{and}\quad O^{k}_{1, m} \subseteq O_{1, m}\cap \bigcup_{j=n(k-1)+1}^{nk} E_j,\quad m=1, 2, ..., M.$$

Since $\Gamma_{i, j}$ are $(\mathcal F', \eps')$-invariant, the sets $\Gamma^\circ_k$ are also $(\mathcal F', \eps')$-invariant. By \eqref{small-O-walk}, one has 
$$\frac{\abs{\partial_{K_0^M}(\Gamma_k^\circ)}}{\abs{\Gamma_k^\circ}} < \frac{\delta}{16},$$
and therefore, together with \eqref{large-n-0} and  \eqref{small-larger-set}, for any $\mu\in\mathcal M_1(X, \Gamma)$ and $i=0, 1$,

\begin{eqnarray*}
\mu(O^k_{i, 1}) & \geq & \mu(O_{i, 1} \cap\bigsqcup_{\gamma\in \mathrm{int}_{K_0^M}(\Gamma^\circ_{k})} B\gamma) \\
& \geq & \mu(O_{i, 1} \cap  \bigsqcup_{\gamma \in \Gamma^\circ_k} B\gamma) - \mu(\bigsqcup_{\gamma \in \partial_{K_0^M} (\Gamma^\circ_k)} B\gamma)  \\
& \geq & \mu(O_{i, 1} \cap  \bigsqcup_{\gamma \in \Gamma^\circ_k} B\gamma) -  \frac{\delta}{16}\abs{\Gamma^\circ_k} \mu(B) \\
& > & \mu(O_{i, 1} \cap  \bigsqcup_{j=n(k-1)+4}^{nk} E_j) -\frac{\delta}{16n} \mu(\bigsqcup_{j=1}^{n^2} E_j)\\
&\geq & \mu(O_{i, 1} \cap  \bigsqcup_{j=n(k-1)+1}^{nk} E_j)  - \frac{3}{n^2} \mu(\bigsqcup_{j=1}^{n^2}E_j) - \frac{\delta}{16n}\mu(\bigsqcup_{j=1}^{n^2} E_j) \\
&\geq & \frac{\delta}{4n}\mu(\bigsqcup_{j=1}^{n^2}E_j)   - \frac{\delta}{8n}\mu(\bigsqcup_{j=1}^{n^2} E_j) \\
& = &   \frac{\delta}{8n}\mu(\bigsqcup_{j=1}^{n^2} E_j).
\end{eqnarray*}
This proves Property (\ref{property-2-lem}), as desired.
\end{proof}

\begin{lem}\label{division-space}
Let $\Gamma$ be an infinite discrete amenable group, and  let $(X, \Gamma)$ be a minimal topological dynamical system with the (URP). Let $\lambda>0$ be arbitrary, and let $$O_{0, 1}, ..., O_{0, M},  O_{1, 1}, ..., O_{1, M} \subseteq X$$ be mutually disjoint non-empty open sets and
$$\{\kappa_{0,1}(=e), \kappa_{0, 2}, ..., \kappa_{0, M}, \kappa_{1,1}(=e), \kappa_{1, 2}, ..., \kappa_{1, M}\}\subseteq\Gamma$$ a finite family in $\Gamma$ such that $$O_{i, m}=O_{i, 1}\kappa_{i, m},\quad i=0, 1, \ m=1, ..., M.$$
Let $K\subseteq\Gamma$ be a symmetric finite set. Let $N>0$ be arbitrary.

Then there exist $n\in\mathbb N$ with $n>N$ and an order zero c.p.c.~map $$\phi: \mathrm{M}_{n^2}(\Comp) \to A,$$
where $A=\mathrm{C}(X)\rtimes\Gamma$, 
such that if $$h:=\phi(1)\quad\mathrm{and}\quad e_{i} := \phi(e_{i, i}),\quad i=1, 2, ..., n^2,$$
and 
$$b_k:= e_{n(k-1)+1} + \cdots + e_{n(k-1)+n},\quad k=1, 2, ..., n,$$
then 
$$e_i \in\mathrm{C}(X)\subseteq \mathrm{C}(X)\rtimes\Gamma,$$
and with $$E_i = e_{i}^{-1}((0, 1]),\quad i=1, 2, ..., n^2,$$
one has
\begin{enumerate}

\item\label{property-1-cob} for each $k=1, 2, ..., n$, there are mutually disjoint open sets $O^{k}_{0, 1}, ..., O^{k}_{0, M}$ and $O^{k}_{1, 1}, ..., O^{k}_{1, M}$  such that
         \begin{enumerate}
         \item $O^k_{0, m} \subseteq O_{0, m}\cap \bigsqcup_{i=4}^{n} E_{n(k-1)+i}$ and $O^k_{1, m} \subseteq O_{1, m}\cap \bigsqcup_{i=1}^{n} E_{n(k-1)+i}$, $m=1, 2, ..., M$,
         \item $O^k_{i, m} = O^k_{i, 1}\kappa_{i, m}$, $i=0, 1$, $m=1, 2, ..., M$,
         \item $$ \lambda \mu(O^k_{0, 1}) > \frac{3}{n^2}\quad\mathrm{and}\quad \lambda \mu(O^k_{1, 1}) > \mu(X\setminus \bigsqcup_{i=1}^{n^2} E_i),\quad \mu\in\mathcal M_1(X, \Gamma),$$
         \end{enumerate}

\item\label{property-3-cob} $$b_{k_1} \perp u_\gamma b_{k_2} u^*_\gamma,\quad \gamma\in K,\ k_1 \neq k_2,\ 1\leq k_1, k_2\leq n,$$ where $u_\gamma\in A$ is the canonical unitary of $\gamma$.

\end{enumerate}
\end{lem}

\begin{proof}
Applying Lemma \ref{division-tower} with respect to $O_{0, 1}, O_{0, 2}, ..., O_{0, M}$ and $O_{1, 1}, O_{1, 2}, ..., O_{1, M}$, and 
$$\delta:=\min\{\mu(O_{i, m}): i=0, 1, m=1, ..., M,\  \mu\in\mathcal M_1(X, \Gamma)\}>0,$$
one obtains $(\mathcal F', \eps')$ and $n$. Since $n$ can be chosen arbitrarily large, we may assume that $n>N$ and 
\begin{equation}\label{large-n}
\frac{3}{n^2} < \lambda\frac{3\delta}{32n} < \frac{1}{4}.
\end{equation}

Since $(X, \Gamma)$ is assumed to have the (URP), there exist open towers $$(B_1, F_1), ..., (B_S, F_S)$$ such that each $F_s$, $s=1, ..., S$, is $(\mathcal F', \eps')$-invariant and
\begin{equation}\label{small-left-global}
\mu(X\setminus\bigsqcup_{s=1}^S\bigsqcup_{\gamma\in F_s} B_s\gamma) < \lambda \frac{\delta}{32n},\quad \mu\in\mathcal M_1(X, \Gamma).
\end{equation}

For each tower $(B_s, F_s)$, since $F_s$ is $(\mathcal F', \eps')$-invariant, by Lemma \ref{division-tower}, there is an order zero c.p.c.~map $$\phi_s: \mathrm{M}_{n^2}(\Comp) \to A,$$
where $A = \mathrm{C}(X)\rtimes\Gamma$, 
such that if $$h_s:=\phi_s(1)\quad\mathrm{and}\quad e^{(s)}_{i} := \phi_s(e_{i, i}),\quad i=1, 2, ..., n^2,$$
and 
$$b_{s, k}:= e^{(s)}_{n(k-1)+1} + \cdots + e^{(s)}_{n(k-1)+n},\quad k=1, 2, ..., n,$$
then 
$$e^{(s)}_{i} \in\mathrm{C}(X)$$
and
if denote by $$E_{s, i} = (e^{(s)}_{i})^{-1}((0, 1]),\quad i=1, 2, ..., n^2,$$
then
\begin{enumerate}

\item\label{property-1-lem-comb} $$\bigsqcup_{i=1}^{n^2} E_{s, i} \subseteq \bigsqcup_{\gamma\in F_s} B_s\gamma,$$ and
         \begin{equation}\label{small-left-ind} 
         \mu(\bigsqcup_{\gamma\in F_s} B_s\gamma \setminus \bigsqcup_{i=1}^{n^2} E_{s, i}) < \lambda \frac{\delta}{16n}\mu(\bigsqcup_{\gamma\in F_s} B_s\gamma),\quad \mu\in\mathcal M_1(X, \Gamma).
         \end{equation}

\item\label{property-2-lem-comb} for each $k=1, 2, ..., n$, there are open sets $O^{k, s}_{0, 1}, ..., O^{k, s}_{0, M}$ and $O^{k, s}_{1, 1}, ..., O^{k, s}_{1, M}$ such that

         \begin{enumerate}
         
         \item\label{property-2-1-lem-comb} $O^{k, s}_{0, m} \subseteq O_{0, m}\cap \bigsqcup_{j=n(k-1)+4}^{nk} E_{s, j}$ and $O^{k, s}_{1, m} \subseteq O_{1, m}\cap \bigsqcup_{j=n(k-1)+1}^{nk} E_{s, j}$, $m=1, 2, ..., M$,
         
         \item\label{property-2-2-lem-comb} $O^{k, s}_{i, m} = O^{k, s}_{i, 1}\kappa_{i, m}$, $i=0, 1$, $m=1, 2, ..., M$,
         
         \item 
         \begin{equation}\label{dense-O-1-ind}
         \mu(O^{k, s}_{0, 1}),\  \mu(O^{k, s}_{1, 1}) > \frac{\delta}{8n}\mu(\bigsqcup_{i=1}^{n^2}E_{s, i}),\quad \mu\in\mathcal M_1(X, \Gamma).
         \end{equation}
         
         \end{enumerate}

\item\label{property-4-lem-comb} $b_{s, k_1} \perp u_\gamma b_{s, k_2} u^*_\gamma,$ $\gamma\in K,\ k_1\neq k_2,\ 1\leq k_1, k_2\leq n.$

\end{enumerate}

Then, the order zero c.p.c.~map
\begin{equation}\label{defn-0-map}
\phi:=\sum_{s=1}^S\phi_s
\end{equation} 
satisfies the requirements.

Indeed, it follows \eqref{defn-0-map} that 
$$h= \phi(1) = \sum_{s=1}^S\phi_s(1) = h_1 + \cdots + h_S,$$
$$e_i= \phi(e_{i, i}) = \sum_{s=1}^S \phi_s(e_{i, i}) = e^{(1)}_{i} + \cdots + e^{(S)}_{i},\quad i=1, 2, ..., n^2.$$
In particular,
$$b_k= b_{1, i} + \cdots + b_{S, i},\quad i=1, 2, ..., n^2,$$
and 
$$E_i = e^{-1}_i ((0, 1])= E_{1, i} \sqcup \cdots \sqcup E_{S, i}, \quad i=1, 2, ..., n^2.$$

For each $i=0, 1$,  $k=1, 2, ..., n$, and $m=1, 2, ..., M$, set $$O^k_{i, m} = \bigsqcup_{s=1}^S O_{i, m}^{k, s}.$$ By Conditions (\ref{property-2-1-lem-comb}) and (\ref{property-2-2-lem-comb}), it is clear that
$$O^k_{0, m} \subseteq O_{0, m}\cap \bigsqcup_{j=4}^{n} E_{n(k-1)+j}\quad\mathrm{and}\quad O^k_{i, m} \subseteq O_{i, m}\cap \bigsqcup_{j=1}^{n} E_{n(k-1)+j},$$ and $$O^k_{i, m} = O^k_{i, 1}\kappa_m,\quad i=0, 1, \ m=1, 2, ..., M.$$
         
By \eqref{small-left-global}, \eqref{small-left-ind}, and \eqref{large-n},  for any $\mu\in\mathcal M_1(X, \Gamma)$, 
\begin{eqnarray}\label{small-left}
\mu(X\setminus \bigsqcup_{i=1}^{n^2} E_i)  & = & \mu(X\setminus \bigsqcup_{s=1}^S \bigsqcup_{i=1}^{n^2} E_{s, i}) \\
& =  & \mu( X\setminus \bigsqcup_{s=1}^S \bigsqcup_{\gamma\in F_s} B_s\gamma )  +   \sum_{s=1}^S\mu(\bigsqcup_{\gamma\in F_s} B_s\gamma \setminus \bigsqcup_{i=1}^{n^2} E_{s, i}) \nonumber \\
& < &\lambda\frac{\delta}{32n} + \sum_{s=1}^S \lambda \frac{\delta}{16n}\mu(\bigsqcup_{\gamma\in F_s} B_s\gamma) \nonumber \\ 
& < &\lambda\frac{\delta}{32n} + \lambda\frac{\delta}{16n} <  \lambda\frac{3\delta}{32n} < \frac{1}{4}.\nonumber 
\end{eqnarray}
and then, by \eqref{dense-O-1-ind} and \eqref{small-left}
\begin{eqnarray*}
\lambda \mu(O^k_{i, 1}) & = &\lambda\sum_{s=1}^S\mu(O_{i, 1}^{k, s}) 
 > \lambda  \frac{\delta}{8n} \sum_{s=1}^S \mu(\bigsqcup_{i=1}^{n^2}E_{s, i})\\
&> & \lambda  \frac{\delta}{8n} (1-\frac{1}{4})  =   \lambda  \frac{3\delta}{32n} \\
&> & \mu(X\setminus \bigsqcup_{i=1}^{n^2} E_i).
\end{eqnarray*}
Also note that, by \eqref{large-n}, 
$$\lambda \mu(O^k_{0, 1}) > \lambda  \frac{3\delta}{32n} > \frac{3}{n^2}.$$
This verifies Property (\ref{property-1-cob}).

Property (\ref{property-3-cob}) follows from Condition (\ref{property-4-lem-comb}) straightforwardly. 
\end{proof}

Next, let us perturb further the order zero map $\phi$ obtained in Lemma \ref{division-space}. First, we make the following simple observation.

\begin{lem}\label{open-close-gap}
Let $X$ be compact metrizable space, and let $T$ be a compact set of probability Borel measures. 
\begin{enumerate}
\item If $O\subseteq X$ is an open set and $\lambda, \delta>0$ satisfy $$\lambda\mu(O) > \delta,\quad\mu\in T,$$ then there is a closed set $D\subseteq O$ such that $$\lambda\mu(D) > \delta,\quad\mu\in T.$$
\item If $O\subseteq X$ is an open set and $C\subseteq X$ is closed set satisfying $$\lambda\mu(O) > \mu(C),\quad \mu\in T,$$ for some $\lambda>0$, then there exist a closed set $D\subseteq O$ and an open set $F \supseteq C$ such that $$\lambda\mu(D) > \mu(F),\quad \mu\in T.$$
\end{enumerate}
\end{lem}

\begin{proof}
Let us prove the second statement only. The first statement can be shown with a similar argument.

For any $\mu\in T$, pick continuous functions $f_\mu, g_\mu: X \to [0, 1]$ such that $f_\mu|_{X\setminus O} = 0$, $g_\mu|_C =1$, and 
$$\lambda \tau_\mu(f_\mu) > \tau_\mu(g_\mu) + \delta_\mu $$ 
for some $\delta_\mu >0$, where $\tau_\mu(f):=\int f d\mu$. Then, pick a open neighborhood $N_\mu$ of $\mu$ such that 
$$\lambda\abs{\tau_{\mu}(f_\mu) - \tau_{\mu'}(f_\mu)} < \frac{\delta_\mu}{4}\quad\mathrm{and}\quad \abs{\tau_{\mu}(g_\mu) - \tau_{\mu'}(g_\mu)} < \frac{\delta_\mu}{4} ,\quad \tau'\in N_\mu,$$ and a straightforward calculation shows
$$\lambda\tau_{\mu'}(f_\mu) > \tau_{\mu'}(g_\mu) + \frac{\delta_\mu}{2},\quad \mu'\in N_\mu.$$
Since $T$ is compact, there is a finite open cover of $T$ consists of $N_{\mu_1}$, ..., $N_{\mu_n}$, where $\mu_1, ..., \mu_n \in T$. With $$f:=\max\{f_{\mu_1}, ..., f_{\mu_n} \},\quad g:=\min\{g_{\mu_1}, ..., g_{\mu_n} \},\quad\mathrm{and}\quad \delta := \frac{1}{2}\min\{\delta_{\mu_1}, ..., \delta_{\mu_n} \},$$ one has
$$f|_{X\setminus O} = 0,\quad g|_C =1,\quad\mathrm{and}\quad \lambda \tau_\mu(f) > \tau_\mu(g) + \delta,\quad \mu\in T.$$ 
Then, with a sufficiently small $\eps>0$, the closed set $D:=f^{-1}([\eps, 1])$ and the open set $F:= g^{-1}((1-\eps, 1])$  satisfy the lemma.
\end{proof}

\begin{lem}\label{division-space-map}
Let $\Gamma$ be an infinite group, and let $(X, \Gamma)$ be a minimal topological dynamical system with the (URP).  Let $\lambda>0$ be arbitrary, and let $O_{0, 1}, ..., O_{0, M},  O_{1, 1}, ..., O_{1, M} \subseteq X$ be mutually disjoint non-empty open sets and
$$\{\kappa_{0,1}(=e), \kappa_{0, 2}, ..., \kappa_{0, M}, \kappa_{1,1}(=e), \kappa_{1, 2}, ..., \kappa_{1, M}\}\subseteq\Gamma$$ a finite family such that $$O_{i, m}=O_{i, 1}\kappa_{i, m},\quad i=0, 1, \ m=1, ..., M.$$
Let $K\subseteq\Gamma$ be a symmetric finite set. Let $N>0$ be arbitrary.

Then there is an order zero c.p.c.~map $$\phi: \mathrm{M}_{n^2}(\Comp) \to \mathrm{C}(X) \rtimes\Gamma$$ for some $n>N$, 
such that with $$h:=\phi(1)\quad\mathrm{and}\quad e_{i} := \phi(e_{i, i}),\quad i=1, 2, ..., n^2,$$
and 
$$b_k:= e_{n(k-1)+1} + \cdots + e_{n(k-1)+n},\quad k=1, 2, ..., n,$$
so that
$$e_i \in\mathrm{C}(X)\subseteq A,$$
we have

\begin{enumerate}

\item\label{lem-map-cond-1} for each $k=1, 2, ..., n$, there are mutually orthogonal positive functions $$c_{k, 1}, ..., c_{k, M}, d_{k, 1}, ..., d_{k, M} \in\mathrm{C}(X)$$ such that 

\begin{enumerate}
          \item\label{lem-cond-1a} $c_{k, m} \in\mathrm{Her}(O_{0, m})$ and $d_{k, m} \in\mathrm{Her}(O_{1, m})$, $m=1, 2, ..., M$,
          \item\label{lem-cond-1-orth} $c_{k, m} \perp (e_{(k-1)n+1} + e_{(k-1)n+2}  + e_{(k-1)n+3})$, $m=1, 2, ..., M$,
          \item\label{lem-cond-1b} $c_{k, m}b_k = c_{k, m}$ and $d_{k, m}b_k = d_{k, m}$, $m =1, 2, ..., M,$
          \item\label{lem-cond-1c} $c_{k, m} = u^*_{\kappa_m} c_{k, 1}u_{\kappa_m}$ and $d_{k, m} = u^*_{\kappa_m} d_{k, 1}u_{\kappa_m}$, $m=1, 2, ..., M,$ and 
          \item\label{lem-cond-1d} $\lambda \mathrm{d}_\tau(c_{k, 1})> \frac{3}{n^2}$ and $\lambda \mathrm{d}_\tau(d_{k, 1})> \mathrm{d}_\tau(1-h)$, $\tau\in\mathrm{T}(A),$
\end{enumerate}

\item\label{lem-map-cond-2}

 $$b_{k_1} \perp u_\gamma b_{k_2} u^*_\gamma,\quad \gamma\in K,\  k_1\neq k_2,\ 1\leq k_1, k_2\leq n,$$ where $u_\gamma\in A = \mathrm{C}(X) \rtimes\Gamma$ is the canonical unitary.

\end{enumerate}
\end{lem}
\begin{proof}
It follows Lemma \ref{division-space} that  there exist $n\in\mathbb N$ with $n>N$ 
and an order zero c.p.c.~map $$\phi': \mathrm{M}_{n^2}(\Comp) \to A$$ 
such that, with $$h':=\phi'(1)\quad\mathrm{and}\quad e'_{i} := \phi'(e_{i, i}),\quad i=1, 2, ..., n^2,$$
and 
$$b'_k:= e'_{n(k-1)+1} + \cdots + e'_{n(k-1)+n},\quad k=1, 2, ..., n,$$
so that 
$$e'_i \in\mathrm{C}(X),$$
and with 
$$E_i = (e_i')^{-1}((0, 1]),\quad i=1, 2, ..., n^2,$$
we have 
\begin{enumerate}

\item for each $k=1, 2, ..., n$, there are
mutually disjoint open sets $O^{k}_{0, 1}, ..., O^{k}_{0, M}$ and $O^{k}_{1, 1}, ..., O^{k}_{1, M}$  such that
         \begin{enumerate}
         \item $O^k_{0, m} \subseteq O_{0, m}\cap \bigsqcup_{i=4}^{n} E_{n(k-1)+i}$ and $O^k_{1, m} \subseteq O_{1, m}\cap \bigsqcup_{i=1}^{n} E_{n(k-1)+i}$, $m=1, 2, ..., M$,
         \item $O^k_{i, m} = O^k_{i, 1}\kappa_{i, m}$, $i=0, 1$, $m=1, 2, ..., M$, and 
         \item\label{lbd-division} $$ \lambda \mu(O^k_{0, 1}) > \frac{3}{n^2}\quad\mathrm{and}\quad \lambda \mu(O^k_{1, 1}) > \mu(X\setminus \bigsqcup_{i=1}^{n^2} E_i),\quad \mu\in\mathcal M_1(X, \Gamma),$$
         \end{enumerate}

\item\label{cond-ref-lem-3} $$b'_{k_1} \perp u_\gamma b'_{k_2} u^*_\gamma,\quad \gamma\in K,\ k_1\neq k_2,\ 1\leq k_1, k_2\leq n.$$

\end{enumerate}

Since $\mathcal M_1(X, \Gamma)$ is compact, $O^k_{0, 1}$ and $O^k_{1, 1}$ are open, and $X\setminus \bigsqcup_{i=1}^{n^2} E_i$ is closed, by Condition \eqref{lbd-division} and Lemma \ref{open-close-gap}, there are closed sets $D_{i, 1}^k \subseteq O_{i, 1}^k$, $i=0, 1$, and an open set $U\supseteq X\setminus \bigsqcup_{i=1}^{n^2} E_i  $ such that 
\begin{equation}\label{gap-02}
\lambda \mu(D^k_{0, 1}) > \frac{3}{n^2}  \quad \mathrm{and}\quad  \lambda \mu(D^k_{1, 1}) > \mu(U),\quad \mu\in\mathcal M_1(X, \Gamma).
\end{equation}

For any $\eps>0$, define \begin{equation}\label{contain-V}
V_\eps:=\mathrm{int}(f_\eps(h')^{-1}(\{1\})) =  \{x\in X: h'(x) > \eps \},
\end{equation}
and consider the open sets 
\begin{equation}\label{small-open} 
W^k_{i, \eps}:=(O^k_{i, 1}\cap V_\eps) \cap (O^k_{i, 2}\cap V_\eps) \kappa_{i, 2}^{-1} \cap\cdots\cap (O^k_{i, M}\cap V_\eps)\kappa_{i, M}^{-1},
\end{equation} 
which increase to $O_{i, 1}^k$ as $\eps\to 0$, $i=0, 1$.
Since $D_{i, 1}^k$ is compact, there is $\eps>0$ sufficiently small that
$$W_{i, \eps}^k \supseteq D_{i, 1}^k,\quad i=0, 1.$$ Pick such an $\eps$, and assume also, as we may, that 
$$U \supseteq \{x\in X: h'(x) < \eps \},$$ and note that then $$U\supseteq \{x\in X: h'(x) < \eps \} = (1-f_\eps(h'))^{-1}((0, 1]). $$
Hence, in view of \eqref{gap-02},
\begin{equation}\label{large-small-03-0}
\lambda\mu(W^k_{0, \eps}) > \lambda\mu(D^k_{0, 1}) > \frac{3}{n^2},\quad \mu\in\mathcal{M}_1(X, \Gamma),
\end{equation}
and
\begin{equation}\label{large-small-03-1}
\lambda\mu(W^k_{1, \eps}) > \lambda\mu(D^k_{1, 1}) >\mu(U)  \geq \mu((1-f_\eps(h'))^{-1}((0, 1])),\quad \mu\in\mathcal{M}_1(X, \Gamma),
\end{equation}
It follows from \eqref{small-open} that
\begin{equation}\label{set-contain} 
W^k_{i, \eps} \kappa_{i, m} \subseteq  V_\eps,\quad i=0, 1,\ m=1, 2, ..., M.
\end{equation}


%
%
%
%

Set $$\varphi_{W^k_{0, \eps}} = c_{k, 1}\quad\mathrm{and}\quad u_{\kappa_m}^* c_{k, 1}u_{\kappa_m} = c_{k, m},\quad m=2, 3, ..., M,$$
and
$$ \varphi_{W^k_{1, \eps}} = d_{k, 1} \quad\mathrm{and}\quad  u_{\kappa_m}^* c_{k, 1}u_{\kappa_m} = c_{k, m},\quad m=2, 3, ..., M.$$
Note that, by \eqref{small-open}, $$c_{k, m} \in \mathrm{Her}(O^k_{0, m})\quad\mathrm{and}\quad d_{k, m} \in \mathrm{Her}(O^k_{1, m}),\quad m=1, 2, ..., M.$$
It follows from \eqref{set-contain} and \eqref{contain-V} that 
$$c_{k, m} f_\eps(b_k') = c_{k, m} f_\eps(h') = c_{k, m}\quad\mathrm{and}\quad d_{k, m} f_\eps(b_k') = d_{k, m} f_\eps(h') = d_{k, m},$$
and it follows from \eqref{large-small-03-0} and \eqref{large-small-03-1} that for any $\tau\in\mathrm{T}(A)$,
\begin{equation*}
\lambda \mathrm{d}_\tau(c_{k, 1})  = \lambda \mu_\tau(W^k_{0, \eps}) > \frac{3}{n^2}
\end{equation*}
and
\begin{eqnarray*}
\lambda \mathrm{d}_\tau(d_{k, 1}) &= & \lambda \mu_\tau(W^k_{1, \eps}) > \mu_\tau((1-f_\eps(h'))^{-1}((0, 1])) = \mathrm{d}_\tau(1- f_\eps(h')).
\end{eqnarray*}
Then $$\phi:= f_\eps(\phi'): \mathrm{M}_{n^2}(\Comp) \to A$$ is the desired order zero map. 

Indeed, noting that $$h=\phi(1) = f_\eps(h'),$$ the existence of $c_{k, m}$, $d_{k, m}$, $k=1, ..., n$, $m=1, ..., M$, and Property \ref{lem-map-cond-1} 
are verified above.

Consider any $b_{k_1}, b_{k_2}$ with $k_1\neq k_2$, $1\leq k_1, k_2 \leq n$. Note that
\begin{equation*}
b_{k_1} = f_\eps(b'_{k_1}) \in \mathrm{Her}(b_{k_1}')\quad \mathrm{and}\quad u_\gamma b_{k_2}u^*_\gamma = f_\eps(u_\gamma b'_{k_2}u^*_\gamma) \in \mathrm{Her}(u_\gamma b_{k_2}'u^*_\gamma), \quad \gamma\in K,
\end{equation*}
and therefore, it follows from Condition \ref{cond-ref-lem-3} that $$b_{k_1} \perp u_\gamma b_{k_2}u^*_\gamma,\quad  \gamma\in K.$$ This verified Property \ref{lem-map-cond-2},
as desired.
\end{proof}

We are now ready for the main results of the paper.

\begin{prop}\label{main-prop}
Let $\Gamma$ be an infinite countable discrete amenable group, and let $(X, \Gamma)$ be a minimal free topological dynamical system with the (URP) and (COS). Then the crossed product C*-algebra $\mathrm{C}(X) \rtimes \Gamma$ has Property (D). 
\end{prop}

\begin{proof}
Let $a\in\mathrm{ZD}(A)$ and let $\eps>0$ be arbitrary. It follows from Proposition \ref{alg-red} that there are unitaries $u_1, u_2\in A$, $a'\in\mathrm{C}_{\mathrm c}(\Gamma, \mathrm{C}(X))$ and a non-empty open set $E\subseteq X$ such that $$\norm{u_1au_2 - a'} < \eps \quad\mathrm{and}\quad \varphi_Ea' = a'\varphi_E = 0.$$ In the following, let us verify that $a'$ is actually a $\mathcal D_0$-element. Since $\eps$ is arbitary, this shows that $A$ has Property (D).


Note that, since $(X, \Gamma)$ has the (COS), the sub-C*-algebra $\mathrm{C}(X)$ has the $(\lambda, M)$-Cuntz comparison insider $A$ for some $\lambda\in(0, +\infty)$ and $M\in\mathbb N$. Fix $\lambda$ and $M$.

Write 
\begin{equation}\label{alg-element}
a'=\sum_{\gamma\in K} f_\gamma u_\gamma,
\end{equation} 
where $f_\gamma\in\mathrm{C}(X)$ and $K\subseteq \Gamma$ is a symmetric finite set.

Consider the open set $E$. Since $(X, \Gamma)$ is minimal, all orbits are dense, and hence there exist non-empty mutually orthogonal open sets $$O_{0, 1}, ..., O_{0, M}, O_{1, 1}, ..., O_{1, M}\subseteq E$$ and 
\begin{equation}\label{smaller-zero-sets}
\{\kappa_{0,1}(=e), \kappa_{0, 2}, ..., \kappa_{0, M}, \kappa_{1,1}(=e), \kappa_{1, 2}, ..., \kappa_{1, M}\}\subseteq\Gamma
\end{equation} 
such that $$O_{i, m}=O_{i, 1}\kappa_{i, m},\quad i=0, 1, \ m=1, ..., M.$$

Since $(X, \Gamma)$ has the (URP), it follows from Lemma \ref{division-space-map} that there is $n>3$ and an order zero c.p.c.~map $$\phi: \mathrm{M}_{n^2}(\Comp) \to A$$ 
such that if $$h:=\phi(1),\quad e_{i} := \phi(e_{i, i}),\quad i=1, 2, ..., n^2,$$
$$s_k:=e_{(k-1)p+1} + \cdots + e_{(k-1)p+3}, \quad k=1, ..., n,$$
and 
$$E_k:= e_{n(k-1)+1} + \cdots + e_{n(k-1)+n},\quad k=1, 2, ..., n,$$
then 
\begin{equation}\label{feet-in-X}
e_i \in\mathrm{C}(X)
\end{equation}
and
\begin{enumerate}

\item\label{2-lem-cond-1} for each $k=1, 2, ..., n$, there are mutually orthogonal positive functions $$c_{k, 1}, ..., c_{k, M}, d_{k, 1}, ..., d_{k, M} \in\mathrm{C}(X)$$ such that 

\begin{enumerate}
          \item\label{2-lem-cond-1a} $c_{k, m} \in\mathrm{Her}(O_{0, m})$ and $d_{k, m} \in\mathrm{Her}(O_{1, m})$, $m=1, 2, ..., M$,
          \item\label{2-lem-cond-1e} $c_{k, m} \perp s_k$, $m=1, 2, ..., M$,
          \item\label{2-lem-cond-1b} $c_{k, m}E_k = c_{k, m}$ and $d_{k, m}E_k = d_{k, m}$, $m =1, 2, ..., M,$
          \item\label{2-lem-cond-1c} $c_{k, m} = u^*_{\kappa_m} c_{k, 1}u_{\kappa_m}$ and $d_{k, m} = u^*_{\kappa_m} d_{k, 1}u_{\kappa_m}$, $m=1, 2, ..., M,$ and 
          \item\label{2-lem-cond-1d} $\lambda \mathrm{d}_\tau(c_{k, 1})> \frac{3}{n^2}$ and $\lambda \mathrm{d}_\tau(d_{k, 1})> \mathrm{d}_\tau(1-h)$, $\tau\in\mathrm{T}(A),$
\end{enumerate}

\item\label{2-lem-cond-2}  $$E_{k_1} \perp u_\gamma E_{k_2} u^*_\gamma,\quad \gamma\in K,\  k_1\neq k_2,\ 1\leq k_1, k_2\leq n,$$ where $u_\gamma\in A$ is the canonical unitary of $\gamma$.
\end{enumerate}

Let us verify that the order zero map $\phi$ satisfies Definition \ref{defn-Prop-D0} with $p=q=n$, $l=1$, and $r = 3$.  (With the given $p, q, l, r$ and the property $n > 3$, it is straightforward to verify that (\ref{defn-sizes}) of Definition \ref{defn-Prop-D0} holds.)

Note that, by Equations \eqref{alg-element}, \eqref{feet-in-X}, and Condition (\ref{2-lem-cond-2}), for any $k_1 \neq k_2$, $1\leq k_1, k_2 \leq n$,
\begin{eqnarray*}
E_{k_1} a' E_{k_2} & = & E_{k_1} (\sum_{\gamma\in K} f_\gamma u_\gamma) E_{k_2}  =  \sum_{\gamma\in K} E_{k_1}f_\gamma u_\gamma E_{k_2} \\
& = & \sum_{\gamma\in K} f_\gamma E_{k_1} u_\gamma E_{k_2} \\
& = & \sum_{\gamma\in K} f_\gamma (E_{k_1} u_\gamma E_{k_2}u^*_\gamma )u_\gamma = 0.
\end{eqnarray*}
In particular, this verifies (\ref{defn-diag-perp-3b}) of Definition \ref{defn-Prop-D0}. 

Set $$c_k= c_{k, 1} + \cdots + c_{k, M} \quad\mathrm{and}\quad d_k= d_{k, 1} + \cdots + d_{k, M},\quad k=1, ..., n.$$ Then, (\ref{defn-D-in-zero}), (\ref{defn-diag-perp-3e}), and (\ref{defn-diag-perp-3f}) of Definition \ref{defn-Prop-D0} follow directly from Conditions (\ref{2-lem-cond-1a}), (\ref{2-lem-cond-1e}), and (\ref{2-lem-cond-1b}) above.

As for (\ref{defn-diag-perp-3g}) of Definition \ref{defn-Prop-D0}, note that it follows from Condition (\ref{2-lem-cond-1d}) above that
$$\mathrm{d}_\tau(s_k) \leq \frac{3}{n^2} < \lambda\mathrm{d}_\tau(c_{k, 1})\quad\mathrm{and}\quad \mathrm{d}_\tau(1-h) < \lambda \mathrm{d}_\tau(d_{k, 1}),\quad \tau\in \mathrm{T}(A).$$
Since $(X, \Gamma)$ has $(\lambda, M)$-Cuntz comparison of open sets and $c_{k, 1}, d_{k, 1}, h, s_k \in\mathrm{C}(X)$, one has
$$s_k \precsim \underbrace{c_{k, 1}\oplus\cdots\oplus c_{k, 1}}_M \quad\mathrm{and}\quad 1-h \precsim \underbrace{d_{k, 1}\oplus\cdots\oplus d_{k, 1}}_M.$$
By Condition (\ref{2-lem-cond-1c}) above, the positive elements $c_{k, m}$, $m=1, ..., M$, are mutually orthogonal and mutually Cuntz equivalent, and the positive elements $d_{k, m}$, $m=1, ..., M$ are mutually orthogonal and mutually Cuntz equivalent. One then has $$c_k\sim \underbrace{c_{k, 1}\oplus\cdots\oplus c_{k, 1}}_M\quad\mathrm{and}\quad d_k\sim \underbrace{d_{k, 1}\oplus\cdots\oplus d_{k, 1}}_M,$$ and hence $$s_k \precsim c_k\quad\mathrm{and}\quad 1-h \precsim d_k.$$
This shows that $a'$ is a $\mathcal D_0$-element, as asserted.
\end{proof}

\begin{thm}\label{main-thm}
Let $\Gamma$ be a countable discrete amenable group, and let $(X, \Gamma)$ be a free and minimal topological dynamical system with the (URP) and (COS). Then $\mathrm{tsr}(\mathrm{C}(X) \rtimes \Gamma)= 1$.
\end{thm}

\begin{proof}
If $\abs{\Gamma} < \infty$, since $(X, \Gamma)$ is minimal,  the space $X$ must consist of finitely many points and $\mathrm{C}(X) \rtimes \Gamma\cong\M{\abs{\Gamma}}{\Comp}$. In particular, it has stable rank one. 

If $\abs{\Gamma} =\infty$, then it follows from Proposition \ref{main-prop} that $\mathrm{C}(X) \rtimes \Gamma$ has Property (D). Since $\mathrm{C}(X) \rtimes \Gamma$ is finite, it follows from Theorem \ref{thm-ab-tsr1} that $\mathrm{tsr}(\mathrm{C}(X) \rtimes \Gamma) = 1$.
\end{proof}


\begin{cor}\label{sr-Z}
Let $(X, \Int^d)$ be a free and minimal topological dynamical system. Then $\mathrm{tsr}(\mathrm{C}(X) \rtimes \Int^d) = 1.$
\end{cor}

\begin{proof}
By Theorem 4.2 and Theorem 5.5 of \cite{Niu-MD-Zd}, any free and minimal dynamical system $(X, \Int^d)$ has the (URP) and (COS). It then follows from Theorem \ref{main-thm} that $\mathrm{tsr}(\mathrm{C}(X) \rtimes \Int^d) = 1.$
\end{proof}

\begin{rem}
Without simplicity, the C*-algebra $\mathrm{C}(X) \rtimes\Gamma$ might not have stable rank one in general, even if $X$ is the Cantor set, $\Gamma=\Int$, and $(X, \Int)$ has finitely many minimal closed invariant subsets (see \cite{Poon-PAMS-89} or \cite{BNS-Cantor}).
\end{rem}

\begin{cor}\label{cor-cancellation}
Let $(X, \Int^d)$ be a free and minimal dynamical system, and set $A=\mathrm{C}(X) \rtimes\Int^d$. Then 
\begin{enumerate}
\item\label{can-proj} $A$ has cancellation of projections:  if $p, q\in A \otimes \mathcal K$ are two projections such that $p\oplus r \sim q\oplus r$ for some projections $r\in A\otimes\mathcal K$, then $p\sim q$.
\item\label{can-cuntz} $A$ has weak cancellation in the Cuntz semigroup: if $a, b, c$ are elements of the Cuntz semigroup of $A$ with $a+c \ll b+c$, then $a \ll b$. 
\item\label{K1-U1} The canonical map $\mathrm{U}(A)/\mathrm{U}_0(A) \to \Kone(A)$ is an isomorphism. That is, any unitary of $A \otimes \mathcal K$ is homotopic to a unitary of $A$, and if a unitary $u$ of $A$ is connected to the identity with a path of unitaries of $\widetilde{A\otimes K}$, then $u$ can be connected to the identity by a path of unitaries of $A$.
\end{enumerate}
\end{cor}

\begin{proof}
Statements \ref{can-proj} and \ref{K1-U1} are well known facts for C*-algebras with stable rank one (\cite{Rieffel-DimStr}). Statement \ref{can-cuntz} follows from  Theorem 4.3 of \cite{RW-Z}. 

(An earlier, equivalent, version of the weak cancellation in the Cuntz semigroup was obtained in \cite{Elliott-sr1}. For the
convenience of the reader, let us provide a proof that the cancelation result in \cite{Elliott-sr1} is equivalent to that of \cite{RW-Z}. Recall that following three statements: 
\begin{enumerate}
\item\label{wc-v1} If 
$a, b\in\mathrm{W}(A)$ are such that $a+[c] \leq b+[(c-\eps)_+]$ for some positive element $c\in\mathrm{M}_\infty(A)$, then $a\leq b$. (\cite{RW-Z})
\item\label{wc-v2} If 
$a, b, c\in\mathrm{Cu}(A)$ are such that $a+c \ll b+c$, then $a \ll b$. (\cite{Elliott-sr1})
\item\label{wc-v3} If 
$a, b, c, d\in\mathrm{Cu}(A)$ are such that $a+c \leq d \ll d \leq b+c$, then $a \ll b$. (\cite{Elliott-sr1}) 
\end{enumerate}
Let us show that these three statements are all equivalent.

It is clear that $(\ref{wc-v2}) \Rightarrow (\ref{wc-v3})$. 

As for $(\ref{wc-v1}) \Rightarrow (\ref{wc-v2})$, let $a, b, c \in (A\otimes\mathcal K)^+$ with $$[a] + [c] \ll [b] + [c].$$ Pick $\eps>0$ such that
$$[a] +[c] \ll [(b-\eps)_+] + [(c-\eps)_+],$$
and then pick an arbitrary $\eps'\in (0, \eps)$ so we have
$$[(a-\eps')_+] +[(c-\eps')_+] \ll [(b-\eps)_+] + [(c-\eps)_+].$$
Note that there are $a', b', c' \in (\mathrm{M}_\infty(A))^+$ such that $$[a'] = [(a-\eps')_+],\quad [b'] = [(b-\eps)_+]\quad\mathrm{and}\quad [c'] = [(c-\eps')_+].$$
Hence
$$[a'] +[c'] \ll [b'] + [(c'-(\eps-\eps'))_+],$$
and by (\ref{wc-v1}), $$[(a-\eps')_+] = [a'] \leq [b']=[(b-\eps)_+].$$
Since $\eps'$ is arbitrary, $$[a] \leq [(b-\eps)_+]\ll [b].$$ This proves (\ref{wc-v2}).

Let us show $(\ref{wc-v3}) \Rightarrow (\ref{wc-v1})$. By adjoining a unit, let us assume that $A$ is unital. Let $a, b, c\in (\mathrm{M}_\infty(A))^+$  such that $$[a]+[c] \leq [b]+[(c-\eps)_+]$$ for some $\eps>0$. Let $\eps' \in (0, \eps)$ be arbitrary, and hence $$[a]+[c] \leq [b]+[(c-\eps')_+].$$ 
Let $h_{\eps'}:[0, 1] \to [0, 1]$ be a continuous function which is nonzero on $[0, \eps')$ and zero everywhere else (so $h_{\eps'}(a) \perp (a-\eps')_+$).
Then
\begin{eqnarray*}
&& (a-\eps')_+\oplus (c-\eps')_+ \oplus h_{\eps'}(a) \oplus h_{\eps'}(c)\\
& \approx& ((a-\eps')_+ + h_{\eps'}(a)) \oplus ((c-\eps')_+ + h_{\eps'}(c))  \\\
&\precsim& 1_2 \\
&\approx&( a + h_{\eps'}(a) ) \oplus ( c + h_{\eps'}(c) ) \\
& \precsim & a \oplus c \oplus h_{\eps'}(a) \oplus h_{\eps'}(c)  \\
& \precsim & b \oplus  (c-\eps')_+ \oplus h_{\eps'}(a) \oplus h_{\eps'}(c) .
\end{eqnarray*}
By (\ref{wc-v3}), $$(a-\eps')_+\ll b.$$
Since $\eps'$ is arbitrary, we have $a\precsim b$. This proves (\ref{wc-v1}).)
\end{proof}

By Theorem 4.1 of \cite{CE-str1}, the Cuntz semigroup classifies homomorphisms from an inductive limit of interval algebras (AI algebra) to a C*-algebra $A$ with stable rank one. Therefore we have the following corollary.
\begin{cor}\label{cor-AI}
Let $(X, \Int^d)$ be a free and minimal dynamical system. Let $\phi_1, \phi_2: I \to A =\mathrm{C}(X)\rtimes\Int^d$ be two homomorphisms, where $I$ is an AI algebra. Then $\phi_1$ and $\phi_2$ are approximately unitarily equivalent if, and only if, $[\phi_1] = [\phi_2]$ at the level of the Cuntz semigroups.
\end{cor}

The next corollary follows from \cite{Thiel-sr1}:
\begin{cor}\label{Cu-surj}
Let $(X, \Gamma)$ be a free and minimal dynamical system with the (URP) and (COS). Then for every $f\in\mathrm{LAff}(\mathrm{T}(A))_{++}$, where $A=\mathrm{C}(X)\rtimes\Gamma$, there exists $a\in (A\otimes\mathcal K)^+$ such that $$\mathrm{d}_\tau(a) = f(\tau),\quad \tau\in\mathrm{T}(A).$$

Moreover, if $A$ has strict comparison of positive elements, then the Cuntz semigroup of $A$ is almost divisible (see \cite{Thiel-sr1}). In this case, there are canonical order-isomorphisms 
$$\mathrm{Cu}(A)\cong V(A)\sqcup\mathrm{LAff}(\mathrm{T}(A))_{++} \cong\mathrm{Cu(A\otimes\mathcal Z)}.$$

In particular, the statements above hold for $\Gamma=\Int^d$.
\end{cor}
\begin{proof}
This follows directly from Theorem 8.11 and Corollary 8.12 of \cite{Thiel-sr1}.
\end{proof}

%
%
%
%


In fact, if $\mathrm{C}(X)\rtimes\Int^d$ has strict comparison of positive elements, then it actually is Jiang-Su stable:
\begin{cor}\label{cor-Z}
Let $(X, \Gamma)$ be a free and minimal dynamical system with the (URP), and denote by  $A=\mathrm{C}(X)\rtimes\Gamma$. Then $A \cong A\otimes\mathcal Z$ if, and only if, $A$ has strict comparison of positive elements. (In other words, $A$  satisfies the Toms-Winter conjecture). In particular, the statement holds for $\Gamma=\Int^d$.
\end{cor}

\begin{proof}
One only needs to show the ``if" part. Let us show that $A$ is tracially $0$-divisible in the sense of Definition 3.5(ii) of \cite{Winter-Z-stable-02}. That is, for any positive contraction $a\in\mathrm{M}_\infty(A)$, any $k\in \mathbb N$, and any $\eps>0$, there is an order zero map $$\phi: \mathrm{M}_k(\Comp) \to \mathrm{Her}(a),$$ where $\mathrm{Her}(a)$ is the hereditary sub-C*-algebra generated by $a$, such that
$$\tau(\phi(1_k)) > \tau(a) - \eps,\quad  \tau\in \mathrm{T}(A).$$
It then follows from Proposition 4.7 of \cite{Niu-MD-Z-absorbing} and the strict comparison assumption that $A$ is tracially $\mathcal Z$-stable. Since $A$ is nuclear, by \cite{Matui-Sato-CP} and \cite{HO-Z}, it follows that $A\cong A\otimes\mathcal Z$.

The proof of the tracial $0$-divisibility is similar to that of Corollary 3.2 of \cite{Niu-MD-Z-absorbing}. For the given positive contraction $a$, let us consider the lower semicontinuous affine function
$$\mathrm{T}(A) \ni \tau \mapsto \frac{1}{k}\mathrm{d}_\tau(a) \in (0, \infty).$$ Since $A$ is assumed to have strict comparison of positive elements, $(X, \Gamma)$ particularly has the (COS). Then, by Corollary \ref{Cu-surj}, there is a positive element $x\in A\otimes\mathcal K$ such that
$$ \mathrm{d}_\tau(x) = \frac{1}{k}\mathrm{d}_\tau(a),\quad \tau \in\mathrm{T}(A).$$

For each pair of positive numbers $\delta_1< \delta_2$, define the continuous function
$$f_{\delta_1, \delta_2}(t) = \left\{ \begin{array}{ll} 0, & t \leq \delta_1, \\ \frac{t - \delta_1}{\delta_2 - \delta_1}, & \delta_1 < t < \delta_2, \\1, & t\geq \delta_2. \end{array} \right.$$
Also consider the continuous function $$f_\eps(t) := \max\{t-\eps, 0\},\quad t\in\mathbb R.$$ 
Then, since $A$ is simple,  with a sufficiently small $\delta>0$ (see, Remark 2.7 of \cite{Winter-Z-stable-02}), one has 
$$
\tau(f_{2\delta, 3\delta}(x)) > \frac{1}{k}\tau(f_{\eps}(a)) > \frac{1}{k}(\tau(a)-\eps),\quad \tau\in\mathrm{T}(A),
$$
and use the simplicity again, there is $\delta'>0$ such that
$$\tau(f_{\delta/2, \delta}(x)) < \mathrm{d}_\tau(x)-\delta'=  \frac{1}{k}(\mathrm{d}_\tau(a))-\delta',\quad \tau\in\mathrm{T}(A).$$
Thus, after a perturbation of $x$, there is a positive element $x'\in \mathrm{M}_\infty(A)$ such that
\begin{equation}\label{dini}
\tau(f_{2\delta, 3\delta}(x')) > \frac{1}{k}(\tau(a)-\eps),\quad \tau\in\mathrm{T}(A),
\end{equation}
and
$$\tau(f_{\delta/2, \delta}(x')) <   \frac{1}{k}\mathrm{d}_\tau(a)-\delta',\quad  \tau\in\mathrm{T}(A).$$

Note that $$k\mathrm{d}_\tau(f_{\delta, 2\delta}(x')) < k \tau(f_{\delta/2, \delta}(x')) <  \mathrm{d}_\tau(a),\quad  \tau\in\mathrm{T}(A).$$ Since $A$ has strict comparison, one has $k[f_{\delta/2, \delta}(x')] < [a]$ in $\mathrm{W}(A)$ ($x'$ and $a$ are in $\mathrm{M}_\infty(A)$). By Proposition 2.12 of \cite{Winter-Z-stable-02}, there is an order zero map $\phi:\mathrm{M}_k(\Comp) \to \mathrm{Her}(a)$ such that
$$\phi(e_{1, 1}) \approx f_{2\delta, 3\delta}(x'),$$
where $a\approx b$ denotes the relation $a=vv^*$, $b=v^*v$ for some $v$. In particular, by \eqref{dini},
$$\tau(\phi(1_k)) = k\tau(\phi(e_{1,1})) = k\tau(f_{2\delta, 3\delta}(x')) > \tau(a) -\eps,\quad \tau\in\mathrm{T}(A),$$
as desired.
%
\end{proof}


Since the real rank of a C*-algebra $A$ is at most $2\cdot \mathrm{tsr}(A) - 1$, one has the following estimate:
\begin{cor}\label{cor-RR}
Let $(X, \Int^d)$ be a free and minimal dynamical system. The real rank of $\mathrm{C}(X) \rtimes\Int^d$ is either $0$ or $1$.
\end{cor}

\begin{rem}
Consider a simple unital AH algebra $A$ with diagonal maps. It is known that if $A$ has real rank zero (or just projections separate traces), then $A$ is classifiable (\cite{Niu-MD}). Does the same statement hold for the crossed-product C*-algebras $\mathrm{C}(X) \rtimes\Int$ (or $\mathrm{C}(X) \rtimes\Gamma$ in general)? That is, if $\mathrm{C}(X) \rtimes\Int$ (or $\mathrm{C}(X) \rtimes\Gamma$,  in general) has real rank zero,  does $\mathrm{C}(X) \rtimes\Int$ (or $\mathrm{C}(X) \rtimes\Gamma$, in general) absorb the Jiang-Su algebra $\mathcal Z$ tensorially? What if one only assumes that projections separate traces instead of real rank zero?  
\end{rem}

Let $\Gamma$ be a countable discrete group with sub-exponential growth,  and let $(X, \Gamma)$ be a free and minimal dynamical system. Assume that $(X, \Gamma)$ is an extension of a minimal $\Gamma$-action on a Cantor set. Then it was shown in \cite{Suzuki-sr1} that the C*-algebra $\mathrm{C}(X) \rtimes \Gamma$ has stable rank one. Note that, by Corollary 3.8 and Corollary 8.11 of \cite{Niu-MD-Z}, the dynamical system $(X, \Gamma)$ has the (URP) and (COS), and therefore this result also can be deduced from Theorem \ref{main-thm}.
\begin{cor}[cf.~Main Theorem of \cite{Suzuki-sr1}]
Let $\Gamma$ be a countable discrete group with sub-exponential growth,  let $(X, \Gamma)$ be a free and minimal dynamical system. Assume that $(X, \Gamma)$ is an extension of a $\Gamma$-action on the Cantor set.  Then $\mathrm{tsr}(\mathrm{C}(X) \rtimes \Gamma) = 1.$
\end{cor}

\section{Two remarks on Property (D)}\label{D-remark}

In this final section, let us remark that simple $\mathcal Z$-stable C*-algebras and simple AH-algebras with diagonal maps all have Property (D). These C*-algebras (if finite for the case of $\mathcal Z$-stable C*-algebras) are known to have stable rank one (see \cite{Ror-Z-stable} and \cite{EHT-sr1}).  

\subsection{$\mathcal Z$-stable C*-algebras}

Let $A$ be a unital simple exact C*-algebra such that $A\cong A\otimes\mathcal Z$, where $\mathcal Z$ is the Jiang-Su algebra. Note that, by \cite{Ror-Z-stable}, $A$ has strict comparison of positive elements (we include infinite C*-algebras, which have empty tracial simplices). 

Let $a\in\mathrm{ZD}(A)$ with $\norm{a}=1$ and let $\eps>0$ be arbitrary. Pick $d_1, d_2\in A^+$ such that $\norm{d_1} = \norm{d_2} = 1$ and $$d_1a = ad_2 = 0.$$ By regarding $A$ as $A\otimes\mathcal Z\otimes\mathcal Z\otimes \mathcal Z$, one obtains $\tilde{a}, \tilde{d_1}, \tilde{d_2} \in A\otimes \mathcal Z\otimes 1 \otimes 1$ with norm one, where $\tilde{d}_1, \tilde{d_2}$ are positive, and a unitary $u\in A\otimes\mathcal Z\otimes\mathcal Z\otimes 1$ such that 
$$\norm{uau^*-\tilde{a}}<\frac{\eps}{12},\quad \norm{ud_1u^*-\tilde{d_1}} < \frac{\eps}{12},\quad \norm{ud_2u^* - \tilde{d_2}} < \frac{\eps}{12},$$
and then
$$\norm{\tilde{d_1}\tilde{a}} < \frac{\eps}{6}\quad\mathrm{and}\quad \norm{\tilde{a}\tilde{d_2}} < \frac{\eps}{6}.$$ 
With a small perturbation of $\tilde{d_1}$ and $\tilde{d_2}$, one may assume that there are positive elements $d_1', d_2'\in A\otimes \mathcal Z \otimes 1\otimes 1$ with $\norm{d_1'} =  \norm{d_2'} =1$ such that $$d'_1 \tilde{d_1} = d'_1\quad\mathrm{and}\quad d'_2 \tilde{d_2} = d'_2.$$
Note that 
$$ (1-\tilde{d_1})\tilde{a}(1-\tilde{d_2}) \approx_{\frac{\eps}{2}} \tilde{a}\quad\mathrm{and}\quad d_1'(1-\tilde{d_1})\tilde{a}(1-\tilde{d_2}) = (1-\tilde{d_1})\tilde{a}(1-\tilde{d_2}) d_2' =0 .$$ 
Pick two orthogonal non-zero positive elements $s_1, s_2\in 1\otimes 1 \otimes \mathcal Z \otimes 1$, and consider the positive elements  $\tilde{d}_1 s_1$ and $\tilde{d}_2 s_2$. 
Since $s_1, s_2$ commute with $\tilde{d_1}, \tilde{d_2}$, one has that $$\tilde{d}_1 s_1\perp \tilde{d}_2 s_2 \quad \mathrm{and}\quad (\tilde{d}_1s_1)((1-d'_1)\tilde{a}(1-\tilde{d_2})) = ((1-\tilde{d_1})\tilde{a}(1-\tilde{d_2})) (d_2's_2) =0.$$ Since $A\otimes\mathcal Z\otimes\mathcal Z \otimes 1$ is simple, there is $v\in A\otimes\mathcal Z\otimes\mathcal Z\otimes 1$ with $\norm{v} =1$ such that $$vv^*\in \mathrm{Her}(d_1s_1)\quad \mathrm{and}\quad v^*v \in \mathrm{Her}(d_2s_2),$$ and, moreover, using polar decomposition and making a further perturbation, one may assume that there is a positive element $b$ such that $\norm{b} = 1$ and $(vv^*) b = b$ (and hence $b\in \mathrm{Her}(d_1s_1)$). It then follows from Lemma \ref{switch} that there is a unitary $w\in A\otimes\mathcal Z\otimes\mathcal Z \otimes 1$ such that
$$wbw^* \in \mathrm{Her}(d_2s_2).$$ Thus, with $$a' := (1-\tilde{d_1})\tilde{a}(1-\tilde{d_2})w,$$ one has
$$\norm{uau^*w -a'}< \eps\quad\mathrm{and}\quad b a' = 0 = a'b.$$

Let us show that $a'$ is a $\mathcal D_0$-element, and thus that $A$ has Property (D).

Since $A$ is not of type I, there are positive elements $c, d\in \overline{b(A\otimes\mathcal Z\otimes \mathcal Z\otimes 1)b}$ such that $c \perp d$ and $\norm{c} = \norm{d} = 1$. Since $A$ is simple, there is $\delta>0$ such that $$\tau(c), \tau(d) > \delta,\quad \tau\in\mathrm{T}(A).$$

Now, consider the embedding $\phi': \mathcal Z \to 1 \otimes 1 \otimes 1 \otimes \mathcal Z$, 
and note that
\begin{equation}\label{center-Z}
[a', \phi'(a)] = 0 \quad \mathrm{and} \quad[b, \phi'(a)] = 0,\quad a\in\mathcal Z.
\end{equation}
Pick $n\in\mathbb N$ sufficiently large that
$$(n-3)\delta > 6,$$
and pick a standard embedding 
$$\iota: \mathrm{M}_{n^2}(\mathrm{C}_0((0, 1]))\to \mathcal Z_{n^2, n^2+1}\to \mathcal Z.$$

 Denote by $\phi''$ the order zero map induced by the homomorphism $\phi'\circ\iota$, and choose $\eps'>0$ sufficiently small that 
 $$\tau(f_{\eps'}(\phi''(1))) > 1- \delta/2n,\quad \tau\in\mathrm{T}(A).$$
For each $k=1, 2, ..., n$, define
$$c_k = c \cdot f_{\eps'}(\phi'')(e_{(k-1)n+4}+\cdots +e_{(k-1)n+n} ),$$
and
$$d_k = d \cdot f_{\eps'}(\phi'')(e_{(k-1)n+1}+\cdots +e_{(k-1)n+n} ).$$
A straightforward calculation (using  \eqref{center-Z}) shows that $c_k$, $d_k\in \overline{b(A\otimes\mathcal Z\otimes \mathcal Z\otimes \mathcal Z)b}$, $c_k\perp d_k$ and
$$c_k E_k = c_k \quad\mathrm{and}\quad d_k E_k = d_k,$$
where $E_k:= f_{\frac{\eps'}{2}}(\phi'')(e_{(k-1)n+1}+\cdots +e_{(k-1)n+n} )$.

Note that, for any $\tau \in \mathrm{T}(A)$, $$\mathrm{d}_\tau(c_k)> \delta \cdot  \frac{n-3}{n^2} \cdot \frac{2n-\delta}{2n} > \frac{3}{n^2} \cdot \frac{4n-2}{2n} > \frac{3}{n^2}  $$    
and  
$$\mathrm{d}_\tau(d_k) > \delta \cdot  \frac{1}{n} \cdot \frac{2n-\delta}{2n} > \delta \cdot  \frac{1}{n} \cdot \frac{2n-1}{2n}  > \frac{\delta}{2n} > \mathrm{d}_\tau(1- f_{\frac{\eps}{2}}(\phi''(1))).$$ Since $A$ has strict comparison of positive elements, one has that
$$s_k \precsim c_k\quad\mathrm{and}\quad 1- f_{\frac{\eps}{2}}(\phi''(1)) \precsim d_k.$$
This shows that $a'$ is a $\mathcal D_0$-element (with $\phi=f_{\frac{\eps'}{2}}(\phi'')$, $p=q=n$, $r=3$, and $l=1$ in Definition \ref{defn-Prop-D0}), as asserted.

\subsection{AH algebras with diagonal maps}

Recall that an AH algebra with diagonal maps is the limit of a unital inductive sequence $(A_n, \psi_n)$, where $$A_n = \bigoplus _{i=1}^{h_n}\mathrm{M}_{k_{n, i}}(\mathrm{C}(X_{n, i}))$$ for some compact metrizable space $X_{n, i}$, and if 
$$D_n :=  \bigoplus_i \{\mathrm{diag}\{f_1, f_2, ..., f_{k_{n, i}}\}: f_k\in\mathrm{C}(X_{n, i})\} \subseteq \bigoplus_{i}\mathrm{M}_{k_{n, i}}(\mathrm{C}(X_{n, i}))=A_n,$$
then
$$\psi_n(D_n)\subseteq D_{n+1}.$$


Let $A$ be a simple AH algebra with diagonal maps. It then follows from Theorem 3.4 of \cite{EHT-sr1} that $A$ has Property (D); we leave the details to the reader. Alternatively, let us propose the following approach which is similar to our approach to the crossed product C*-algebra $\mathrm{C}(X) \rtimes\Gamma$:
%
%
Consider the limit diagonal algebra $$D:=\varinjlim D_n \subseteq  \varinjlim A_n = A.$$ Then the commutative sub-C*-algebra $D$ actually behaves like the sub-C*-algebra $\mathrm{C}(X)$ of $\mathrm{C}(X)\rtimes\Gamma$.
 
Let $a\in A$ satisfy $\norm{a} = 1$ and $a\in\mathrm{ZD}(A)$, i.e.,  $d_1a=ad_2=0$ for some non-zero positive elements $d_1, d_2$, and let $\eps>0$ be arbitrary. With a telescoping of the inductive sequence if necessary, there are $\tilde{a}, \tilde{d}_1, \tilde{d}_2 \in A_1$ with norm one such that 
$$\norm{a-\tilde{a}} < \frac{\eps}{2}\quad\mathrm{and}\quad\norm{\tilde{d}_1 \tilde{a}},  \norm{\tilde{a} \tilde{d}_1} < \frac{\eps}{12},$$ 
where $\tilde{d}_1, \tilde{d}_2$ are positive. 
Since $\tilde{d}_1, \tilde{d}_2$ has norm one, there is $x_0\in X_{1, i}$ for some $i$ such that $$\norm{\tilde{d}_1(x_0)}=1\quad\mathrm{and}\quad \norm{\tilde{d}_2(x_0)} = 1.$$ 
Since $\tilde{d}_1, \tilde{d}_2$ are positive, by conjugating some constant unitary matrices, one may assume that $\tilde{d}_1$ and $\tilde{d}_2$ are diagonal matrices at $x_0$. Hence, by cutting the diagonal entry which has value $1$ at $x_0$, one can find a positive element $h\in D_1$ which is constant equal to $1$ on a small neighbourhood of $x_0$ such that 
$$h\tilde{d}_1 \approx_\frac{\eps}{12} h\quad\mathrm{and}\quad  h\tilde{d}_2 \approx_\frac{\eps}{12} h.$$ Then a straightforward calculation shows that 
$$\norm{h\tilde{a}} < \frac{\eps}{6}\quad\mathrm{and}\quad \norm{\tilde{a}h} < \frac{\eps}{6}.$$ 
Since $h$ is constant equal to $1$ on a neighbourhood of $x_0$, there is a positive element $b\in D_n$ with norm $1$ such that  $bh=b$. Then
$$a':=(1-h)\tilde{a}(1-h)$$ and $b$ satisfy
$$\norm{a - a'}< \eps\quad\mathrm{and}\quad ba' = a'b = 0.$$

Now, let us show that $a'$ is a $\mathcal D_0$-element, and thus that $A$ has Property (D). (Recall that, at this stage, $a$ has been replaced by $u_1au_2$ as allowed in Definition \ref{defn-Prop-D}.)

Choose positive orthogonal functions $c, d\in D_1$ such that $c, d\in bD_1b$. Set $$\delta = \min\{\tau(c), \tau(d); \tau\in \mathrm{T}(A)\}>0.$$ 
By another telescoping if necessary, one may assume that
\begin{equation}\label{pre-dense}
\frac{3}{k_{1, i}} < \frac{\delta}{2} < \frac{\min\{\mathrm{rank}(c(x)), \mathrm{rank}(d(x)) \}}{k_{1, i}},\quad x\in X_{1, i}.
\end{equation}
Choose $l\in\mathbb N$ such that $$\frac{12}{l-3} < \frac{\delta}{2}.$$

Set $K:=\max\{k_{1, 1}, ..., k_{1, h_1}\}+1$. Consider $A_2$, and to simplify notation, rewrite $A_{2} = \bigoplus_{s=1}^S \mathrm{M}_{k_s}(\mathrm{C}(X_s))$. With a telescoping of the inductive sequence if necessary, one has that, inside each direct summand of $A_{2}$,
\begin{enumerate}

\item\label{cond-diag-1}  the element $a'$ is a matrix of continuous functions with 
          \begin{equation}\label{band-diag} a'_{i, j} = 0,\quad\mathrm{if}\ \abs{i - j}\geq K,\end{equation} 

\item\label{cond-diag-2}  with $$c=\bigoplus_{s=1}^S \mathrm{diag}\{c_1^{(s)}, ..., c^{(s)}_{k_s}\}\quad \mathrm{and}\quad d=\bigoplus_{s=1}^S \mathrm{diag}\{d_1^{(s)}, ..., d^{(s)}_{k_s}\},$$ where $c^{(s)}_{i}, d^{(s)}_{i}  \in\mathrm{C}(X_s)$, by \eqref{pre-dense}, for any $ L =1, ...,  k_s$, 
\begin{equation}\label{dense-c} 
\frac{\delta}{2}(1 -\frac{2K}{L})  < \frac{\abs{\{i_0 \leq i \leq i+L-1: c_i^{(s)}(x) \neq 0 \}}}{L},\quad 1\leq i_0 <k_s-L,\  x\in X_s,
\end{equation} 
and
\begin{equation}\label{dense-d} 
\frac{\delta}{2}(1 -\frac{2K}{L})  < \frac{\abs{\{i_0 \leq i \leq i+L-1: d_i^{(s)}(x) \neq 0 \}}}{L},\quad 1\leq i_0 <k_s-L,\  x\in X_s,
\end{equation}

\item\label{cond-diag-3} with $k_s = m_s l^2 + r_s$, $0\leq r_s < l^2$, one has 
\begin{equation}\label{large-m-1}
m_sl > 2K, \quad \frac{K}{m_s} < \frac{1}{2}, \quad \mathrm{and}\quad \frac{4r_s}{m_sl- 2K} < \frac{\delta}{2}.
\end{equation}

\end{enumerate}

Then, for each $s=1, ..., S$, consider the elements $\mathrm{M}_{k_s}(\Comp) \subseteq \mathrm{M}_{k_s}(\mathrm{C}(X_s))$, and consider the projection 
$$p^{(s)}_i:=\mathrm{diag}\{\overbrace{\underbrace{0_{m_s}, ..., 0_{m_s}, 1_{m_s} }_{im_s}, 0_{m_s}, ..., 0_{m_s}}^{l^2m_s}, 0_{r_s}\},\quad i=1, ..., l^2.$$
Note that $p^{(s)}_1, p^{(s)}_2, ..., p^{(s)}_{l^2} \subseteq \mathrm{M}_{k_s}(\Comp)$ have the same rank and are mutually orthogonal. Therefore, there is a homomorphism
$$\phi_s: \mathrm{M}_{l^2}(\Comp) \ni e_{i, i} \mapsto p^{(s)}_i \in \mathrm{M}_{k_s}(\Comp)  \subseteq  A_{2}.$$ 
Consider the direct sum map 
$$ \phi: = \bigoplus _s \phi_s: \mathrm{M}_{l^2}(\Comp) \to \bigoplus_s \mathrm{M}_{k_s}(\Comp)  \subseteq  A_{2},$$
and set $ \phi(e_{i, i}) = e_i$, $i=1, 2, ..., l^2$, $e_{l(k-1)+1} + \cdots + e_{l(k-1)+4} = s_k$, $e_{l(k-1)+1} + \cdots + e_{lk}= E_k$, $k=1, ..., l$, and $\phi(1) = h$.

Then, since $m_sl >2K$, by \eqref{band-diag}, $$E_{k_1} a' E_{k_2} = 0,\quad k_2 - k_1 \geq 2.$$

For each $k=1, 2, ..., l$, consider the diagonal elements 
$$ c_{k}:=\bigoplus_{s}\mathrm{diag}\{\overbrace{\underbrace{0_{m_s},..., 0_{m_s}}_{l(k-1)m_s}, \underbrace{\underbrace{0_{m_s},..., 0_{m_s}}_{4m_s}, c^{(s)}_{(l(k-1)+4)m_s+1}, ..., c^{(s)}_{lkm_s}}_{lm_s}, 0_{m_s}, ..., 0_{m_s}}^{l^2m_s}, 0_{r_s} \},$$ 
and
$$d_{k}:=\bigoplus_s \mathrm{diag}\{\overbrace{\underbrace{0_{m_s},..., 0_{m_s}}_{l(k-1)m_s}, \underbrace{d^{(s)}_{l(k-1)m_s+1}, ..., d^{(s)}_{lkm_s}}_{lm_s}, 0_{m_s}, ..., 0_{m_s}}^{l^2m_s}, 0_{r_s} \}.$$ 
Then it is clear that $c_k \perp s_k$, $c_k \perp d_k$, $c_kE_k = c_k$ and $d_kE_k = d_k$.

Note that, by \eqref{dense-c}, \eqref{dense-d}, and \eqref{large-m-1},
$$\frac{1}{4}\mathrm{rank}(c_k(x)) > \frac{1}{4} \cdot \frac{\delta}{2}((l-4)m_s- 2K) > 3 m_s = \mathrm{rank}(s_k(x))$$
and
$$\frac{1}{4}\mathrm{rank}(d_k(x)) >\frac{1}{4} \cdot \frac{\delta}{2}(lm_s- 2K) >  r_s = \mathrm{rank}((1-h)(x)).$$
Since $s_k, c_k, d_k$ and $1-h$ are diagonal elements, by Theorem 7.8 of \cite{Niu-MD-Z}, one has 
$$s_k \precsim c_k\quad\mathrm{and}\quad 1-h \precsim d_k.$$ Therefore, $a'$ is a $\mathcal D_0$-element (with $p=q=l$, $l=2$, and $r=4$ in Definition \ref{defn-Prop-D0}).

\bibliographystyle{plainurl}
\bibliography{operator_algebras}

\end{document}